\documentclass[a4paper, 11pt, reqno]{amsart}
\usepackage{amsthm,amssymb,amsmath,enumerate,graphicx,psfrag,enumitem,mathrsfs,bbm}
\usepackage[margin=2.2cm]{geometry}
\usepackage{hyperref}
\hypersetup{colorlinks=true,
   citecolor=blue,
   filecolor=blue,
   linkcolor=blue,
   urlcolor=blue
}

\usepackage[textsize=footnotesize]{todonotes}
\usepackage{algorithm}
\usepackage{tikz}

\allowdisplaybreaks
\usetikzlibrary{backgrounds}

\newtheorem{definition}{Definition}[section]
\newtheorem{claim}{Claim}
\newtheorem{remark}{Remark}
\newtheorem{proposition}[definition]{Proposition}
\newtheorem{theorem}[definition]{Theorem}

\newtheorem{lemma}[definition]{Lemma}

\numberwithin{equation}{section}

\newcommand{\comment}[1]{}
\newcommand{\N}{\mathbb N}
\newcommand{\R}{\mathbb R}

\newcommand{\es}{\emptyset}

\newcommand{\cB}{\mathcal{B}}
\newcommand{\cC}{\mathcal{C}}
\newcommand{\cD}{\mathcal{D}}
\newcommand{\cE}{\mathcal{E}}
\newcommand{\cF}{\mathcal{F}}

\newcommand{\cW}{\mathcal{W}}
\newcommand{\cX}{\mathcal{X}}

\newcommand{\bG}{\mathbf{G}}

\newcommand{\bX}{\mathbf{X}}

\newcommand{\bF}{\mathbf{F}}

\newcommand{\bq}{\mathbf{q}}

\newcommand{\ih}{\mathbf{i}}
\newcommand{\ihb}{{\underline{\mathbf{i}}}}

\newcommand{\bPr}{\mathbb{P}}

\newcommand{\cen}{{\rm Cen}}
\newcommand{\last}{{L^{last}}}

\newcommand{\sE}{\mathscr{E}}

\newcommand{\Pro}{\mathbb{P}}
\newcommand{\Exp}{\mathbb{E}}

\newcommand{\den}{{\rm den}}

\newcommand{\li}{{\mathcal{L}}}
\newcommand{\he}{{\mathcal{H}}}
\newcommand{\lea}{\mathscr{L}}
\newcommand{\LA}{\Lambda}

\newcommand{\fsum}[3]{#1^{\Sigma}\left(#2,#3\right)}

\newcommand{\norm}[1]{\|#1\|}

\renewcommand{\epsilon}{\varepsilon}

\newcommand{\sm}{\setminus}
\newcommand{\sub}{\subseteq}

\newcommand{\COMMENT}[1]{}

\newcounter{step}

\title{Spanning trees in randomly perturbed graphs}
\author{Felix Joos}

\address{School of Mathematics, University of Birmingham, 
Edgbaston, Birmingham, B15 2TT, United Kingdom}
\email{f.joos@bham.ac.uk, j.kim.3@bham.ac.uk}

\author{Jaehoon Kim}

\thanks{The research leading to these results was partially supported by the EPSRC, grant no.~EP/M009408/1, and the DFG, grant no.~JO 1457/1-1 (F.~Joos).
The research was  also partially supported by the European Research Council under the European Union's Seventh Framework Programme (FP/2007--2013) / ERC Grant 306349 (J.~Kim). }

\date{\today}

\begin{document}

\begin{abstract}
A classical result of 
Koml\'os, S\'ark\"ozy and  Szemer\'edi states that every $n$-vertex graph with minimum degree at least $(1/2+ o(1))n$ contains every $n$-vertex tree with maximum degree $O(n/\log{n})$ as a subgraph, and the bounds on the degree conditions are sharp. 
On the other hand, Krivelevich, Kwan and Sudakov recently proved that for every $n$-vertex graph $G_\alpha$ with minimum degree at least $\alpha n$ for any fixed $\alpha >0$ and every $n$-vertex tree $T$ with bounded maximum degree, one can still find a copy of $T$ in $G_\alpha$ with high probability after adding $O(n)$  randomly-chosen edges to $G_\alpha$.
We extend their results to trees with unbounded maximum degree. 
More precisely, for a given $n^{o(1)}\leq \Delta\leq cn/\log n$ and $\alpha>0$, 
we determine the precise number (up to a constant factor) of random edges  that we need to add to an arbitrary $n$-vertex graph $G_\alpha$ with minimum degree $\alpha n$ in order to guarantee a copy of any fixed $n$-vertex tree $T$ with maximum degree at most~$\Delta$ with high probability.
\end{abstract}
\maketitle

\section{Introduction}

One central theme of extremal combinatorics deals with the question which conditions on a `dense' graph $G$ imply the existence of a `sparse'/`small' graph $H$ as a subgraph of $G$.
The earliest results of this type include Mantel's theorem and its generalisation by Tur\'an,
which states that $G$ contains a complete graph on $r$ vertices whenever its number of edges is at least $(1-1/(r-1)+o(1))\binom{n}{2}$.
Another cornerstone in the area is due to Dirac~\cite{Dir52} who proved that whenever the minimum degree $\delta(G)$ of $G$ is at least $n/2$,
the graph $G$ contains a spanning cycle, known as a Hamilton cycle, and thus also a spanning path. 

This was improved 40 years later,
when Koml\'os, S\'ark\"ozy and  Szemer\'edi~\cite{KSS97} proved in a seminal paper that the condition of $\delta(G)\geq (1/2 + o(1))n$ ensures the containment of every bounded-degree $n$-vertex tree as a subgraph, 
and in \cite{KSS01}, they enormously extended this result to any $n$-vertex tree of maximum degree $O(n/\log n)$;
for refinements of the statement see~\cite{CLNS10}.
In 2009, 
B\"ottcher, Schacht and Taraz \cite{BST09} found a minimum degree condition which implies
the containment of a subgraph $H$ from a more general graph class (than trees) of bounded maximum degree graphs (known as the bandwidth theorem).

We emphasise that the mentioned minimum degree conditions cannot be further improved;
in particular, 
the disjoint union of two complete graphs $K_{n/2}$ contains neither a Hamilton cycle nor any spanning tree. 
Similarly, the almost balanced complete bipartite graph $K_{n/2-1, n/2+1}$ also does not contain a Hamilton cycle nor most $n$-vertex trees.
However, these extremal examples admit very specific structures. 
Indeed, `typical' graphs,
as binomial random graphs, which we denote by $\bG(n,p)$, with a fixed edge density $p>0$ contain many Hamilton cycles
as well as any fixed spanning tree of maximum degree  $O(n/\log n)$ with high probability.
More precisely,
it is known that the choice $p\geq \log n/n$ ensures the existence of a Hamilton cycle
and recently
Montgomery announced a proof showing that for any $n$-vertex tree $T$ with bounded maximum degree,
$\bG(n,p)$ contains $T$ with high probability whenever $p\geq c'\log n/n$ where $c'$ depends on $\Delta(T)$.
Krivelevich~\cite{Kri10} showed that the condition $p= \Omega(\Delta(T)\log n/n)$ ensures the containment of any fixed $T$ whenever $\Delta(T)\geq n^{\epsilon}$ for some small $\epsilon>0$.
We remark that there is an $n$-vertex tree $T$ of maximum degree $O(n/\log n)$ 
such that $\bG(n,0.9)$ does not contain $T$ as a subgraph with high probability;
in particular, there are graphs of minimum degree at least $0.8n$ that do not contain $T$ as a subgraph.

As an interpolation of both aforementioned models,
Bohman, Frieze and Martin \cite{BFM03} considered the following question, which has initiated a lot of research since then~\cite{BTW17,BHKM18,BFKM04,BHKMPP18,BMPP18,HZ18,  KKS16, KST04,MM18}.
Given any fixed $\alpha >0$ and
any $n$-vertex graph $G_\alpha$ with $\delta(G)\geq \alpha n$,
which lower bound on $p$ guarantees with high probability a Hamilton cycle in $G_\alpha\cup \bG(n,p)$?
This type of question combines extremal and probabilistic aspects in one graph model,
which is nowadays known as dense randomly perturbed graph model.
In fact, Bohman et al.~proved that $p=\Theta_\alpha(1/n)$ is the right answer to their question.
Thus, in this case $p$ can be taken smaller by a $\log n$-factor in comparison to the purely random graph model.
Interestingly,
it turns out that for several settings the omission of a $\log n$-factor is the correct answer.

Exactly this phenomenon 
also appears in the work of Krivelevich, Kwan and Sudakov in~\cite{KKS17}.
They proved the natural generalisation of Bohman et al.
by showing that for a given $n$-vertex tree $T$ with bounded maximum degree, say bounded by $\Delta$,
if $p= \Theta_{\alpha, \Delta}(1/n)$,
then $G_\alpha \cup \bG(n,p)$ contains $T$ with high probability.
Hence they translate the setting of the first paper of Koml\'os, S\'ark\"ozy, and Szemer\'edi into this randomly perturbed graph model. 
Here we consider trees of arbitrary maximum degree.
Interestingly, the optimal value for $p$ exhibits a certain threshold behaviour.

\begin{theorem}\label{thm: main}
For each $k\in \mathbb{N}$ and $\alpha >0$, there exists $M>0$ such that the following holds.
Suppose that $G$ is an $n$-vertex graph with $\delta(G)\geq \alpha n$ and $T$ is an $n$-vertex tree with $n^{1/(k+1)} \leq \Delta(T) < \min\{ n^{1/k}, \frac{n}{M\log{n}}\}$, 
and $R\in \bG(n, M p)$ is a random graph on $V(G)$ with $p= \max\big\{ n^{- \frac{k}{k+1}}, \Delta(T)^{k+1}n^{-2}\big\}$, 
then $T\subseteq G\cup R$ with probability $1-o(1)$.
\end{theorem}

Let us add here a few remarks.
\begin{itemize}
\item The bound on $p$ is sharp up to a constant factor for any $\alpha \leq 1/2$. (If $\alpha>1/2$, then $p=0$ is enough.)
See Proposition~\ref{prop: extremal example} for further details.
\item Whenever $\Delta=\Theta(n^{1/k})$, we can only omit a $\log n$-factor in comparison to the $\bG(n,p)$-model.
However, in all other cases $p$ can be taken (significantly) smaller.
\item In certain ranges for $\Delta$,
increasing $\Delta$ does not lead to a change in the bound on $p$.
For example,
the class of trees with maximum degree at most $n^{1/2}$
requires the same bound on $p$ as the class of trees with maximum degree $n^{3/4}$.
See Figure~\ref{Fig1} for an illustration. 
\end{itemize}

To see that the condition $\Delta=O(n/\log n)$ is needed, observe that $\bG(n,0.9)$ does not contain vertex sets of size $o(\log n)$ that dominate the graph.
Hence the tree that arises from the disjoint union of $o(\log n)$ stars with $\Omega(n/\log n)$ leaves
by adding a vertex and joining it to the centres of the stars is not a subgraph of $\bG(n,0.9)$ (with probability $1-o(1)$).

\begin{figure}[bt]%
\begin{center}
\begin{tikzpicture}[scale = 0.7,every text node part/.style={align=center}]

\draw[thick,->] (-0.5,0)--  (6.2,0);
\draw[thick,] (5,-0.1)-- node[anchor=north] {$1$}  (5,0.1);
\draw[thick,] (5/2,-0.1)-- node[anchor=north] {$\frac{1}{2}$}  (2.5,0.1);
\draw[thick,] (5/3,-0.1)-- node[anchor=north] {$\frac{1}{3}$}  (5/3,0.1);
\draw[thick,] (5/4,-0.1)-- node[anchor=north] {$\frac{1}{4}$}  (5/4,0.1);
\draw[thick,] (5/5,-0.1)-- node[anchor=north] {}  (5/5,0.1);
\draw[thick,->] (0,-0.5)-- (0,6);

\draw[thick] (5.8,0) node[anchor=north] {$\frac{\log \Delta}{\log n}$};
\draw[thick] (0,5.8) node[anchor=east] {$\frac{\log np}{\log n}$};

\draw[thick,] (-0.1,5)-- node[anchor=east] {$1$}  (0.1,5);
\draw[thick,] (-0.1,5/2)-- node[anchor=east] {$\frac{1}{2}$}  (0.1,2.5);
\draw[thick,] (-0.1,5/3)-- node[anchor=east] {$\frac{1}{3}$}  (0.1,5/3);
\draw[thick,] (-0.1,5/4)-- node[anchor=east] {}  (0.1,5/4);
\draw[thick,] (-0.1,5/5)-- node[anchor=east] {}  (0.1,5/5);

\draw[thick]
(5/1,5/1)-- (5/4+5/2,5/2)--
(5/2,5/2)-- (2*5/2/3+5/3/3,5/3)--
(5/3,5/3)--(3*5/3/4+5/4/4,5/4)--
(5/4,5/4)--(4*5/4/5+5/5/5,5/5)--
(5/5,5/5)--(5*5/5/6+5/6/6,5/6)--
(5/6,5/6)--(6*5/6/7+5/7/7,5/7)--
(5/7,5/7)--(7*5/7/8+5/8/8,5/8)--
(5/8,5/8)--(8*5/8/9+5/9/9,5/9)--
(5/9,5/9)--(9*5/9/10+5/10/10,5/10)--
(5/10,5/10)--(10*5/10/11+5/11/11,5/11)--
(5/11,5/11)--(11*5/11/12+5/12/12,5/12)--
(5/12,5/12)--(12*5/12/13+5/13/13,5/13)--
(5/13,5/13)--(13*5/13/14+5/14/14,5/14)--
(5/14,5/14)--(14*5/14/15+5/15/15,5/15)--
(5/15,5/15)--(0,0)
;

\end{tikzpicture}
\end{center}
\caption{An illustration of the statement of Theorem~\ref{thm: main}.
The graph shows the magnitude of $np$ in terms of $\Delta$.}%
\label{Fig1}
\end{figure}
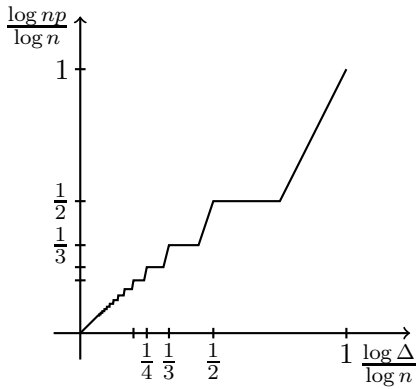

We want to point out here that randomly perturbed graphs can also be seen from a different angle as discussed above.
Let us equip the set of $n$-vertex graphs with a metric, namely the edit-distance.
Given $G_\alpha$ and a tree $T$,
we investigate how many graphs $G'$ in the $m$-neighbourhood of $G_\alpha$ contain $T$ as a subgraph.
This can be easily modelled by adding $m$ edges at random to $G$,
which is almost exactly the graph model we consider in this paper.%
\footnote{We gloss here over the technicality that graphs in the neighbourhood of $G_\alpha$ may also contain fewer edges, 
but as this does not affect the results, we
ignore this for now.}
Hence, the randomly perturbed graph model measures how `typical' a property is from a local point of view.

Randomly perturbed model appears also in theoretical computer science.
In their ground-breaking work~\cite{ST04}, Spielman and Teng introduced the notion of smoothed analysis of algorithms.
They evaluate the performance of algorithms in more `realistic' settings by using randomly perturbed inputs and
a combination of worst-case and average-case analysis.

\section{Outline of the proof}
\label{sec: outline}
Assume, as in the setting of Theorem~\ref{thm: main}, we have an $n$-vertex graph $G$ with minimum degree at least $\alpha n$, an $n$-vertex tree $T$ with maximum degree at most $\Delta$, and a random graph $R=\mathbf{G}(n,Mp)$ on the vertex set $V(G)$. 
We aim to embed $T$ into $G\cup R$.
We extensively use the following facts.
\begin{itemize}
\item[($a$1)] For any set $U\subseteq V(G)$ of size which is linear in $n$,
a $(1-\epsilon)|U|$-vertex forest $F$ with maximum degree $O(np)$ embeds into $R[U]$ (see Lemma~\ref{lem: random embedding} and Lemma~\ref{lem: random embedding simple}).  Moreover, the image of $V(F)$ is a random subset of $U$ (see Remark~\ref{rmk: symmetry}).
\item[($a$2)] For any set $U\subseteq V(G)$ of size which is linear in $n$, 
a $|U|$-vertex forest $F$ with maximum degree $O(\frac{np}{\log{n}})$ embeds into $R[U]$ (see Lemma~\ref{lem: light leaves embedding}).
\item[($b$1)] Suppose $G'$ is a bipartite graph with vertex partition $(V_1,V_2)$
and  $F$ is a star-forest with at most $|V_2|$ leaves. 
If the centres of $F$ are already `quasi-randomly' mapped to $V_1$ and $\Delta(F)=o(|V_2|)$, 
then we can extend the mapping to an embedding of $F$ into $G'$ by embedding $L(F)$, the leaves of $F$, into $V_2$.
(see Lemma~\ref{lem: random matching behaves random} and Lemma~\ref{lem: p matching}).
 Moreover, in our embedding, we can ensure that $L(F)$ is `quasi-randomly' mapped into $V_2$ (see Lemma~\ref{lem: random matching behaves random}).
\end{itemize}

Our approach is as follows.
We first apply the regularity lemma to $G$ to obtain a subgraph $G'$ and a partition $\{V_{\ih}\}_{\ih \in [r]\times[2]}$ of $V(G)$ such that $G'[V_{(i,1)}, V_{(i,2)}]$ is $(\epsilon,d)$-super-regular for all $i \in [r]$.
We decompose $T$ into subforests $F_1,\dots, F_{k+1}$, $F'_1,\dots, F'_{k}$, $L_1$ and $\last$ such that $\Delta(F)=O(np)$ for all $F\in \{F_1,\dots, F_{k+1}, L_1\}$ and $F'_1,\dots, F'_k, L_1, \last$ are star-forests.
We embed the edges of $F_1$ into $R$ by using ($a$1) as it has maximum degree $O(np)$ and we embed $V(F'_1)$ onto the `unused' vertices of $G'$  by using the super-regularity of $G'$ and ($b$1).
Iteratively, we embed $F_2, F'_2,\dots, F'_{k}, F_{k+1}$ onto `unused' vertices of $V(G)$. 
Finally, we want to complete the exact embedding by embedding $L_1$ by using ($a$1), and $\last$ by using ($b$1) onto the remaining vertices of $G$.

For this approach, we need to make sure that we can repeat this procedure until the end. 
We can use ($a$1) for any subset $U$ of $V(G)$, thus we can always embed $F_i$ into the remaining vertices. However, in order to use ($b$1), we need to ensure that the centre vertices are `quasi-randomly' embedded. In order to ensure this, we extensively use the `moreover part' of ($a$1) and ($b$1). Every time we embed $F_i$ or $F'_i$ into $G'$, we always make sure that the image of the embedding is chosen in a `quasi-random' way (see \ref{Phi3} and \ref{Phi4} in Section~\ref{sec: embedding}).

One big obstacle for this approach is that as $T$ and $G$ both contain exactly $n$ vertices;
in particular, we need to find an exact embedding of $\last$ into $G'$ at the last step. 
Suppose first that  $T$ contains many `light' leaves, that is, leaves whose neighbour has degree at most $O({np}/{\log n})$, or many vertices of degree $2$.
Then the situation is easier as we can reserve such `light' vertices for the last step  and in the last step, we embed them into $R$ using ($a$2) or Lemma~\ref{lem: bare paths embedding} (see Section~\ref{sec: few heavy leaves}). 

Suppose now that $T$ does not have many `light' leaves nor many vertices of degree $2$.
This implies that there are many `heavy' leaves, that is, leaves whose neighbour has degree $\Omega({np}/{\log n})$. 
As ($a$1) does not apply to spanning trees, it is necessary to use ($b$1) to find an exact embedding of these `heavy' leaves at the last step. 
For this purpose, we reserve some leaves $\last$ at the beginning and we embed them at the last step by using ($b$1) to finish the algorithm.
In order to use ($b$1) for $\last$, the graph $\last$ must be a star-forest of $T$ which only consists of leaf-edges of $T$ so that we do not have to embed any more edges after embedding $\last$.

There are several further obstacles.
For example, 
after all the centres $x_1,\dots, x_{s}$ of $\last$ are embedded into $V_{i,1}$, 
the number $\sum_{i\in [s]} d_{\last}(x_i) $ of leaves attached to the centres might not equal to the number of vertices left in $V_{i,2}$. 
In this case, it is impossible to find an exact embedding using ($b$1). 
To overcome this, we reserved a set $L_1$ of leaf-edges of $T$ in the beginning. 
Furthermore, we will reserve a small subgraph $F^{\circ}$ of $\last$. 
Before we embed $\last$ into $G'$,
we embed exactly the right number of leaf-vertices of $L_1 \cup F^\circ$ into each $V_{\ih}$ for each $\ih \in [r]\times [2]$ by using ($a$1). 
Hence this problem does not occur when we are about to embed $\last\setminus F^{\circ}$.

We may also face the problem that $|\last|$ is too small (say, size of $O(1)$).
Then we may not be able to guarantee that the remaining small number of vertices still induces super-regular pairs in $G'$ which is required to use ($b$1). 
However, by a clever choice of the edge decomposition $E(T)$
into the forests described above,
it is possible to ensure that $|\last|=\Omega(n/\log{n})$ (see the definition of $\last$ and \ref{F12} in Section~\ref{sec:prep}).

Another problem is that we might not be able to obtain strong enough `quasi-randomness' on the distribution of images of centre vertices of $\last$ to apply ($b$1). 
This happens when most of centre vertices of $\last$ embedded on $V_{\ih}$ are embedded using ($b$1) rather than ($a$1).
To better estimate the `quasi-randomness', 
we define subsets $\bigcup_{i\in [k]} \LA_i \subseteq \last$ of vertices whose parents are embedded using ($a$1).
As parents of these vertices (which form the centre vertices of $\last$) are embedded by ($a$1) rather than ($b$1) (that is, they are embedded into $R$ rather than $G$), 
the distribution of the images of these centre vertices of $\last$ satisfies strong `quasi-random' assumptions.
Thus as long as enough of such vertices are embedded into each $V_{\ih}$, 
we have a sufficiently strong `quasi-randomness' distribution to apply a weaker version of ($b$1) (see Lemma~\ref{lem: p matching}). 
We can actually ensure that each $V_{\ih}$ contains enough images of such parents (see \eqref{eq: size of Xih W in case 2} in Section~\ref{sec: distribution}).

The organisation of the paper is as follows.
In Section~\ref{sec: preliminaries}, 
we introduce notation, state some probabilistic tools,
and present some results involving the graph regularity set up.
In Section~\ref{sec: embed distributing}, 
we prove ($a$1)--($b$1) and we prove Lemma~\ref{lem: vector distribution} which we use in Section~\ref{sec: distribution}
to assign the vertices of $T$ to the different sets in $\{V_{\ih}\}_{\ih \in [r]\times[2]}$.
In Section~\ref{sec:prep}, we construct the (edge) decomposition $F_1,\dots, F_{k+1},$ $F'_1,\dots, F'_{k},$ $L_1, \last$ of $T$, 
and we verify several properties of this decomposition for later use.
In Section~\ref{sec: distribution}, 
by using Lemma~\ref{lem: vector distribution}, 
we determine for each vertex $x$ in $T$ into which $V_{\ih}$ it will be embedded.
In Section~\ref{sec: embedding}, finally we construct the actual embedding of $T$ into $G\cup R$.
In Section~\ref{sec: few heavy leaves}, we consider the case when either $T$ contains many vertices of degree $2$ or not many `heavy' leaves. 
Both cases can be easily deduced from the results before.

\section{Preliminaries}
\label{sec: preliminaries}

\subsection{Basic definitions}

Let $\mathbb{N}$ denote the set of all positive integers and let $\mathbb{N}_{0}$ denote the set of all non-negative integers. 
We often treat large numbers as integers whenever this does not affect the argument.
For $n\in \mathbb{N}$, let $[n]:=\{1,\dots, n\}$.
For $a,b,c\in \mathbb{R}$, 
we write $a = b\pm c$ if $b-c \leq a \leq b+c$.
We write $\log x:=\log_e x$ for all $x>0$.
The constants in the hierarchies used to state our results are chosen from right to left. 
More precisely, if we for example claim that a result holds whenever $0< a \ll b \ll c \leq 1$, then this means that there are non-decreasing functions $f^* : (0, 1] \rightarrow (0, 1]$ and $g^* : (0, 1] \rightarrow (0, 1]$ such that the result holds for all $0 < a, b, c \leq 1 $ with $b \leq f^*(c)$ and $a \leq g^*(b)$.  
Every asymptotic notation refers to the parameter $n$ if not stated otherwise.

For a finite set $A$,  a function $f:A \rightarrow \mathbb{R}$, and $p\in \mathbb{N}$, 
we define $\norm{f}_p :=(\sum_{a\in A} |f(a)|^p)^{1/p}$ and $\norm{f}_{\infty} := \max_{a\in A} |f(a)|$.
For $\ih =(i,h) \in \mathbb{N}\times [2]$, we define $\ihb:=(i,3-h)$.

Let $A,B$ be two disjoint finite sets. 
For a function $\psi : A \rightarrow B$ and a set $A'$, 
we denote by $\left.\psi \right|_{A'}$ the restriction of $\psi$ on $A'\cap A$.
For an injective function $\psi :A\rightarrow B$, a function $f: A\rightarrow\mathbb{R}$, and a set $B'$, we define
\begin{align}\label{def: fsum}
\fsum{f}{\psi}{B'} := \sum_{a\in A \colon \psi(a)\in B'\cap B} f(a).
\end{align}
Note that if $B'$ is disjoint from $B$, then $\fsum{f}{\psi}{B'}=0$.

Let $G$ be a graph.
We slightly abuse notion by identifying a graph with its edge set.
For a collection $E$ of edges, we treat it as a graph with vertex set $\bigcup_{uv \in E} \{u,v\}$ and edge set $E$.
Let $e(G)$ be the number of edges of $G$.
Let $u,v \in V(G)$ and let $U,V\sub V(G)$ be disjoint.
We write $G[U,V]$ to denote the bipartite \mbox{(multi-)subgraph} of $G$ induced by the edges joining $U$ and $V$ and let $e_G(U,V):=e(G[U,V])$.
In addition, let $\den_G(U,V):= e_G(U,V)/(|U||V|)$.
We define the \emph{degree} $d_{G}(v)$ of $v$ in $G$ by $|N_{G}(v)|$.
We further define $d_{G}(u,v):=|N_{G}(u)\cap N_{G}(v)|$. 
We define $d_{G,U}(v):=|N_{G}(v)\cap U|$ and $d_{G,U}(u,v):= |N_{G}(u)\cap N_{G}(v)\cap U|$.
We interpret the $d_G(\cdot)$ as a function $d_G: V(G)\to \N_0$
and so we have $\fsum{d_G}{\psi}{V'}=\sum_{v\in V'\cap \psi(U)}d_G(\psi^{-1}(v))$ for any injective function $\psi: U\to V$ and $V'\sub V$.

We say that a bipartite graph $G$ with vertex partition $(A,B)$ is \emph{$(\epsilon,d)$-regular} if for all sets $A'\subseteq A$, $B'\subseteq B$ with  $|A'|\geq \epsilon |A|$, $|B'|\geq \epsilon |B|$,
we have
\begin{align*}
	| \den_G(A',B')- d| < \epsilon.
\end{align*}
If $G$ is $(\epsilon,d)$-regular for some $d\in [0,1]$, then we say $G$ is $\epsilon$-regular.
If $G$ is $(\epsilon,d)$-regular and $d_{G}(a)= (d\pm \epsilon)|B|$ for all $a\in A$ and $d_{G}(b)= (d\pm \epsilon)|A|$ for all $b\in B$, 
then we say that $G$ is {\em $(\epsilon,d)$-super-regular}. 
For a bipartite graph $G$ with vertex partition $(A,B)$,
let $J_G(A,d,\epsilon)$ be a graph (that may contain loops) defined by
\begin{align}\label{eq: def J graph}
V(J_G(A,d,\epsilon))=A \text{ and }
E(J_G(A,d,\epsilon)) = \{ aa'\colon a\in A,a'\in A, d_{G}(a,a') \neq (d^2\pm 3 \epsilon)|B|\}.
\end{align}
We call $J_G(A,d,\epsilon)$ an \emph{irregularity-graph} (with respect to $A$) of $G$.

We say that a forest $F$ is a \emph{star-forest} if every component of $F$ is a star.
A vertex of degree~$1$ is a \emph{leaf} and for a forest $F$, we denote by $L(F)$ the set of leaves of $F$.
For a forest $F$, let $C(F)$ be the collection of components of $F$ that contain at least one edge.
For a tree $T$, we let $P_T(u,v)$ be the unique path in $T$ connecting $u$ and $v$ and $|P_T(u,v)|$ be the number of edges in the path.

A tree $T$ is \emph{rooted} if it contains a distinguished vertex $r$ -- its root; 
we often write $(T,r)$ for a rooted tree with root $r$.
We say a tuple $(F,R)$ is a \textit{rooted forest} if $F$ is a forest and $R\sub V(F)$ contains exactly one vertex of every component of $F$ -- their roots.
For a vertex $u\in V(T)$ in a rooted tree $(T,r)$, 
we let $T(u)$ be the subtree of $T$ induced all vertices $w$ such that $u \in V(P_T(r,w))$.
For any subtree $T'$ of $(T,r)$, 
we let $r(T')$ be the root of $T'$, 
which is the unique vertex $x\in V(T')$ such that $V(P_T(r,x)) \cap V(T') =\{x\}$ and we always consider $T'$ as a rooted tree $(T',r(T'))$. 
For any subforest $F$ of $(T,r)$, 
let $R(F):= \{ r(T'): T'\in C(F)\}$. 
For a rooted tree $(T,r)$,
we define the \emph{height of $T$} to be the length of the longest path between $r$ and a leaf of $T$.
Moreover,
let $D_T(u) := N_T(u)\cap V(T(u))$ be the set of  \emph{descendants} (children) of $u$ and let $a_T(u):= N_T(u) \setminus V(T(u))$ be its \emph{ancestor} (which does not exist if $u$ is the root of $T$).
Let $D_T^\ell(u):= \{ w \in T(u): |P_T(u,w)| = \ell \}$ and let $S_T(u):= \{ u u' : u'\in D_T(u)\}$.
In addition, let $D_T^{0}(u):= \{u\}$.
For a star-forest $F$, 
we let $\cen(F)$ be the set of centres of all star-components of $F$ (if a star has only one edge, we assume that one vertex is given as a centre).

For a tree $T$ and a vertex $v\in V(T)$, let $(A_T(v), B_T(v))$ be the unique bipartition of $V(T)$ such that $v\in A_T(v)$ and $v\notin B_T(v)$.
For a tree $T$, we say a path $P=u_1\dots u_{k}$ is a \emph{bare path} if $d_T(u_i)=2$ for all $i\in [k]$.
For $\ell \in \mathbb{N}$ and edge-disjoint paths $P_1,\dots, P_{\ell}$, 
we write $P=P_1P_2\dots P_{\ell}$ for the concatenation of $P_1,\dots, P_{\ell}$ provided $E(P)= \bigcup_{i=1}^{\ell} E(P_i)$ forms a path.

\subsection{Probabilistic tools}

A sequence $X_0,\dots, X_N$ of random variables is a {\em martingale} if $X_0$ is a fixed real number and $\mathbb{E}[X_{n}\mid X_0,\dots,X_{n-1}] = X_{n-1}$ for all $n\in [N]$. 
Our applications of Azuma's inequality will involve \emph{exposure martingales}. 
These are martingales of the form $X_i:=\mathbb{E}[X\mid Y_1,\dots, Y_i]$, where $X$ and $Y_1,\dots,Y_i$ are some previously defined random variables.

\begin{theorem}[Azuma's inequality \cite{Azu67, Hoe63}]\label{Azuma} 
Suppose that $\lambda>0$ and that $X_0,\dots, X_N$ is a martingale such that
$|X_{i}- X_{i-1}|\leq c_i$ for all $i\in [N]$.
Then 
\begin{align*}
\bPr[\left|X_N-X_0\right|\geq \lambda]\leq 2e^{\frac{-\lambda^2}{2\sum_{i\in [N]}c_i^2}}.
\end{align*}
\end{theorem}

For $n\in \mathbb{N}$ and $0\leq p\leq 1$, 
we write $Bin(n,p)$ to denote the binomial distribution with parameters $n$ and $p$. 
For $m,n,N\in \mathbb{N}$ with $m,n<N$, 
the \emph{hypergeometric distribution} with parameters $N$, $n$ and $m$ is the distribution of the random variable $X$ defined as follows. 
Let $S$ be a random subset of $\{1,2, \dots, N\}$ of size $n$ and let $X:=|S\cap \{1,2,\dots, m\}|$. 
We will use the following bound, which is a simple form of Chernoff-Hoeffding's inequality.

\begin{lemma}[Chernoff-Hoeffding's inequality, see {\cite[Remark 2.5 and Theorem 2.10]{JLR00}}] \label{Chernoff Bounds}
Suppose that $X\sim Bin(n,p)$, 
then $\bPr[|X - \mathbb{E}(X)| \geq t] \leq 2e^{-t^2/(3pn)}$ if $t\leq 3np/2$.
If $X$ has a hypergeometric distribution with parameters $N,n,m$, 
then
$\bPr[|X - \mathbb{E}(X)| \geq t] \leq 2e^{-2t^2/n}$.
\end{lemma}

Let $\cX = (x_1, \dots , x_N)$ be a finite ordered collection of $N$ not necessarily distinct real numbers.
A random sample $(X_1,\dots, X_n)$ drawn without replacement of size $n\leq N$ from $\cX$ can be generated as follows: 
First let $I_1:= [N]$, and for each $j\in [n]$, we sequentially choose $i_j$ uniformly at random from $I_j$
and set $I_{j+1}:= I_j \setminus \{i_j\}$ and $X_j:=x_{i_j}$.

\begin{theorem}[Bernstein's inequality,  see \cite{BLM13}] \label{Bernstein}
Let $\cX=(x_1,\dots, x_N)$ be a finite collection of $N$ real numbers and let $(X_1,\dots, X_n)$ be a random sample drawn without replacement from $\cX$.
Let $a:= \min_{i\in [N]} x_i$ and $b:=\max_{i\in [N]} x_i$.
Let $\mu := \mathbb{E}[\sum_{i=1}^{n} X_i]=\frac{n}{N}\sum_{i=1}^Nx_i$ and $\sigma^2 := \frac{1}{N} \sum_{i\in [N]} (x_i-\mu/N)^2$.
Then for all $\lambda >0$, we have
$ \bPr\left[ \left|\sum_{i=1}^{n} X_i - \mu\right| \geq \lambda  \right] \leq \exp\left(- \frac{\lambda^2}{2\sigma^2n + (2/3)(b-a)\lambda}\right).$
\end{theorem}


\begin{theorem}[\cite{KKOT:ta}]\label{MM}
Suppose $n\in \mathbb{N}$ with $0<1/n \ll \epsilon \ll d\leq 1$.
Let $G$ be an $(\epsilon,d)$-super-regular bipartite graph with vertex partition $(U,W)$ such that $|U|=|W|=n$.
Suppose $\sigma:U\to W$ is a perfect matching in $G$ chosen uniformly at random (where we treat a perfect matching as a bijection), then for any edge $uv\in E(G)$ with $u\in U$ and $v\in W$, we have
$\bPr[ \sigma(u)=v ] = (1\pm \epsilon^{1/20}) /{dn}.$
\end{theorem}

The following three lemmas are explicitly stated in~\cite{Kri10};
however,
the first one is elementary and the third is a special case of multipartite version of the seminal result of Johansson, Kahn, and Vu~\cite{JKV08}.

\begin{lemma}\label{lem: number of leaves}
Suppose $k, \ell, n \in \mathbb{N}$ and $T$ is an $n$-vertex tree with at most $\ell$ leaves. 
Then $T$ contains a collection of at least $n/k- 2\ell$ vertex-disjoint $k$-vertex bare paths.
\end{lemma}

\begin{lemma}\label{lem: light leaves embedding}
Suppose $k,n,\Delta\in \N$ and $d_1,\dots, d_{k}$ are non-negative integers such that $d_i\leq \Delta$ and $\sum_{i=1}^{k} d_i = n$. 
Let $A=\{a_1,\dots, a_k\}$ and $B$ be a set of $n$ vertices that is disjoint from $A$. 
Let $G= \bG(|A\cup B|,p)$ with $V(G)=A\cup B$. 
If $p\geq {2\Delta \log{n}}/{n}$, then $G$ contains with probability $1-o(1)$ a collection $S_1,\dots, S_k$ of vertex-disjoint stars such that 
$a_i$ is the centre of $S_i$ and $S_i$  has $d_i$ neighbours in $B$.
\end{lemma}

\begin{lemma}\label{lem: bare paths embedding}
Let $k\geq 4$ and $0< 1/n \ll 1/M \ll 1/k$. Let $G = \bG( kn,p)$. Let $S=\{s_1,\dots, s_n\}$, $T=\{t_1,\dots, t_n\}$ be disjoint vertex subsets of $V(G)$. If $p \geq M(\frac{\log{n}}{n^{k-2}})^{1/(k-1)}$, then $G$ contains a family $\{P_i: i\in [n]\}$ of vertex-disjoint paths 
where each $P_i$ is $k$-vertex path connecting $s_i$ and $t_i$ with probability $1-o(1)$.
\end{lemma}

\subsection{Results involving $\epsilon$-regularity}

The following lemma shows that if $G$ is an $\epsilon$-regular bipartite graph, then its irregularity-graph is small.
\begin{lemma}\label{lem: max degree HGde}
Suppose $0<2 \epsilon < d <1$, and that $G$ is an $(\epsilon,d)$-super-regular graph with vertex partition $(A,B)$.
Then $\Delta(J_G(A,d,\epsilon)) \leq 2\epsilon |A|$.
\end{lemma}
\begin{proof}
For each $u\in A$, let $A_u^{+}:= \{ u'\in A: d_{G}(u,u') > (d^2+ 3\epsilon)|B|\}$ and
$A_{u}^{-}:= \{ u'\in A: d_{G}(u,u') < (d^2 - 3\epsilon)|B|\}$.
Then $d_{J_G(A,d,\epsilon)}(u) = |A_u^{+}| + |A_u^{-}|$.

It is easy to see that $\den_{G}(A_u^{+}, N_{G}(u)) > d+\epsilon$ and $\den_{G}(A_u^{-}, N_{G}(u)) < d-\epsilon$ (observe that $u\notin A_u^{-}$). 
Since $G$ is $(\epsilon,d)$-super-regular and $|N_{G}(u)|\geq (d-\epsilon)|B| \geq \epsilon |B|$, we conclude that $|A_u^{+}|< \epsilon |A|$ and $|A_u^{-}|<\epsilon |A|$.
Thus $d_{J_G(A,d,\epsilon)}(u) = |A_u^{+}| + |A_u^{-}| \leq 2\epsilon |A|$ for any $u\in A$. 
\end{proof}

Next theorem is proved in \cite{DLR95}. 
(In~\cite{DLR95} it is proved in the case when $|A|=|B|$ with $16\epsilon^{1/5}$ instead of $\epsilon^{1/6}$.
The version stated below can be easily derived from this.)
This with the previous lemma together asserts that a bipartite graph $G$ being $\epsilon$-regular is roughly equivalent to an appropriate irregularity-graph of $G$ being small.

\begin{theorem}\label{thm: almost quasirandom}
Suppose $n\in \N$ with $0<1/n\ll\epsilon \ll \alpha,p \leq 1$.
Suppose $G$ is a bipartite graph with vertex partition $(A,B)$ such that $\alpha n\leq |A| \leq \alpha^{-1}n$, $|B|=n$ and
$| E(J_G(A,d,\epsilon))|  \leq \epsilon n^2.$
Then $G$ is $(\epsilon^{1/6},d)$-regular. 
\end{theorem}

The following lemma is proved in \cite{KKS17}. Note that our definition of super-regularity is slightly different from theirs. (Their definition ensures that a super-regular pair has a lower bound on minimum degree, on the other hand our definition ensures that a super-regular pair has both an upper and a lower bound on the degree of a vertex. 
This notion is required when we use Theorem~3.4.)
So we introduce a subgraph $G'$ to adjust the statement to our setting.

\begin{lemma}[\cite{KKS17}]\label{lem: partition}
Suppose $0< 1/C \ll 1/C', \epsilon, d, 1/t \ll \alpha$. 
Let $G$ be an $n$-vertex graph with $\delta(G)\geq \alpha n$. 
Then there exists a partition $\{ V_\ih\}_{\ih\in [r]\times[2]}$ of $v(G)$ and a subgraph $G'\subseteq G$ such that $C'\leq r \leq C$ and  $G'[V_{(i,1)}, V_{(i,2)}]$ is a $(\epsilon,d)$-super-regular bipartite graph and $n/(tr) \leq |V_{(i,1)}| \leq |V_{(i,2)}| \leq tn/r$ for all $i\in [r]$.
\end{lemma}

\subsection{Sharpness example}

We also prove the following proposition which shows that the bound on $p$ in Theorem~\ref{thm: main} is sharp up to a constant factor.

\begin{proposition}\label{prop: extremal example}
Suppose $n,k\in \mathbb{N}$ with $0< 1/n \ll c \ll 1/k \leq 1$
and suppose $G=K_{\frac{n}{2},\frac{n}{2}}$.
If $n^{1/(k+1)} \leq \Delta <  n^{1/k}$, 
then there exists an $n$-vertex tree $T$ with $\Delta(T)\leq \Delta$ such that 
$T\nsubseteq G\cup R$ with probability $1-o(1)$ where
$R\in \bG(n, c p/2)$ is a random graph on $V(G)$ and $p= \max\big\{ n^{- \frac{k}{k+1}}, \Delta^{k+1}n^{-2}\big\}$.
\end{proposition}
\begin{proof}
Let $\cE$ be the event that every vertex $v\in V(R)$ satisfies $d_{R}(v)  \leq cnp$. 
It is easy to see that $\mathbb{P}[\cE] \geq 1- n^{-1}$.
It is enough to show that there exists an $n$-vertex tree $T$ with $\Delta(T)\leq \Delta$
such that whenever $\cE$ holds then $T\nsubseteq G\cup R$.

First, assume that $\Delta > n^{\frac{k+2}{(k+1)^2}}$. 
Then
$p= \Delta^{k+1}n^{-2} > n^{- \frac{k}{k+1}}$.
Let $s \in \{ \lceil \frac{n}{(\Delta/2)^k}\rceil ,\lceil \frac{n}{(\Delta/2)^k}\rceil+1\}$ be an odd integer.
Thus  $s\geq 3$ as $\Delta^k < n$.
Let $m:= \lfloor (n-1)/s\rfloor$ and consider an $m$-vertex tree $T'$ with root $\hat{x}$ of height $k$ 
such that every internal vertex has degree at least $\Delta/5$ and at most $2\Delta/3$. 
Note that such a tree exists as $\sum_{\ell=0}^{k} (2\Delta/3-1)^{\ell} > m > \sum_{\ell=0}^{k} (\Delta/5)^{\ell}\geq \Delta^k/5$.
Hence
\begin{align}\label{eq: Dk size}
|D^{k}_{T'}(\hat{x})| \geq |V(T')| - \sum_{\ell\in \{0,\ldots,k-1\}} |D^{\ell}_{T'}(\hat{x})| 
\geq m - \Delta^{k-1} 
\geq (1- 5 \Delta^{-1} ) m.
\end{align}

We consider $s$ vertex-disjoint copies $T_1,\dots, T_{s}$ of $T'$ with the roots $x_1,\dots, x_{s}$, respectively.
We add 
a vertex $x$ and edges $xx_i$ for all $i\in [s]$
and further
at most $s$ vertices and edges in such a way that the new graph $T$ is an $n$-vertex tree with $\Delta(T)\leq \Delta$.
Assume, for a contradiction, $\cE$ holds and there exists an embedding $\phi$ of $T$ into $G\cup R$. 

For each $i\in [s]$, we let $U_i$ be the set of all vertices $y\in D_{T_i}^k(x_i)$ such that all edges of the path that joins $x_i$ and $y$ are embedded into $E(G)$; that is,
$$U_i:= \{ y \in D^{k}_{T_i}(x_i) : E(\phi(P_{T_i}(x_i, y))) \subseteq E(G )\}.$$
Let $(A,B)$ be the bipartition of $G$.
As for each $y\in U_i$, every edge in $E(\phi(P_{T_i}(x_i,y)))$ joins a vertex in $A$ with a vertex in $B$ , 
it is easy to see that either $\phi(U_i)\subseteq A$ or $\phi(U_i)\subseteq B$ holds depending on whether $\phi(x_i)\in A$ or $\phi(x_i)\in B$.
Observe that for each $y\in V(T_i)\sm U_i$, 
the vertex $y$ is contained in $T_i(z')$ for some 
$z\in V(T_i)$ and $z'\in D_{T_i}(z)$ such that $\phi(zz') \subseteq E(R)$. 
Hence
\begin{eqnarray*}
|U_i| &\geq& |D^{k}_{T_i}(x_i)| - \sum_{\substack{z\in V(T_i), z'\in D_{T_i}(z)\colon \phi(zz') \subseteq E(R) }} |T_i(z')| \\
&\geq& |D^{k}_{T_i}(x_i)| - \sum_{\ell \in [k]} \sum_{z \in D^{\ell-1}_{T_i}(x_i) } \sum_{z'\in D_{T_i} (z)\cap N_R(\phi(z)) } |T_i(z')|  \\
&\geq& |D^{k}_{T_i}(x_i)| - \sum_{\ell \in [k]} \sum_{z \in D^{\ell-1}_{T_i}(x_i) } 2\Delta^{k-\ell} d_{R}(\phi(z))  \\
&\stackrel{\eqref{eq: Dk size}}{\geq} &  (1- 5 \Delta^{-1} ) m -  \sum_{\ell \in [k]} 2\Delta^{k-1} \cdot 2cnp 
\geq (1- 5 \Delta^{-1} ) m - \sqrt{c}\Delta^{2k} n^{-1}.  
\end{eqnarray*}
Here, we obtain the third inequality as $|T_i(z)|\leq 2\Delta^{k-\ell}$ for $z\in D^{\ell}_{T_i}(x_i)$. 
As $s$ is an odd integer and each $U_i$ entirely belongs to $A$ or $B$, 
without loss of generality,
we can assume that there exists $I\sub [s]$ such that $|I|\geq (s+1)/2$ and $\bigcup_{i\in I} U_i \subseteq A$.
Then 
\begin{align*}
\bigg|\bigcup_{i\in I} U_i \bigg| 
&\geq \frac{s+1}{2}\left((1- 5 \Delta^{-1} ) m - \sqrt{c}\Delta^{2k} n^{-1}\right)\\
&\geq n/2 + m/2 - s - 10 n/\Delta - s \sqrt{c}\Delta^{2k} n^{-1} \\
&\geq  n/2 + \Delta^k/4 - 20 n/\Delta  - \frac{2n}{(\Delta/2)^k} \cdot \sqrt{c}\Delta^{2k} n^{-1}\\
&\geq n/2 + \Delta^k/8.
\end{align*}
This is a contradiction as either all vertices of $\phi(\bigcup_{i\in I} U_i)$ belong to $A$ or 
all vertices of $\phi(\bigcup_{i\in I} U_i)$ belong to $B$ whereas $|A|=|B|=n/2$.

Now assume that $\Delta < n^{\frac{k+2}{(k+1)^2}}$. 
Thus $p = n^{- \frac{k}{k+1}}$.
Consider an $n$-vertex tree $T$ with root $r$ of height $k+1$ such that every internal vertices has degree at least $np/2 = n^{-1/(k+1)}/2$ and at most $2np= 2n^{-1/(k+1)}$. 
It is easy to see that such a tree $T$ exists.
We assume for a contradiction that $T$ embeds into $G\cup R$ and $\cE$ holds.
It is easy to check (by a similar, but much a simpler argument as before) 
that at least $(1 - 10kc)n > n/2$ vertices are embedded into one of either $A$ or $B$. 
However, this is a contradiction as $|A|=|B|=n/2$.
\end{proof}

\section{Embedding and Distributing lemmas}

\label{sec: embed distributing}

In this section, we state and prove several lemmas, which we use later in the proof of our main result.
In Lemma~\ref{lem: random embedding},
we show that if $p$ is large enough, then a rooted forest can be embedded into $\bG(n,p)$ in such a way that certain additional properties hold;
in particular, we can specify target sets for every vertex and also require that the vertices are well-distributed with respect to a specified weight function.

\begin{lemma}\label{lem: random embedding}
Suppose $0< 1/n, 1/w \ll \epsilon,1/r, 1/t$ and $w \leq \log^2{n}$.
Suppose that $(F,R)$ is a rooted forest and $(R, X_1,\dots, X_r,X'_1,\dots, X'_r)$ is a partition of $V(F)$.
Suppose $V$ is a set of $n$ vertices,
 $(U,V_1,\dots, V_r, V'_1,\dots, V'_r)$ is a partition of $V$, and $G=\bG(n,p)$ is a random graph on $V$.
Suppose that $\phi':R\rightarrow U$ is an injective function and $f: \bigcup_{i\in [r]} X'_i \rightarrow [0,1]$.
Suppose that for each $i\in [r]$, we have a multiset $\cB_i$ of subsets of $V'_i$
and suppose that the following hold for each $i\in [r]$:
\begin{enumerate}[label={\rm (A\arabic*)$_{\ref{lem: random embedding}}$}]
\item\label{L41A1}  $|V_i| \geq |X_i| + 12 p^{-1} \Delta(F) + 30 p^{-1} \log{n}$,
\item\label{L41A2} 
$|V'_i| \geq  p^{-1} \log^6{n}$ and $|X'_i| \leq \log^2{n}$,
\item\label{L41A3} for each $B\in \cB_i$, we have $|B| \leq \epsilon |V'_i|$, and
\item\label{L41A4} $\sum_{x'\in X'_i} f(x')  \leq 1$ and for each $x\in X'_i$, we have  $0\leq f(x)\leq w^{-1}$.
\end{enumerate}
Then with probability at least $1- n^{-2}$, there exist a multiset $\cB'_i\subseteq \cB_i$ for each $i\in [r]$ and an embedding $\phi$ of $F$ into $G$ which extends $\phi'$ such that the following hold for each $i\in [r]$: 
\begin{enumerate}[label={\rm (B\arabic*)$_{\ref{lem: random embedding}}$}]
\item\label{L41B1} $\phi(X_i)\subseteq V_i$ and $\phi(X'_i)\subseteq V'_i$, 
\item\label{L41B2}  $|\cB'_i| \leq 2^{-\epsilon^2 w} |\cB_i|$, and
\item\label{L41B3} for each $B\in \cB_i\setminus \cB'_i$, we have
$\fsum{f}{\phi}{B} \leq \epsilon^{1/2}$.
\end{enumerate}
\end{lemma}
\begin{proof}
For each $i\in [r]$, let $n_i:=|V'_i|$, $m_i := |X'_i|$ and $q_i:=\sum_{x'\in X'_i} f(x')$.
Note that \ref{L41A2} implies 
\begin{align}\label{eq: sec 4 X'i sizes}
n_i= |V'_i| \geq \log^6{n} \geq w^3 \enspace \text{and} \enspace m_i = |X'_i| \leq  \log^2{n}.
\end{align}
By adding some vertices to some sets $B\in \cB_i$ if necessary, 
we may assume that $|B| =\epsilon n_i$ for all $i\in [r]$ and $B\in \cB_i$. Note that if we obtain a function $\phi$ and a multiset $\cB'_i$ satisfying \ref{L41B1}--\ref{L41B3} for these multisets, 
then $\phi$ also satisfies \ref{L41B1}--\ref{L41B3} for the original multisets.

For each component $T$ of $F$, we consider a breath-first-search ordering $(x^T_{1},\dots, x^T_{|V(T)|})$ of each component of $F$, starting with its root $\{x^T_{1}\} = R\cap V(T)$.
Whenever we have the choice, we give the vertices in $X_1\cup\dots\cup X_r$ priority over the vertices in $X'_1\cup\dots\cup X'_r$;
that is, for every vertex $x$ in $T$, the children of $x$ in the former set precede the children of $x$ in the latter set.
We consider an ordering $\pi=(x_1,\dots, x_{|V(F)|})$ 
such that $R= \{x_1,\dots, x_{|R|}\}$ and 
$(x_{|R|+1},\dots, x_{|V(F)|})$ is an arbitrary concatenation of $(x^T_{2},\dots, x^T_{|V(T)|})$ of all components $T$ of $F$. 
Then for all $j\in [|V(F)|]\setminus [|R|]$, 
\begin{equation}\label{eq: BFS ordering}
\begin{minipage}[c]{0.9\textwidth} \em
the ancestor of $x_j$ precedes $x_j$ and all children of $x_j$ appear consecutively after $x_j$.
\end{minipage}
\end{equation}
We remark that if we dropped the conditions  \ref{L41B2} and  \ref{L41B3},
then a simply greedy algorithm would yield the desired statement.

As for all $i\in [r]$, we have $X'_i\subseteq V(F)$, 
the ordering $\pi$ naturally gives rise to an ordering $(x_{i,1},\dots, x_{i,m_i} )$ of vertices in $X'_i$ 
such that if $x_{i,j} = x_{\ell}$ and $x_{i,j'}=x_{\ell'}$ for some $j<j' \in [m_i]$, 
then $\ell <\ell'$. 
For each $s\in [m_i]$, let $X'_{i,s}:=\{ x_{i,1},\dots, x_{i,s}\}$.

For all $i\in [r]$,  $s'\leq s \in [m_i]\cup \{0\}$, and $q\in [0,q_i]$, let
$\cW:=( f(x_{i,s+1}),\dots, f(x_{i,m_i}) , 0, \dots, 0)\in \R^{n_i -s}$, 
and let $(W'_{s'+1},\dots, W'_{\epsilon n_i})$ be a random sample drawn without replacement from $\cW$.
Let  
\begin{align}\label{eq: Wi ss'q def}
p^s_i(s',q) := \mathbb{P}\bigg[ q+ \sum_{j \in [\epsilon n_i]\setminus [s']} W'_j \geq  \epsilon^{1/2}\bigg].
\end{align}
First we estimate $p_i^0(0,0)$.
We have 
\begin{align*}
\mu:= 
\mathbb{E}\bigg[\sum_{j \in [\epsilon n_i] } W'_i \bigg] 
= \epsilon \sum_{j \in [m_i]} f(x_{i,j}) = \epsilon q_i \text{ and } 
\sigma^2:= \frac{1}{n_i} \sum_{j\in [m_i]} (f(x_{i,j})- \epsilon q_i/n_i )^2 + \sum_{j\in [n_i]\setminus [m_i]} ( 0 -\epsilon q_i/n_i)^2.
\end{align*}
Then by the convexity of sums of squares, we conclude that
\begin{eqnarray*}
\sigma^2  
&\leq& \frac{1}{n_i} \Big(\frac{q_i}{\max_{x'\in X_i'}f(x')}\Big( \max_{x'\in X_i'}f(x')- \frac{\epsilon q_i}{n_i} \Big)^2 
+ \Big( n_i- \frac{q_i}{\max_{x'\in X_i'}f(x')}\Big)\Big( \frac{\epsilon q_i}{n_i}\Big)^2 \Big) \\
&\stackrel{\text{\ref{L41A4}}}{\leq}& \frac{1}{n_i} \Big(w q_i\Big( w^{-1}- \frac{\epsilon q_i}{n_i} \Big)^2 + ( n_i- w q_i)\Big( \frac{\epsilon q_i}{n_i}\Big)^2 \Big) 
\leq \frac{2}{w n_i}.
\end{eqnarray*}
\COMMENT{
\begin{align*}
&\sigma^2:= \frac{1}{n_i} \sum_{j\in [m_i]} (f(x_{i,j})- \mu/n_i )^2 + \sum_{j\in [n_i]\setminus [m_i]} ( 0 -\mu/n_i)^2 \leq \frac{1}{n_i} \left(w q_i( w^{-1}- \frac{\epsilon q_i}{n_i})^2 + ( n_i- w q_i)( \frac{\epsilon q_i}{n_i})^2 \right) \\
&\leq \frac{q_i }{w n_i} - \frac{2\epsilon q_i^2}{n_i^2} + \frac{\epsilon^2 q_i^2}{n_i^2} 
\leq \frac{2}{w n_i}.
\end{align*}
}
Here, we obtain the final equality since \eqref{eq: sec 4 X'i sizes} implies $n_i \geq w^3$ and \ref{L41A4} implies $q_i\leq 1$.
As $\mu = \epsilon q_i \leq \epsilon$, Bernstein's inequality~(Theorem~\ref{Bernstein}) implies that
\begin{align}\label{eq: Wi000 value}
p^0_i(0,0) 
\leq \exp\Big(- \frac{ (\epsilon^{1/2}-\epsilon)^2 }{ 4 w^{-1}  + (2/3)(w^{-1}-0)\epsilon^{1/2}  } \Big) 
\leq \exp(- \epsilon w/5).
\end{align}
Moreover, it is easy to see that the following holds for any $s'\in [m_i]$:
\begin{align}\label{eq: Wi at the end}
p^{m_i}_i(s',q) = \left\{\begin{array}{ll} 1 & \text{ if } q \geq \epsilon^{1/2} \\
0 & \text{ if } q <  \epsilon^{1/2}.
\end{array}\right.
\end{align}
Recall the notation introduced in \eqref{def: fsum}.
Suppose $B\in \cB_i$ and $\psi: X'_{i,s}\rightarrow V'_i$ is an injective function such that 
$|B\cap \psi(X'_{i,s})|=s'$ and $\fsum{f}{\psi}{B}=q$.
Suppose $\psi'$ is an injective function chosen uniformly at random among all injective functions from $X'_{i}$ to $V'_i$ extending $\psi$. 
Let $\cE_{B}$ be the event that  $\fsum{f}{\psi'}{B} \geq \epsilon^{1/2}$.
Then
$$\mathbb{P} [ \cE_{B}]= p^s_i(s',q).$$
Furthermore, we observe that
\begin{align}\label{eq: Wi relation}
p^s_i(s',q)&
= \mathbb{P}\big[ \psi'(x_{i,s+1}) \in B\big]\mathbb{P}\big[ \cE_{B} \mid \psi'(x_{i,s+1}) \in B\big] + \mathbb{P}\big[ \psi'(x_{i,s+1}) \notin B\big] \mathbb{P} \big[\cE_{B}\mid \psi'(x_{i,s+1}) \notin B\big] \nonumber \\
&= \frac{\epsilon n_i - s'}{n_i -s} p^{s+1}_i(s'+1,q+ f(x_{i,s+1}))
+ \frac{(1-\epsilon) n_i - s+ s'}{n_i -s}p^{s+1}_i(s',q) \nonumber \\
&= (1\pm  \log^{-3}n )\Big( \epsilon p^{s+1}_i(s'+1,q+ f(x_{i,s+1})) + (1-\epsilon) p^{s+1}_i(s',q) \Big).
\end{align}
Here, we obtain the final equality from \eqref{eq: sec 4 X'i sizes} as $s,s' \leq m_i$. 
For all $i\in [r]$, $s\in [m_i]$ and an injective function $\psi: X'_{i,s}\rightarrow V_i$, we define
$$E_i(\psi):= \sum_{B\in \cB} p^s_i\big(|B\cap \psi(X'_{i,s})|, \fsum{f}{\psi}{B} \big).$$
Moreover, for $i\in [r]$ and an empty function $\psi_0:\emptyset\rightarrow \emptyset$, we have
\begin{eqnarray}\label{eq: E000 value}
E_i(\psi_0) = p^0_i(0,0) |\cB| \stackrel{\eqref{eq: Wi000 value}}{\leq}
\exp(-\epsilon w/5) |\cB|. 
\end{eqnarray}
Roughly speaking,
$E_i(\psi)$ measures how `good' the partial embedding $\psi$ is.
To ensure \ref{L41B2} and \ref{L41B3},
we aim to choose $\phi$
such that $E_i(\phi|_{X_i'})$ is not too large.

Let $G_1, G_2$ be pairwise independent random graphs such that $G_1\cup G_2\subseteq G$ and $G_1,G_2=\bG(n,p/3)$.
Let $\cE_0$ be the event that for all $i\in [r]$, $v\in V(G)$ and $B\in \cB_i$ , we have 
\begin{align}\label{eq: B and V' sizes p}
d_{G_1,V'_i}(v) = n_i p \pm (n_ip)^{3/5} \enspace \text{and}  \enspace 
d_{G_1,B}(v) = p|B| \pm (n_ip)^{3/5} = \epsilon n_ip \pm (n_ip)^{3/5}.
\end{align}
Note that \ref{L41A2} implies that $n_ip \geq \log^6{n}$.
Thus Chernoff's inequality (Lemma~\ref{Chernoff Bounds}) implies that 
\begin{eqnarray}\label{eq: cE1 prob}
\mathbb{P}[\cE_0] \geq 1- n^{-4}.
\end{eqnarray}

Now we begin our algorithm which gradually extends $\phi'$ to our desired embedding $\phi$ of $F$ into $G$ in at most $n$ steps, and each step will be successful with probability at least $1- n^{-4}$. 
The success of each step only depends on whether a potential set of edges
in $G_1$ or $G_2$ contains roughly as many edges as we expect it to have.
These potential sets of edges will be disjoint.

First, assume that $\cE_0$ holds (this is the only property of $G_1$ we will use).
Let $\phi_{|R|}:=\phi'$.
Assume we have defined $\phi_h$ for some $h \in [ |V(F)| ]\sm [|R|-1]$
satisfying the following, where $X^h:= \{x_1,\dots, x_h\}$:
\begin{enumerate}[label=(\text{$\Phi$\arabic*)$_{\ref{lem: random embedding}}^h$}]
\item\label{sec 4 Phi 1} $\phi_h$ embeds $F[X^h]$ into $G$, 
\item\label{sec 4 Phi 2} for each $i\in [r]$, we have $\phi_h(X_i\cap X^h) \subseteq V_i$, and 
$\phi_h(X'_i\cap X^h) \subseteq V'_i$,
\item\label{sec 4 Phi 3} for each $i\in [r]$,
we have $E_i( \left.\phi_h\right|_{X'_i}) \leq (1+ 3\log^{-2}{n} )^{|X^h\cap X'_i|} \exp(-\epsilon w/5) |\cB_i|$, and
\item\label{sec 4 Phi 4} for each $h'\in [h]$, either all children or no child of $x_{h'}$ lie in $X^h$.
\end{enumerate}
Clearly, $\phi_{|R|}$ satisfies ($\Phi$1)$_{\ref{lem: random embedding}}^{|R|}$, ($\Phi$2)$_{\ref{lem: random embedding}}^{|R|}$ and ($\Phi$4)$_{\ref{lem: random embedding}}^{|R|}$.
By \eqref{eq: E000 value}, ($\Phi$3)$_{\ref{lem: random embedding}}^{|R|}$ holds as well.
Let $x$ be the unique neighbour of $x_{h+1}$ in $\{x_1,\ldots,x_h\}$ (the ancestor of $x_{h+1}$) and $y:=\phi_h(x)$.
By \eqref{eq: BFS ordering} and~\ref{sec 4 Phi 4}, the set of children of $x$ is $\{x_{h+1},\dots, x_{h+d}\}$
for some $d\in [\Delta(F)]$.
For each $i\in [r]$, let 
$\widehat{V}_i := V_{ i }\setminus \phi_h( X^h).$ 
By our choice of the ordering $(x_1,\dots, x_{|V(F)|})$, there exists $d'\in [d]$ such that
$$\{x_{\ell_1},\dots, x_{\ell_{d'}} \} = \{x_{h+1}\dots x_{h+d}\}\cap \bigcup_{i\in [r]} X'_i,  
\enspace \text{ where } \enspace \ell_j := h+d-d'+j \text{ for each }j\in [d']. $$
Now we expose the neighbours of $y$ in $G_2$ 
and let $\cE_{h}$ be the event that $d_{G_2,\widehat{V}_i}(y) \geq \Delta(F)$ for each $i\in [r]$.
Note that 
$$|\widehat{V}_i| 
\stackrel{\ref{sec 4 Phi 2}}{\geq}  |V_i| - |X_i|
 \stackrel{\text{\ref{L41A1}}}{\geq} 12p^{-1}\Delta(F) + 30p^{-1}\log{n}.$$
Thus, by Chernoff's inequality (Lemma~\ref{Chernoff Bounds}), we have
$$\mathbb{P}[\cE_h] 
\geq 1 - r \mathbb{P}\big[ {\rm Bin}( |V_i| - |X_i| , p/3 ) < \Delta(F) \big]
\leq 1- e^{ - \frac{ (3\Delta(F) + 10 \log{n})^2}{ 3(4\Delta(F) + 10\log{n}) }}  < 1 - n^{-3.2}.$$
Now, assume that $\cE_h$ holds.
For each $j\in [d-d']$, 
we embed vertices in $\{x_{h+1},\dots, x_{h+d-d'}\}\cap X_i$ onto distinct vertices in $N_{G_2}(y)\cap \widehat{V}_{i}$.
We denote the new embedding by $\widehat{\phi}_0$ and by construction $\widehat{\phi}_0$ embeds  $X^h \cup \{x_{h+1}\dots x_{h+d-d'}\}$ into $G$.
Note that by construction $\widehat{\phi}_0$ satisfies ($\Phi$1)$_{\ref{lem: random embedding}}^{\ell_0}$--($\Phi$2)$_{\ref{lem: random embedding}}^{\ell_0}$ where $\ell_0:= h+d-d'$.
Moreover, ($\Phi$3)$_{\ref{lem: random embedding}}^{\ell_0}$ also holds by 
($\Phi$1)$_{\ref{lem: random embedding}}^{h}$ as
$X^{\ell_0}\cap X_i = X^{h}\cap X_i$ for each $i\in [r]$.

Now we want to iteratively extend  $\widehat{\phi}_{j}$ to $\widehat{\phi}_{j+1}$ for $j\in \{0,\ldots,d'-1\}$ in such a way
that $\widehat{\phi}_j$ satisfies ($\Phi$1)$_{\ref{lem: random embedding}}^{\ell_j}$--($\Phi$3)$_{\ref{lem: random embedding}}^{\ell_j}$.
We assume now that for some $j\in [d']\cup \{0\}$ we have defined $\widehat{\phi}_j$ satisfying ($\Phi$1)$_{\ref{lem: random embedding}}^{\ell_j}$--($\Phi$3)$_{\ref{lem: random embedding}}^{\ell_j}$.
Let $\widehat{V}'_i:= V'_{i}\setminus \widehat{\phi}_j(X^{\ell_j})$.
Note that since $\cE_0$ holds, for each $i\in [r]$ and $B\in \cB_i$, we have
\begin{eqnarray}\label{eq: hat V B R1 good}
 d_{G_1, \widehat{V}'_i\cap B} (y) &\stackrel{\eqref{eq: sec 4 X'i sizes}}{=}& d_{G_1, V'_i\cap B}(y) \pm \log^2{n} 
 \stackrel{\eqref{eq: B and V' sizes p},\text{\ref{L41A2}}}{=} \epsilon d_{G_1, V'_i}(y)  \pm 2(n_ip)^{3/5} \nonumber \\
&\stackrel{\text{\ref{L41A2}} }{=}&(1 \pm \log^{-2}{n})\epsilon d_{G_1,V'_i}(y).
\end{eqnarray}
We next embed $x_{\ell_{j+1}}$.
To ensure ($\Phi$3)$_{\ref{lem: random embedding}}^{\ell_{j+1}}$, 
we want to use a vertex $u \in N:=N_{G_1,\widehat{V}'_i}(y)$ as the image for $x_{\ell_{j+1}}$
which does not increase the `$E_i$-value' too much.
We now show that such a vertex exists.
Note that there exist $i\in [r]$ and $s\in [m_i-1]\cup \{0\}$ such that $x_{\ell_{j+1}} = x_{i,s+1}$.
We define $\psi:= \widehat{\phi}\big|_{X'_{i,s}}$
and for each $u \in N$,
let $\psi_u$ be a function extending $\psi$ by defining $\psi_u(x_{\ell_{j+1}}) := u$.
For each $B\in \cB_i$, we write $b_B := |B\cap \psi(X'_{i,s})|$.
Hence
\begin{align}\label{eq: Ei psi u estimate}
&\sum_{u\in N} E_i(\psi_u) 
= \sum_{u\in N} \sum_{B\in \cB} p^{s+1}_i\big(|B\cap \psi_u(X'_{i,s+1})|, \fsum{f}{\psi_u}{B}\big) \nonumber \\ 
&=\sum_{B\in \cB} \sum_{u\in N\cap B} p^{s+1}_i \big(b_B+1, \fsum{f}{\psi}{B}+ f(x_{\ell_{j+1}})\big) 
+ \sum_{B\in \cB} \sum_{u\in N\sm B} p^{s+1}_i\big(b_B, \fsum{f}{\psi}{B} \big) \nonumber \\
&= \sum_{B\in \cB} d_{G_1,\widehat{V}'_i\cap B}(y) p^{s+1}_i\big(b_B+1, \fsum{f}{\psi}{B}+ f(x_{\ell_{j+1}})\big)  + \sum_{B\in \cB} d_{G_1,\widehat{V}'_i\setminus B}(y) p^{s+1}_i\big(b_B, \fsum{f}{\psi}{B} \big).
\end{align}
Therefore,
\begin{eqnarray*}
&& \hspace{-2cm} d_{G_1,\widehat{V}'_i}(y)^{-1} \sum_{u\in N} E_i(\psi_u)   \\
&\stackrel{\eqref{eq: hat V B R1 good},\eqref{eq: Ei psi u estimate}}{\leq}& 
 \left(1+  2\log^{-2}{n}\right) \sum_{B\in \cB} \left( \epsilon p^{s+1}_i\big(b_B+1, \fsum{f}{\psi}{B}+ f(x'_{j+1})\big) + (1-\epsilon) p^{s+1}_i\big(b_B, \fsum{f}{\psi}{B} \big) \right)  \\
&\stackrel{(\ref{eq: Wi relation})}{\leq} &  (1+ 3 \log^{-2}{n}) E_i(\psi).
\end{eqnarray*}
This shows that there exists a choice $u\in N$ such that 
$$E_i(\psi_u) \leq (1+3\log^{-2}{n}) E_i(\psi) 
\stackrel{\text{($\Phi$3)$_{\ref{lem: random embedding}}^{\ell_j}$}}{\leq }(1+ 3\log^{-2}{n})^{s+1}\exp(-\epsilon^2 w/5) |\cB_i|.$$
We let $\widehat{\phi}_{j+1}$ be a function which arises from $\widehat{\phi}_{j}$ by defining $\widehat{\phi}_{j+1}(x_{\ell_{j+1}}) = u$. 
Observe that $\widehat{\phi}_{j+1}$ extends $\widehat{\phi}_j$ and ($\Phi$1)$_{\ref{lem: random embedding}}^{\ell_{j+1}}$--($\Phi$3)$_{\ref{lem: random embedding}}^{\ell_{j+1}}$ hold.
By repeating this, as $\ell_{d'}= h+d$, we obtain $\widehat{\phi}_{d'}$ satisfying
($\Phi$1)$_{\ref{lem: random embedding}}^{h+d}$--($\Phi$3)$_{\ref{lem: random embedding}}^{h+d}$.
Let $\phi_{h+d}:=\widehat{\phi}_{d'}$. 
Then ($\Phi$4)$_{\ref{lem: random embedding}}^{h+d}$ also holds by the choice of the ordering $(x_1,\dots, x_{|V(F)|})$.

Observe that the algorithm completes the embedding of $F$ whenever $\cE_0 \wedge \bigwedge_{ x_h \notin L(F) }\{\cE_h\} $  hold,
which happens with probability at least $1-n^{-2}$. 
In this case, we define $\phi:= \phi_{|V(F)|}$ and
 for each $i\in [r]$, let 
$\cB'_i:=\{B\in \cB_i : \fsum{f}{\phi}{B} \geq \epsilon^{1/2}\}.$
Hence \ref{L41B3} holds.
Note that \eqref{eq: Wi at the end} implies that 
\begin{eqnarray*}
|\cB'_i| = E_i(\phi) \stackrel{\text{($\Phi$3)$_{\ref{lem: random embedding}}^{|V(F)|}$}}{\leq} 
(1+ 3\log^{-2}{n})^{m_i}\exp(-\epsilon w/5) |\cB_i| 
\stackrel{\eqref{eq: sec 4 X'i sizes}}{=}100 \exp(-\epsilon w/5) |\cB_i|\leq 2^{- \epsilon^2w}|\cB_i|.
\end{eqnarray*}
Therefore, \ref{L41B2} holds and
($\Phi$1)$_{\ref{lem: random embedding}}^{|V(F)|}$ and ($\Phi$2)$_{\ref{lem: random embedding}}^{|V(F)|}$ imply that \ref{L41B1} hold.
\end{proof}
\begin{remark}\label{rmk: symmetry}
One can use Lemma~\ref{lem: random embedding} to verify that there exists an embedding $\phi$ of $(F,R)$ into a random graph satisfying \ref{L41B1}--\ref{L41B3}.
For any permutation $\sigma$ acting on $V(G)$ such that $\sigma(V_i) = V_i$ and 
$\left.\sigma\right|_{V'_i}$ is identity map for each $i\in [r]$, 
the injective map $\sigma\circ \phi$ also satisfies \ref{L41B1}--\ref{L41B3}.
For any permutation $\sigma$, $\sigma(G)$ is again a random graph with the same distribution as $G$. 
Let $\cE$ be the event that there exists an embedding $\phi$ satisfying \ref{L41B1}--\ref{L41B3}.
Thus, for any two permutations $\sigma$ and $\sigma'$ on $V_i$, 
conditional on $\cE$, if we choose an embedding $\phi$ satisfying \ref{L41B1}--\ref{L41B3} uniformly at random
then we have
$\mathbb{P}[\left.\phi\right|_{V_i} =\sigma \mid \cE ] =\mathbb{P}[ \left.\phi\right|_{V_i} = \sigma' \mid \cE]$.
In addition, once we assume $\cE$ holds and we have chosen $\phi$ uniformly among all functions satisfying \ref{L41B1}--\ref{L41B3}, for a function $g: V(F)\rightarrow \mathbb{N}$ and any set $U'\subseteq V_i$, the random variable $\sum_{u\in U} g(\phi^{-1}(u))$ is distributed as a random variable which is sampled without replacement from a multiset $\{ f(v) : v\in \phi^{-1}(i)\}$ exactly $|U'|$ times.
\end{remark}

Note that by essentially same proof as  Lemma~\ref{lem: random embedding}, we can prove the following lemma. We omit the proof here.
Note that the above remark also applies for Lemma~\ref{lem: random embedding simple}.

\begin{lemma}\label{lem: random embedding simple}
Suppose $0< 1/n \ll 1/r$.
Suppose that $(F,R)$ is a rooted forest and $(X_1,\dots, X_r)$ is a partition of $V(F)\setminus R$, and 
$\{U\}\cup \{V_i:i\in [r]\}$ is a collection of pairwise disjoint sets and $|\bigcup_{i\in [r]} V_i \cup U|=n$.
Suppose the graph $G=\bG(n,p)$ is a random graph on vertex set $\bigcup_{i\in [r]} V_i \cup U$.
Suppose the following hold:
\begin{enumerate}[label={\rm (A\arabic*)$_{\ref{lem: random embedding simple}}$}]
\item
\label{A1,41}
$|V_i| \geq |X_i| + 12 p^{-1} \Delta(F) + 30 p^{-1} \log{n}$ for each $i\in [r]$, and
\item \label{A2,41} $\phi'$ is an injective map of $R$ into $U$.
\end{enumerate}
Then with probability at least $1- n^{-2}$, there exists an embedding $\phi$ of $F$ into $G$ which extends $\phi'$ such that $\phi(X_i) \subseteq V_i$ for all $i\in [r]$.
\end{lemma}

Suppose $U,V$ are two disjoint sets with $|U|\leq |V|$ and $f$ is a weight function on $U$.
The next lemma shows that a random injective function $\sigma :U\to V$ behaves nicely with respect to $\fsum{f}{\sigma}{B}$
for some priorly specified sets $B\subseteq V$.

\begin{lemma}\label{lem: randomly distributing}
Suppose $0<1/n \ll  \epsilon \ll 1/s, 1/t, 1/k$.
Suppose that $U,V$ are disjoint sets such that $|U|\leq |V|=n$.
Suppose that $f_1,\ldots,f_s:U\to \mathbb{N}_0$ and $w_1,\ldots, w_s \in \mathbb{N}$ such that $\norm{f_i}_1 \leq tn$ and 
$\norm{f_i}_\infty \leq {n}/{w_i}$ as well as $w_i \geq \epsilon^{-3} \log{n}$ for each $i\in [s]$.
Suppose $\cB_1,\ldots,\cB_s$ are multisets of subsets of $V$ and $|\cB_i|\leq n^k$ for each $i\in [s]$.
Let $\sigma: U\to V$ be an injective function chosen uniformly at random among all possible injective functions from $U$ into $V$.
Then with probability at least $1- n^{-3}$, for each $i\in [s]$ and $B\in \cB$, we have
$$\fsum{f_i}{\sigma}{B} = \frac{|B|\cdot \norm{f_i}_1 }{|V|} \pm  \epsilon \sqrt{w_i n \norm{f_i}_\infty} .$$
\end{lemma}
\begin{proof}
We assume for now that $i\in [s]$ is fixed.
For each $v\in V$,
let $X_v:= \fsum{f_i}{\sigma}{v}$ be the random variable which equals $f_i(u)$ if $\sigma(u)=v$ for some $u\in U$ and $0$ otherwise.
For any $B\sub V$, we define
\begin{align*}
\mu_{B}:=\sum_{b \in B}\Exp[X_b]=  \frac{|B|\cdot\norm{f_i}_1}{|V|} ,
\end{align*}
and observe that 
\begin{align*}
 \sum_{u\in U }\Big( f_i(u) - \frac{\mu_{B}}{|V|}\Big)^2 
+ (|V|-|U|)\Big(0 - \frac{\mu_{B}}{|V|}\Big)^2
\leq \norm{f_i}_2^2+ \frac{\mu_{B}^2}{|V|}
\leq \norm{f_i}_2^2+ \frac{\norm{f_i}_1^2}{|V|}
\leq 2\norm{f_i}_1 \norm{f_i}_\infty.
\end{align*}
Bernstein's inequality (Lemma~\ref{Bernstein}) implies that
\begin{align*}\notag
\bPr\bigg[ \Big|\sum_{b\in B } X_b- \mu_{B}\Big| >    \epsilon\sqrt{w_i n \norm{f_i}_\infty } \bigg]  
&\leq \exp\Big( -\frac{ \epsilon^2 w_i n \norm{f_i}_\infty }{ 4 \norm{f_i}_1\norm{f_i}_\infty  + (2 /3)\epsilon w_i^{1/2} n^{1/2} \norm{f_i}_\infty^{3/2}  } \Big)  \\
&\leq \exp\bigg( - \min\Big\{ \frac{\epsilon^2 w_i}{8t}, \frac{3\epsilon w_i^{1/2} n^{1/2}  }{4\norm{f_i}_\infty^{1/2}}  \Big\} \bigg) \leq n^{-10k}.
\end{align*}
Here, we obtain the second inequality since $\norm{f_i}_1 \leq t n$ and because
${x}/({y+z})\geq \min\{x/(2y),x/(2z)\}$ for all $x,y,z\in \R^+$.
We obtain the final inequality since $\norm{f_i}_\infty  \leq {n}/w_i$ and $w_i \geq \epsilon^{-3} \log{n}$.
A union bound over all $i\in [s]$ and $B\in \cB$
implies that with probability at least $1- n^{-10k} \sum_{i\in [s]}|\cB_i| \geq 1- n^{-3}$, 
for each $B\in \cB_i$, we have 
$$\fsum{f}{\sigma}{B} = \sum_{b\in B} X_b 
=  \frac{|B| \cdot\norm{f_i}_1 }{|V|} \pm  \epsilon \sqrt{w_i n \norm{f_i}_\infty }.$$
\end{proof}
Note that in the above lemma, we may take $w_i$ smaller than $n/\norm{f_i}_\infty$ to obtain a stronger concentration bound.

The next lemma shows that we can embed a star-forest with weights on its leaves into a `quasi-random' bipartite graph
such that the weights are distributed nicely.

\begin{lemma}\label{lem: random matching behaves random}
Suppose $0< 1/n \ll \epsilon \ll d, 1/t, 1/k <1$, and $s\leq n$.
Suppose that $G$ is a bipartite graph with vertex partition $(U,V)$ and $|V|=n$.
Suppose $\cB_1,\ldots,\cB_s$ are multisets of subsets of $V$ and $|\cB_i| \leq n^k$. 
Suppose $F$ is a star-forest with at most $n$ leaves and $\psi:\cen(F)\rightarrow U$ is an injective function.
Suppose $f_1,\ldots,f_s: L(F) \to \mathbb{N}_0$ are functions such that the following hold:
\begin{itemize}
\item[{\rm (A1)$_{\ref{lem: random matching behaves random}}$}] $\norm{f_i}_\infty \leq { \epsilon^4 n}/{ \log^2{n}}$ and $\norm{ f_i}_1 \leq t n$ for all $i\in [s]$.
\item[{\rm (A2)$_{\ref{lem: random matching behaves random}}$}] For each $u\in U$,  
we have $d_{G}(u) = (d \pm \epsilon)n$ and 
for each $v\in V$, we have
$\fsum{d_F}{\psi}{N_G(v)}=\sum_{x\in V(F)\colon \psi(x)\in N_G(v)} d_F(x)=d|L(F)| \pm \epsilon n.$
\item[{\rm (A3)$_{\ref{lem: random matching behaves random}}$}] We have $\displaystyle \sum_{x,x'\colon\psi(x)\psi(x') \in E(J_G(U,d,\epsilon))}  d_{F}(x) d_{F}(x')  \leq \epsilon n^2.$
\end{itemize}
Then there exists an embedding $\phi$ of $F$ into $G$ which extends $\psi$ and satisfies the following:
\begin{itemize}
\item[{\rm (B1)}$_{\ref{lem: random matching behaves random}}$] For all $i\in [s]$ and $B \in \cB_i$, 
we have $$\fsum{f_i}{\phi}{B} = \sum_{x\in \cen(F) }\sum_{y\in N_{F}(x)}  \frac{f_i(y) | N_{G}( \phi(x) )\cap B| }{ dn } \pm  \epsilon^{1/200} n.$$
\end{itemize}
\end{lemma}
\begin{proof}
Observe that we may assume that $\psi$ is a bijection by ignoring the vertices in $U$ outside the image of $\psi$.
Our strategy for the proof is as follows.
We replace every vertex $u\in U$ by $d_{F}(\psi^{-1}(u))$ distinct copies of $u$ and obtain a new bipartite graph $G'$. 
Clearly, there is a bijection between the matchings in $G'$ covering all the copies of vertices in $U$ and the embeddings of $F$ into $G$. 

We write $\ell:= |L(F)|$ and $\{u_1,\dots, u_{m}\}:= U$.
Let $U^*:=\{  u_{i,j}\colon i\in [m], j\in [d_F(\psi^{-1}(u_i))]\}\cup \{u_{0,1},\dots, u_{0,n-\ell} \} $. 
We claim that there exists a bipartite graph $G'$ with vertex partition $(U^*,V)$ such that the following hold:
\begin{itemize}
\item[(a1)] $|E(J_{G'}(U^*,d,2\epsilon))| \leq 2\epsilon n^2$,
\item[(a2)] for each $w\in U^*\cup V$, we have $d_{G'}(w)= (d\pm 2\epsilon)n$, and
\item[(a3)] for each $u_{i,j}\in U^*$ with $i>0$, we have $N_{G'}(u_{i,j})= N_{G}(u_i)$.
\end{itemize} 
To see that such a graph $G'$ exists, we let $N_{G'}(u_{i,j}):= N_{G}(u_i)$ for each $u_{i,j}\in U^*$ with $i>0$, and for each $j\in [n-\ell]$, 
let $N_{G'}(u_{0,j})$ be a subset of $V$ of size $dn$ chosen independently and uniformly at random.
Chernoff's inequality (Lemma~\ref{Chernoff Bounds}) implies that with probability at least $1- n^{-1}$, 
we have the following for all $j\in [n-\ell], v\in V$ and $u_{i,j'}\in U^*\sm\{u_{0,j}\}$:
\begin{align*}
d_{G'}(v) 
&=  \fsum{d_F}{\psi}{N_{G}(v)}  +d(n-\ell)\pm \epsilon n 
\stackrel{\rm (A2)_{\ref{lem: random matching behaves random}}}{=} 
  (d\pm 2\epsilon) n,\\
d_{G'}(u_{0,j}, u_{i,j'}) &=  (d^2 \pm 2\epsilon )n.
\end{align*}
These bounds on the (co)degrees imply that
\begin{align*}
|E(J_{G'}(U^*,d,2\epsilon))| 
=\big|\{ u_{i,j}u_{i',j'} : u_i u_{i'} \in E(J_{G}(U,d,2\epsilon))\big| \leq \sum  d_{F}(x) d_{F}(x') 
\stackrel{\rm (A3)_{\ref{lem: random matching behaves random}}}{\leq} \epsilon n^2.
\end{align*}
Here, the summation in the third term is over all $(x,x')$ with $\psi(x)\psi(x') \in E(J_G(U,d,\epsilon))$. 
Therefore, (a1)--(a3) hold and in particular such a $G'$ exists.

We fix a bijection $\tau\colon \{  u_{i,j} : i\in [m], j\in [d_F(\psi^{-1}(u_i))] \}\to L(F)$ such that 
$\tau(u_{i,j}) \in N_{F}(\psi^{-1}(u_i))$ for all $i\in[m]$ and $j\in [d_F(\psi^{-1}(u_i))]$.
Observe that any matching $\sigma:U^*\to V$ in $G'$ covering $\{  u_{i,j}\colon i\in [m], j\in [d_F(\psi^{-1}(u_i))] \}$ yields an embedding of $F$ into $G$
by mapping $x\in L(F)$ onto $\sigma(\tau^{-1}(x))$.
We will show that a perfect matching of $G'$ chosen uniformly at random among a large set perfect matching leads to an embedding with the desired properties with probability a least $1/2$.

For each $u \in U^*$, let $g_i(u)$ be the `$f_i$-value' of the corresponding leaf given by $\tau$; that is,
$$g_i(u):= \left\{\begin{array}{ll}
f_i( \tau(u)) & \text{ if } u=u_{i',j'} \text{ for some } i'>0, \\
0 & \text{ if } u= u_{0,j'} \text{ for some } j'\in [n-\ell].
\end{array}\right.$$
Let $T:=\epsilon^{-1}\log n$.
Next we partition $U^*$ into sets $U^1,\dots, U^{T}$ and $V$ into sets $V^1,\dots, V^{T}$ 
such that the following hold for all $u,u'\in U^*, v\in V, i\in [s], \ell' \in [T]$ and $B\in \cB_i$:
\begin{itemize}
\item[(a4)] $|U^{\ell'}|= |V^{\ell'}|  = \lfloor {(n+ \ell' -1)}/{T} \rfloor$,
\item[(a5)] $d_{G',V^{\ell'}}(u), d_{G',U^{\ell'}}(v) = (d\pm 3\epsilon)	|U^{\ell'}|$ and  
$d_{G',V^{\ell'}}(u,u') =  d_{G',V}(u,u')/T \pm  n^{2/3},$
\item[(a6)] $\sum_{u'' \in U^{\ell'}} g_i(u'' ) 
=  m_i/T \pm {\epsilon n}/{T}$, where $m_i:=\sum_{x\in L(F)} f_i(x)$,
\item[(a7)]  $|V^{\ell'} \cap B| 
= |B|/T \pm n^{2/3}$
and $|N_{G}(u)\cap B\cap V^{\ell'} | = |N_{G}(u)\cap B|/T \pm n^{2/3}$, and
\item[(a8)]  $|E(J_{G'[U^{\ell'}]}(U^{\ell'},d,2\epsilon))| \leq 3\epsilon |U^{\ell'}|^2$.
\end{itemize}
Indeed, such a partition exists, 
because a random partition $U_1,\dots, U_{T}$ of $U^*$ and $V_1,\dots, V_{T}$ of $V$ chosen uniformly at random such that 
$|U_i|= |V_i|  = \lfloor {( n + i -1)}/T \rfloor$ satisfies (a4)--(a8) with probability at least $1/2$. 
Indeed, that property (a4) holds by construction,  
(a5) and (a7) holds with probability at least $1-n^{-1}$ by Lemma~\ref{Chernoff Bounds}, 
(a6) holds with probability at least $1-n^{-1}$ by Bernstein's inequality (Lemma~\ref{Bernstein}) and 
(a8) holds with probability at least $1-n^{-1}$ by Lemma~\ref{Chernoff Bounds} (for example by showing that every vertex $u\in U^ *$ satisfies 
$d_{J,U^{\ell'}}(u) \leq  d_{J}(u)/T +n^{2/3} $ for each $\ell'\in [T]$ where $J= J_G(U^*,d,2\epsilon)$ with probability at least $1-n^{-2}$.)%
\COMMENT{
To show (a6), we can use Bernstein's inequality.
Note that $\sum_{u\in U_j} g_i(u)$ has a distribution of random sample drawn without replacement from $X$. Moreover, we have $\mu = \frac{\epsilon}{\log{n}} m_i $ and by the convexity of square-sum, we have $\sigma^2 \leq \frac{1}{n} \sum_{u\in U} g_i(u)^2 \leq \frac{\epsilon^4 nm_i} {\log^{2} n}$.
By using these with some calculations, we can see that (a6) holds with probability at least $1-n^{-5}$. By taking union bound, all of the above hold with probability at least $1/2$.
}

By Theorem~\ref{thm: almost quasirandom},  (a5) and (a8) imply that for each $\ell'\in [T]$, the bipartite graph $G[U^{\ell'},U^{\ell'}]$ is $(\epsilon^{1/7},d)$-super-regular.
For each $\ell' \in [T]$, 
we select a perfect matching $\sigma_{\ell'} : U^{\ell'} \rightarrow V^{\ell'}$ of $G[U^{\ell'},V^{\ell'}]$ uniformly at random and 
let $\sigma := \bigcup_{\ell' \in [T]} \sigma_{\ell'}$. 
(We aim to define later $\phi=\sigma\circ \tau^{-1}$.)
Hence, for all $B \in \cB_i$ and $i\in [s]$, we conclude
\begin{eqnarray}\label{eq: Expect}
\mu_{i,B}\hspace{-0.3cm}
&:=& \hspace{-0.3cm} \mathbb{E} \Big[  \sum_{b \in B }\ g_i(\sigma^{-1}(b))  \Big] 
=\sum_{\ell'=1}^{T} \mathbb{E} \Big[  \sum_{b \in B\cap V^{\ell'}   }  g_i(\sigma^{-1}(b)) \Big]
 =\sum_{\ell'=1}^{T} \sum_{u\in U^{\ell'}} g_i(u)\Pro\big[\sigma_{\ell'}(u)\in B\cap V^{\ell'}\big] \nonumber \\
&\stackrel{\rm Thm~\ref{MM}}{=}&\hspace{-0.3cm}
\sum_{\ell' =1}^{T} \sum_{u\in U^{\ell'}} g_i(u) \frac{1\pm 2\epsilon^{1/140}}{ dn/T} |N_{G'}(u)  \cap B\cap V^{\ell'} | \nonumber  \\
& \stackrel{{\rm (a7)}}{=} & \hspace{-0.3cm}
\sum_{\ell'=1}^{T} \sum_{u\in U^{\ell'}}g_i(u)\Big( \frac{|N_{G'}(u)  \cap B| /T \pm n^{2/3} }{ dn/T } \pm  3\epsilon^{1/140} \Big) \nonumber \\
& =& \hspace{-0.3cm} \sum_{u\in U^*}\frac{g_i(u) |N_{G'}(u)  \cap B| }{ dn }   \pm  \epsilon^{1/150} n =  \hspace{-0.2cm} \sum_{x\in \cen(F) }\sum_{y\in N_{F}(x)}  \hspace{-0.3cm} \frac{f_i(y) | N_{G}( \phi(x) )\cap B| }{ dn } \pm  \epsilon^{1/150} n.
\end{eqnarray}
For each $\ell'\in [T]$, let $X^i_{j}(\ell'):= \mathbb{E}[\sum_{b \in B  }g_i(\sigma^{-1}(b)) \mid \sigma_1,\dots, \sigma_{\ell'}].$ 
Then $\{X^i_j(\ell')\}_{\ell'\in \{0,\ldots,T\}}$ is an exposure martingale.
Also it is easy to see that $|X^i_j(\ell')-X^i_j(\ell'-1)|\leq {2t n}/T$
as changing $\sigma_{\ell'}$ can change the value of $X^i_{j}(\ell')$ 
at most $\sum_{u\in U^{\ell'}} g_i(u)\leq  m_i/T + \epsilon n/T \leq {2t n}/T$ (by (a6)).
Thus Azuma's inequality shows that
$$\mathbb{P}\Big[ \sum_{b \in B  } g_i(\sigma^{-1}(b))  = \mu_{i,B} \pm \epsilon^{1/180} n  \Big]
\geq 1- 2 \exp\Big( -\frac{ \epsilon^{1/90}n^2 }{  2T ({2t  n}/T)^2 } \Big) 
\geq 1 - \exp(- \epsilon^{-1/2} \log{n} ).$$
Thus a union bound over all $i\in [s]$, $B \in \cB_i$ together with \eqref{eq: Expect} shows 
that there exists a perfect matching $\sigma $ of $G'[U^*,V]$ such that for any $B\in \cB_i$ and $i\in [s]$ the following holds:
\begin{align}\label{eq: eqeq}
\fsum{f_i}{\sigma \circ \tau^{-1}}{B}=\sum_{b \in B } g_i(\sigma^{-1}(b)) = \sum_{x\in \cen(F) }\sum_{y\in N_{F}(x)}  \frac{f_i(y) | N_{G}( \phi(x) )\cap B'| }{ dn } \pm \epsilon^{1/180} n.
\end{align}
This also yields an embedding $\phi=\sigma \circ \tau^{-1}$ of $F$ into $G$ as desired.
\end{proof}

The following is an easy lemma showing the existence of an embedding of a star-forest
into a bipartite graph that satisfies mild quasi-random properties.
\begin{lemma}\label{lem: p matching}
Suppose $0< 1/n \ll \epsilon \ll d <1$ and $0\leq \nu<\epsilon$.\COMMENT{Here, $\nu$ could be 0 or $1/n$.}
Suppose that $G$ is a bipartite graph with vertex partition $(U,V)$ and $ |V|=n$. 
Let $F$ be a star-forest with at most $n$ leaves and $\psi:\cen(F)\rightarrow U$ is an injective map.
Suppose the following hold:
\begin{enumerate}[label={\rm (A\arabic*)$_{\ref{lem: p matching}}$}]
\item\label{L45A1}
For each $u\in U$, 
we have $d_{G}(u) = (d\pm \epsilon)n$.
\item \label{L45A2}
For all $v\in V$, 
we have $\fsum{d_{F}}{\psi}{N_{G}(v)} \geq \nu n$,
and for all $v\in V$  except at most $\nu n$ vertices, 
we have $\fsum{d_{F}}{\psi}{N_{G}(v)}  = (d \pm \epsilon) n$.
\item\label{L45A3}
We have
$$\sum_{x,x'\colon\psi(x)\psi(x') \in E(J_G(U,d,\epsilon))}  d_{F}(x) d_{F}(x')  \leq \epsilon n^2.$$
\end{enumerate}
Then there exists an embedding $\phi$ of $F$ into $G$ which extends $\psi$.
\end{lemma}
\begin{proof}
We may assume that $\psi$ is a bijection by ignoring the vertices in $U$ outside the image of $\psi$.
Let $\{u_1,\dots, u_{m}\}:= U$.
Next we replace  each vertex $u_i\in U$ by $d_F(\psi^{-1}(u))$ copies of $u_i$. 
Let $U^*:=\{  u_{i,j} : i\in [m], j\in [d_F(\psi^{-1}(u_i))]\}$
and let $G'$ be the bipartite graph with vertex partition $(U^*,V)$ and
$E(G'):= \{ u_{i,j} v: u_iv \in E(G) \}.$
Then \ref{L45A3} implies that $|E(J_{G'}(U^*,d,\epsilon))| \leq \epsilon n^2$.
Consequently, Theorem~\ref{thm: almost quasirandom} implies that $G'$ is $(\epsilon^{1/6},d)$-regular.

Let $V':= \{ v\in V: \fsum{d_{F}}{\psi}{N_{G}(v)} < (d- \epsilon^{1/6}) n\}.$
Then~\ref{L45A2} implies that $|V'|\leq \nu n$.
Since all $v\in V'$ satisfy $d_{G'}(v)\geq \nu n$,
we can greedily pick a matching $M'$ in $G'$ of size $|V'|$ covering $V'$.
As $\nu <\epsilon$, the graph $G' \setminus V(M')$ is still $(2\epsilon^{1/6},d)$-regular, and every vertex $u \in V(G')\setminus V(M')$ satisfies $| N_{G'}(u)\setminus V(M') | = (d\pm 2\epsilon^{1/6})n$.
Theorem~\ref{MM}  implies that $G' \setminus V(M')$ contains a perfect matching $M''$.
Hence $M'\cup M''$ is a perfect matching in $G'$, 
which leads to the desired embedding $\phi$.
\end{proof}

The following lemma provides a partition of a collection of vectors in $\N_0^6$
into well-balanced parts.
We use this lemma later to assign subforests of $T$ to different clusters of the regularity partition (see Section~\ref{sec: distribution}).
Recall that for $\ih=(i,h)\in \mathbb{N}\times [2]$, we write $\ihb$ for $(i,3-h)$.
To be a bit more precise,
for a graph $G$,
suppose we have a partition $\{V_\ih\}_{\ih\in [r]\times[2]}$ of $V(G)$ as given by Lemma~\ref{lem: partition}
and suppose we decide to embed a subtree $T'$ into $G[V_\ih,V_\ihb]$.
Suppose $(A,B)$ is the unique vertex bipartition of $T'$ such that $r(T')\in A$
and we further decide that $r(T')$ shall be embedded into $V_\ih$.
Then $A$ has to be embedded into $V_\ih$ and $B$ into $V_\ihb$.
We associate a vector $\bq\in \N^6$ with such a decision (one for each subtree)
where each coordinate captures how many vertices of a certain type are embedded into certain clusters due to this decision.
Then a partition of the decision vectors corresponds to a assignment of subtrees to vertex classes.

\begin{lemma}\label{lem: vector distribution}
Suppose $r,  \Delta_1,\Delta_2,\Delta_3 \in \mathbb{N}$ and $0<1/r \ll \beta \ll 1/t \leq 1$.
Suppose $\bF \subseteq \mathbb{N}_0^6$ and $\{\alpha_{\ih}\}_{\ih\in [r]\times [2]}$ is a probability distribution on $[r]\times [2]$.
Suppose the following holds for all $\bq=(q_1,\ldots, q_6) \in \bF$:
\begin{itemize}
\item[{\rm (A1)$_{\ref{lem: vector distribution}}$}] $q_1=0$ or $q_2=0$ or $\frac{q_1}{q_2} > 2r^2$ or $\frac{q_2}{q_1} > 2r^2$,
\item[{\rm (A2)$_{\ref{lem: vector distribution}}$}] $q_1, q_2 \leq \Delta_1$, $q_3, q_4 \leq \Delta_2$, $q_5, q_6\leq \Delta_3$, and
\item[{\rm (A3)$_{\ref{lem: vector distribution}}$}]$1/(2tr)\leq \alpha_{\ih} \leq 2t/r$ for each $\ih\in [r]\times [2]$.
\end{itemize}
Then there exists a partition $\{ \bF_{\ih}\}_{\ih\in [r]\times[2] }$ of $\bF$ such that the following hold:
\begin{itemize}
\item[{\rm (B1)$_{\ref{lem: vector distribution}}$}]  For each $\ih, \ih'\in [r]\times [2]$, we have
$$ \Bigg| \alpha_{\ih}^{-1} \Bigg(\sum_{\bq\in \bF_{\ih}} q_1 + \sum_{\bq\in \bF_{\ihb}} q_2 \Bigg) -  
\alpha_{\ihb}^{-1}\Bigg(\sum_{\bq\in \bF_{\ih'}} q_1 + \sum_{\bq\in \bF_{\ihb'}} q_2 \Bigg) \Bigg| 
\leq r^5 \Delta_1,$$
\item[{\rm (B2)$_{\ref{lem: vector distribution}}$}]   $\sum_{\bq\in \bF_{\ih} } q_3 + \sum_{\bq\in \bF_{\ihb}} q_4 
\geq  \alpha_{\ih}\beta^2 \sum_{ \bq\in \bF} (q_3+q_4) - r^2 \Delta_2,$ and
\item[{\rm (B3)$_{\ref{lem: vector distribution}}$}]   $\sum_{\bq\in \bF_{\ih} } q_5 + \sum_{\bq\in \bF_{\ihb}} q_6 
\geq  \alpha_{\ih}\beta^2 \sum_{ \bq\in \bF} (q_5+q_6) -  r^2 \Delta_3.$\COMMENT{We could add a $\beta$ here if needed.}
\end{itemize}
\end{lemma}
\begin{proof}
Our strategy for the proof is as follows.
We first assign a few vectors of $\bF$ randomly to the $2r$ parts of the future partition to ensure 
that (B2)$_{\ref{lem: vector distribution}}$ and (B3)$_{\ref{lem: vector distribution}}$ already hold for every $\ih\in [r]\times [2]$.
Afterwards we greedily assign the rest according to some target function that ensures that (B1)$_{\ref{lem: vector distribution}}$ holds.

For $j\in [3]$, let $m_j:= \sum_{\bq\in \bF} (q_{2j-1}+q_{2j})$.
For each vector $\bq$, we choose an index $i_{\bq}\in [r]\times[2] \cup \{0\}$ independently at random such that 
$\ih \in [r]\times[2]$ is chosen with probability $\beta/(2r) $ and $0$ is chosen with probability $1 - \beta $.
For each $I\in [r]\times[2] \cup \{0\}$, $\ih\in[r]\times[2]$, and  $j\in [3]$, let
\begin{align}\label{eq:Tj}
	\bF'_{I} := \{ \bq : i_{\bq} = I \} \quad\text{ and }
	\quad T^j_{\ih}:= \sum_{\bq\in \bF'_{\ih} } q_{2j-1} + \sum_{\bq\in \bF'_{\ihb}} q_{2j}.
\end{align}
Note that for each $\ih\in [r]\times[2]$ and $j\in [3]$, we have
\begin{align*}
\mathbb{E}[T^j_{\ih}] &= 
\frac{\beta}{2r} \sum_{\bq\in \bF} q_{2j-1} + \frac{\beta}{2r} \sum_{\bq\in \bF} q_{2j} = \frac{\beta m_j}{2r}.
\end{align*}
For each $j\in [3], \ih\in [r]\times[2]$, 
let $\cE(j,\ih)$ be the event that 
$$T^{j}_{\ih} = \frac{\beta m_j}{2r} \pm  \max\left\{ \frac{\beta m_j}{4r} , \frac{\beta r^2\Delta_j}{2} \right\}.$$

Let $T^{j}_{\ih}(s) := \mathbb{E}[T^{j}_{\ih} \mid i_{\bq(1)},\dots, i_{\bq(s)}]$ be an exposure martingale 
where $\bq(1), \dots, \bq(|\bF|)$ is an arbitrary ordering of $\bF$. 
It is easy to check that $|T^{j}_{\ih}(s+1)-T^{j}_{\ih}(s)| \leq \max\{ \bq(s)_{2j-1}, \bq(s)_{2j} \}$.
Thus Azuma's inequality (Theorem~\ref{Azuma}) implies that 
\begin{align*}
\Pro[\cE(j,\ih)] 
&= 1 - 2 \exp\bigg( \frac{ - \beta^2\max\{ m_j^2/(16r^2),  r^4 \Delta_j^2/4\} }{ 2\sum_{ \bq\in \bF} \max\{ (\bq_{2j-1})^2, (\bq_{2j})^2 \}} \bigg) \\
&\geq 1- 2 \exp\bigg( \frac{ - \beta^2\max\{ m_j^2 /(16r^2) ,  r^4 \Delta_j^2/4  \}}{2 m_j \Delta_j } \bigg) 
\geq 1 - 2 e^{-\beta^2 r/50}.
\end{align*}
The final inequality is easy to verify by considering the case $m_j\geq r^3 \Delta_j$ and $m_j < r^3 \Delta_j$ separately.
Thus, a union bound over all $\ih\in [r]\times[2]$ 
with the fact that $1 - 2r\cdot 2e^{-\beta^2 r/50} > 0$ as $1/r \ll \beta$ 
ensures that there exists an assignment such that $\cE(j,\ih)$ holds for all $j\in [3]$ and $\ih\in [r]\times[2]$. 
By some abuse of notation, we let $\{ \bF'_{\ih}\}_{\ih\in [r]\times[2]} $ be a such choice
and let $\bF'_0:=\bF\sm \{ \bF'_{\ih} : \ih\in [r]\times[2]\}$.
Observe that provided $\bF_\ih'\sub \bF_{\ih}$ for every $\ih\in [r]\times [2]$
and some partition $\{\bF_{\ih}\}_{\ih\in [r]\times [2]}$ of $\bF$,
by {\rm (A3)$_{\ref{lem: vector distribution}}$},
both (B2)$_{\ref{lem: vector distribution}}$ and (B3)$_{\ref{lem: vector distribution}}$ hold.

For a partition $\cF = \{ \bF^*_{\ih}\}_{\ih\in [r]\times[2] }$ of $\bF'_0$ (recall that $T_\ih^1$ is defined in~\eqref{eq:Tj}), let
\begin{align*} w_{\cF}(\ih)&:= \sum_{\bq \in \bF^*_{\ih}}q_1 + \sum_{\bq\in \bF^*_{\ihb}}q_2, &  t_{\cF}(\ih)&:= T^1_{\ih} + w_{\cF}(\ih),\\
t_{\max}(\cF)&:= \max_{\ih\in [r]\times[2]} \alpha_{\ih}^{-1} t_{\cF}(\ih), &
t_{\min}(\cF)&:= \min_{\ih\in [r]\times[2]} \alpha_{\ih}^{-1} t_{\cF}(\ih), \\
I_{\max}(\cF) &:= \{ \ih \in [r]\times [2] : \alpha_{\ih}^{-1} t_{\cF}(\ih)=  t_{\max}(\cF) \},& 
I_{\min}(\cF) &:= \{ \ih \in [r]\times [2] : \alpha_{\ih}^{-1} t_{\cF}(\ih)=  t_{\min}(\cF)\}.&
\end{align*}
Later we aim to take $\bF_\ih=\bF_\ih'\cup\bF_\ih^*$.
In order to achieve (B1)$_{\ref{lem: vector distribution}}$, we select $\cF$ such that $t_{\max}(\cF) - t_{\min}(\cF)$ is minimal.

To this end, choose a partition $\cF=\{ \bF^*_{\ih} : \ih\in [r]\times[2]\}$ of $\bF'_0$ such that 
\begin{itemize}
	\item $t_{\max}(\cF) - t_{\min}(\cF)$ is minimal, and subject to this,
	\item $|I_{\max}(\cF)| + |I_{\min}(\cF)|$ is minimal.
\end{itemize}
Let $$\bF^1:= \{\bq \in \bF'_0 : q_1 > q_2\} \text{ and } \bF^2 := \{\bq \in \bF'_0: q_2>q_1\}.$$ 
If  $t_{\max}(\cF) - t_{\min}(\cF)\leq r^5\Delta_1$ holds, 
then  clearly we may set  $\bF_{\ih}:= \bF'_{\ih}\cup \bF^*_{\ih}$ for each $\ih\in [r]\times[2]$ 
and we found the desired partition satisfying (B1)$_{\ref{lem: vector distribution}}$--(B3)$_{\ref{lem: vector distribution}}$.
If  $m_1 \leq r^3 \Delta_1$ holds,
then we obtain $t_{\max}(\cF) - t_{\min}(\cF)\leq r^2 m_1 \leq r^5\Delta_1$ as $\alpha_{\ih}^{-1}\leq r^2$ by (A3)$_{\ref{lem: vector distribution}}$.
We will show that at least one of these scenarios always applies and assume for a contradiction that 
\begin{align}\label{eq: ind assump}
m_1 > r^3 \Delta_1 \enspace \text{ and } \enspace t_{\max}(\cF) - t_{\min}(\cF) > r^5 \Delta_1.
\end{align}
Since $\cE(1,\ih)$ and~\eqref{eq: ind assump} hold, for each $\ih\in [r]\times[2]$, we obtain
 \begin{align}\label{eq: T1ih size} 
T^1_{\ih} \leq \frac{\beta m_1}{r}.
\end{align}
We choose two indices $\ih^* \in I_{\max}(\cF)$, $\ih_*\in I_{\min}(\cF)$.
Note that 
\begin{align}\label{eq: max i*h* value} 
t_{\cF}(\ih^*) = t_{\cF}(\ih^*)\sum_{\ih\in [r]\times[2]} \alpha_{\ih} =
  \alpha_{\ih^*} \sum_{\ih\in [r]\times[2]}  \alpha_{\ih} \alpha^{-1}_{\ih^*} t_{\cF}(\ih^*) 
\geq   \alpha_{\ih^*} \sum_{\ih\in [r]\times[2]}  \alpha_{\ih} \alpha^{-1}_{\ih} t_{\cF}(\ih)
 \geq   \alpha_{\ih^*} m_1.
\end{align}
Note that {\rm (A1)$_{\ref{lem: vector distribution}}$} implies that
\begin{align}\label{eq:q1q2}
	\sum_{\bq \in \bF^1} q_2 \leq \frac{ m_1 }{2r^2+1} \leq \frac{\beta m_1}{2r}\text{ and } \sum_{\bq \in \bF^2} q_1 \leq \frac{\beta m_1}{2r}.
\end{align}
Note that
\begin{eqnarray*}
t_{\cF}(\ih)  
&=& T^1_{\ih}+ \sum_{\bq\in \bF_{\ih}^*} q_1 + \sum_{\bq\in \bF_{\ihb}^*} q_2 
= 
T^1_{\ih}+ \sum_{\bq \in \bF_{\ih}^*\cap \bF^1} q_1 + \sum_{\bq \in \bF_{\ih}^*\cap \bF^2} q_1 + 
\sum_{\bq\in \bF_{\ihb}^*\cap \bF^1} q_2 + \sum_{\bq\in \bF_{\ihb}^*\cap \bF^2} q_2\\
&\stackrel{(\ref{eq: T1ih size}),(\ref{eq:q1q2})}{\leq} & 
 \sum_{\bq \in \bF_{\ih}^*\cap \bF^1} q_1 + \sum_{\bq \in \bF_{\ihb}^*\cap \bF^2} q_2 + \frac{2\beta m_1}{r}.
\end{eqnarray*}
Thus, if $t_{\cF}(\ih) \geq  {3\beta m_1}/{r}$, 
then
\begin{align}\label{eq: enough to extract Delta}
 \sum_{\bq\in \bF_{\ih}^*\cap \bF^1} q_1 + \sum_{\bq \in \bF_{\ihb}^*\cap \bF^2} q_2 > \frac{\beta m_1}{r} \stackrel{(\ref{eq: ind assump})}{\geq} r \Delta_1.
\end{align}
For each $\ih\in [r]\times [2]$ with $t_{\cF}(\ih) \geq  {3\beta m_1}/{r}$, 
we choose a set $\bX_{\ih} \subseteq (\bF^*_{\ih}\cap \bF^1) \cup (\bF^*_{\ihb}\cap \bF^2)$ such that 
$$2\Delta_1\leq \sum_{\bq \in \bX_{\ih}} (q_1+q_2) \leq 4\Delta_1.$$
Indeed, this is possible by  (A2)$_{\ref{lem: vector distribution}}$ and  \eqref{eq: enough to extract Delta}. 
Note that by \eqref{eq: max i*h* value} and the fact that $\beta \ll 1/t$ and $\alpha_{\ih} \geq 1/(2tr)$, 
we have $t_{\cF}(\ih^*)\geq {3\beta m_1}/{r}$.
Now we consider the following three cases.
In each case we construct a partition that contradicts our choice of $\cF$
by reallocating $\bX_{\ih^*},\bX_{\ihb_*}$.
\newline

\noindent {\bf CASE A.} $\alpha_{\ihb^*}^{-1} t_\cF(\ihb^*) < t_{\max}(\cF)  -  r^2\Delta_1 $.
\medskip

\noindent
In this case we use that $\alpha_{\ihb^*}^{-1} t_\cF(\ihb^*)$ is not too large and reallocate $\bX_{\ih^*}$ accordingly.
We define
$$ \bF^{\#}_{\ih^*} := \bF^*_{\ih^*} \triangle \bX_{\ih^*}, \enspace \bF^{\#}_{\ihb^*} := \bF^*_{\ihb^*} \triangle \bX_{\ih^*}, \text{ and }
\bF^{\#}_{\ih} := \bF^*_{\ih} \text{ for each } \ih \in [r]\times [2] \setminus \{ \ih^*, \ihb^*\} $$
and let $\cF^{\#}:=\{\bF^{\#}_{\ih} \}_{\ih \in [r]\times [2]}$. 
Then, since $1/(2tr) \leq \alpha_{\ih^*} , \alpha_{\ihb^*}\leq 2t/r$, we have
\begin{align*}
t_{\min}(\cF)
\stackrel{(\ref{eq: ind assump})}{<}\alpha_{\ih^*}^{-1} (t_{\cF}(\ih^*) - 4\Delta_1 )\leq \alpha_{\ih^*}^{-1} t_{\cF^{\#}}(\ih^*) 
& \stackrel{\text{(A1)$_{\ref{lem: vector distribution}}$}}{\leq} \alpha_{\ih^*}^{-1} \left(t_{\cF}(\ih^*) - 2\Delta_1 + \frac{4\Delta_1}{2r^2}\right)< t_{\max}(\cF),\\
t_{\min}(\cF) \leq  \alpha_{\ihb^*}^{-1} \left(t_{\cF}(\ihb^*)+2\Delta_1 - \frac{4\Delta_1}{2r^2}\right)
\leq \alpha_{\ihb^*}^{-1}  t_{\cF^{\#}}(\ihb^*) 
&  \leq \alpha_{\ihb^*}^{-1} (t_{\cF}(\ih^*) +4\Delta_1)  < t_{\max}(\cF).
\end{align*}
Since $\cF$ and $\cF^{\#}$ coincides on $[r]\times [2]\setminus \{\ih^*, \ihb^*\}$, 
either 
$t_{\max}(\cF^{\#}) -t_{\min}(\cF^{\#})  
 < t_{\max}(\cF) -t_{\min}(\cF)$  or $t_{\max}(\cF^{\#}) -t_{\min}(\cF^{\#})  = t_{\max}(\cF) -t_{\min}(\cF) $
as well as $I_{\max}(\cF^{\#}) \subseteq I_{\max}(\cF) \setminus \{ \ih^*\}$ and $I_{\min}(\cF^{\#})\sub I_{\min}(\cF)$. 
In either way, we obtain a contradiction to the choice of $\cF$. \newline

\noindent {\bf CASE B.} $\alpha_{\ihb_*}^{-1} t_\cF(\ihb_*) > t_{\max}(\cF)  -  r^2\Delta_1.$
\medskip

\noindent
Observe that 
\begin{eqnarray*}
t_{\cF}(\ihb_*) & >& \alpha_{\ihb_*} ( \alpha_{\ih^*}^{-1} t_{\cF}(\ih^*) - r^2\Delta_1) \geq \frac{1}{4t^2} t_{\cF}(\ih^*) -  2tr \Delta_1 \\
&\stackrel{\eqref{eq: max i*h* value}}{\geq}& \frac{\alpha_{\ih^*}m_1}{4t^2} - 2tr \Delta_1 \stackrel{\eqref{eq: ind assump}}{\geq}  \frac{3\beta m_1}{r}.
\end{eqnarray*}
Thus $X_{\ihb_*}$ is defined.
Then similarly as in the Case 1, we can construct $\cF^{\#}$ as
$$ \bF^{\#}_{\ih_*} := \bF^*_{\ih_*} \triangle \bX_{\ihb_*},\enspace  \bF^{\#}_{\ihb_*} := \bF^*_{\ihb_*} \triangle \bX_{\ihb_*}, \text{ and }
\bF^{\#}_{\ih} := \bF^*_{\ih} \text{ for each } \ih \in [r]\times [2] \setminus \{ \ih_*, \ihb_*\}$$
and let $\cF^{\#}:=\{\bF^{\#}_{\ih} \}_{\ih \in [r]\times [2]}$.
Similarly as in the Case 1, we obtain either 
$t_{\max}(\cF^{\#}) -t_{\min}(\cF^{\#})  
 < t_{\max}(\cF) -t_{\min}(\cF)$  or $t_{\max}(\cF^{\#}) -t_{\min}(\cF^{\#})  = t_{\max}(\cF) -t_{\min}(\cF) $
and $I_{\max}(\cF^{\#}) \subseteq I_{\max}(\cF)$ while $I_{\min}(\cF^{\#})=I_{\min}(\cF) \setminus \{ \ih_*\}$, we derive a contradiction. \newline

\noindent {\bf CASE C.} $\alpha_{\ihb^*}^{-1} t_\cF(\ihb^*) \geq t_{\max}(\cF)  -  r^2\Delta_1 $ and $\alpha_{\ihb_*}^{-1} t_\cF(\ihb_*) \leq t_{\max}(\cF)  -  r^2\Delta_1.$
\medskip

\noindent
Note that in this case, by \eqref{eq: ind assump} we have $\ih^*\notin \{\ih_*,\ihb_*\}$.
We let 
$$ \bF^{\#}_{\ih^*} := \bF^*_{\ih^*} \setminus \bX_{\ih^*},\enspace  \bF^{\#}_{\ihb^*} := \bF^*_{\ihb^*} \setminus \bX_{\ih^*}, \enspace
 \bF^{\#}_{\ih_*} := \bF^*_{\ih_*} \cup (\bX_{\ih^*}\cap \bF^1),\enspace  \bF^{\#}_{\ihb_*} := \bF^*_{\ihb_*} \cup (\bX_{\ih^*}\cap \bF^2) , \text{ and }$$
$$\bF^{\#}_{\ih} := \bF^*_{\ih} \text{ for each } \ih \in [r]\times [2] \setminus \{\ih^*, \ihb^*, \ih_*, \ihb_*\}.$$
Let $\cF^{\#}:=\{\bF^{\#}_{\ih} \}_{\ih \in [r]\times [2]}$.
Then, since $1/(2tr) \leq \alpha_{\ih^*},\alpha_{\ihb^*},\alpha_{\ih_*},\alpha_{\ihb_*}\leq 2t/r$, we have
\begin{align*}
t_{\min}(\cF)\stackrel{\eqref{eq: ind assump}}{<} \alpha_{\ih^*}^{-1}( t_{\cF}(\ih^*) - 4\Delta_1 ) \leq
\alpha_{\ih^*}^{-1} t_{\cF^{\#}}(\ih^*) & < \alpha_{\ih^*}^{-1} t_{\cF}(\ih^*)\leq  t_{\max}(\cF), \\
t_{\min} (\cF) \stackrel{\eqref{eq: ind assump}}{<} t_{\max}(\cF) - r^2 \Delta_1 - 8tr\Delta_1 \leq 
\alpha_{\ihb^*}^{-1} t_{\cF^{\#}}(\ihb^*) & \leq \alpha_{\ihb^*}^{-1} t_{\cF}(\ihb^*)  \leq  t_{\max}(\cF),\\
t_{\min}(\cF) = \alpha_{\ih_*}^{-1} t_{\cF}(\ih_*) 
<\alpha_{\ih_*}^{-1} t_{\cF^{\#}}(\ih_*)  & \leq \alpha_{\ih_*}^{-1} (t_{\cF}(\ih_*) +4\Delta_1 )  \stackrel{\eqref{eq: ind assump}}{<}  t_{\max}(\cF),\\
t_{\min}(\cF) \leq \alpha_{\ihb_*}^{-1} t_{\cF}(\ihb_*)
<\alpha_{\ihb_*}^{-1} t_{\cF^{\#}}(\ihb_*)  & \leq \alpha_{\ihb_*}^{-1} (t_{\cF}(\ihb_*) +4\Delta_1 ) <  t_{\max}(\cF).
\end{align*}

Since $\cF$ and $\cF^{\#}$ coincides on $[r]\times [2]\setminus \{\ih^*, \ihb^*, \ih_*, \ihb_*\}$, 
either
$t_{\max}(\cF^{\#}) -t_{\min}(\cF^{\#})  
 < t_{\max}(\cF) -t_{\min}(\cF)$  or $t_{\max}(\cF^{\#}) -t_{\min}(\cF^{\#})  = t_{\max}(\cF) -t_{\min}(\cF) $
and $I_{\max}(\cF^{\#}) \subseteq I_{\max}(\cF) \setminus \{ \ih^*\}$ and $I_{\min}(\cF^{\#})\sub I_{\min}(\cF^{\#})\setminus\{\ih_*\}$. 
In both cases, we obtain the final contradiction.
\end{proof}



\section{Preparation}\label{sec:prep}

We start this section by setting up some terminology, constants and notation for the proof of Theorem~\ref{thm: main}.
For given $k$ and $\alpha$, 
we choose constants so that 
\begin{align}\label{eq: hierarchy}
0< \frac{1}{n_0} \ll \frac{1}{M_*} 
\ll \frac{1}{C} \ll \frac{1}{C'}
\ll \epsilon \ll \epsilon_1 \ll \dots \ll \epsilon_{k} \ll \mu \ll \eta \ll d,\frac{1}{t} \ll \frac{1}{k}, \alpha.
\end{align}
Let $M:= 10 k M_*^7$.
The roles of $C,C',\epsilon,d,t$ are explained in \ref{G1}--\ref{G3}
and $\epsilon_1,\ldots,\epsilon_k$ are `error' parameters for the $k$ steps of our embedding process.
Recall that we are given an $n$-vertex graph $G$ with minimum degree at least $\alpha n$ and $n\geq n_0$, a binomial random graph $R \in \bG(n,M p)$, and an $n$-vertex tree $T$ with $\Delta(T)\leq\Delta$.
Note that in the setting of Theorem~\ref{thm: main} we assume that
\begin{align}\label{eq: def Delta p}
n^{1/(k+1)} \leq \Delta < \min\Big\{ n^{1/k}, \frac{n}{M\log{n}} \Big\}
\enspace \text{ and } \enspace p:= \max\big\{ n^{-k/(k+1)}, \Delta^{k+1} n^{-2}\big\}.
 \end{align}
Let $p' := M_*^6 p$.
It will be convenient to consider mutually independent random graphs $R_1,\dots, R_{k+3} \subseteq R$ such that $R_i \in \bG(n, M_* p')$. 
Such random graphs exist by standard probabilistic arguments.

First, we apply Lemma~\ref{lem: partition} to obtain a partition  $\{ V_\ih \}_{\ih \in  [r]\times [2]}$ of $V(G)$ and a subgraph $G'\subseteq G$ satisfying the following for all $i\in [r]$:
\begin{enumerate}[label={\rm(G\arabic*)}]
\item\label{G1} $G'[V_{\ih}, V_{\ihb}]$ is $(\epsilon,d)$-super-regular,
\item\label{G2} ${n}/{tr} \leq |V_{\ih}|  \leq {t n}/{r}$, and
\item\label{G3} $C' \leq r \leq C$.
\end{enumerate}
Later we will embed $T$ into $G'\cup\bigcup_{i=1}^{k+3} R_i$.
For the remainder of the section, we focus on finding an appropriate edge-decomposition $F_1,\dots, F_{k+1}, F'_1,\dots, F'_{k}, L_1, \last$ of $T$.
The tree $T$ will be a fixed tree for the rest of the paper
and we denote by $\lea$ the set of leaves of $T$.
Choose a leaf $x_1$ of $T$ as the root of $T$ and consider a breath-first-search ordering $x_1,\dots, x_{n}$ of $V(T)$. From now on, for any subtree $T'\subseteq T$, we always consider $T'$ as a rooted tree $(T',x_i)$ where $i= \min\{j : x_j \in V(T')\}$.

As explained in Section~\ref{sec: outline}, we need to decide whether we use Lemma~\ref{lem: p matching}, Lemma~\ref{lem: light leaves embedding} or Lemma~\ref{lem: bare paths embedding} at the last step to finish the embedding. 
We introduce some more terminology.
For a leaf $x$ of $T$ and its ancestor $y$, we say that the vertex $x$ and the edge $xy$ are \emph{heavy} if $|D_T(y )\cap \lea| \geq {np'}/\log{n}$ and \emph{light} otherwise. 
Let $\li$ be the set of all light leaves of $T$ and let $\he$ be the set of all heavy leaves of $T$.
Thus, if a vertex has one heavy leaf as a child, then it has at least $np'/\log n$ heavy leaves as children.

If $|\he| < 4\eta n$ holds, we will use Lemma~\ref{lem: light leaves embedding} or Lemma~\ref{lem: bare paths embedding} to finish the embedding.
This case is much simpler than the remaining one. 
We will deduce it from the case $|\he| \geq 4\eta n$ in Section~\ref{sec: few heavy leaves}.

We assume now that 
\begin{align}\label{eq: enough heavy leafs}
|\he| \geq 4\eta n.
\end{align}

We aim to decompose $E(T)$ into forests $F_1,\dots, F_{k+1}, F'_{1},\dots, F'_{k}, L_1, \last$ and embed the edges in $F_i$ and $L_1$ into $R$ by using Lemma~\ref{lem: random embedding}, and the edges in $F'_i$ and $\last$ into $G'$ by using Lemma~\ref{lem: random matching behaves random}. 
We embed each forest one by one in the following order $F_1, F'_1, F_2, F'_2, \dots, F'_k, F_{k+1}, L_1, \last$.
We also ensure that the roots of a forest in this list belong to the forests that precedes it in this list.
In order to use Lemma~\ref{lem: random embedding}, we want $F_1,\dots, F_{k+1}, L_{1}$ to have maximum degree at most $O(np')$.
Next, we inductively define an edge colouring $c\colon E(T)\to [2]$ and two functions $h\colon V(T)\rightarrow \mathbb{N}_0$ and $h'\colon V(T)\rightarrow \mathbb{N}_0$ and use them to find the decomposition of $E(T)$ promised above.

We first sketch the ideas behind our approach.
We define an edge colouring $c$ such that every path between $x_1$ and a leaf of $T$ contains at most $k$ edges of colour $2$ (see Claim~\ref{cl: height and h}).  
We define $E(F_i)$ as the collection of edges $e$ of colour $1$ such that the path between $x_1$ and $e$ contains $i-1$ edges of colour $2$
and define $E(F'_i)$ as the collection of edges of colour $2$ such that the path between $x_1$ and $e$ (including $e$) contains $i$ edges  of colour $2$. 
This definition ensures that all the roots of $F'_i$ belong to `previous' forests.
For each $x\in V(T)$, $h(x)$ measures the height of the vertex $x$ in the component induced by edges colour $2$ (see Claim~\ref{cl: height and h}), 
and $h'(x)$  roughly measures the maximum number of heavy-leaf children of a descendant $y$ of $x$ in a component induced by edges of colour $2$.
At the end of the section, we prove several properties of this decomposition for later use (see Claim~\ref{cl: properties2} and~\ref{cl: properties}).


Assume that for some $i\in [n]$ we already have defined $h(x_j), h'(x_j)$ and $c(e)$ for all $j\in [n]\sm [i]$ and $e\in S_T(x_j)$. 
Now we will define $h(x_i), h'(x_i)$ and $c(e)$ for all $e\in S_T(x_i)$. If $x_i$ is a leaf of $T$, 
then we simply define 
\begin{align}\label{eq: leaf h h' def}
h(x_i):=0 \text{ and }h'(x_i):= 0.
\end{align}
Assume now that $x_i$ is not a leaf. 
Then $h(x), h'(x)$ are defined for all vertices $x\in D_{T}(x_i)$.
For each $ \ell\geq 0$, we let  
$H_i^\ell:= \{ x\in D_{T}(x_i) : h(x) = \ell \}$
and
\begin{align}\label{eq: h def}
h(x_i)&:=   \left\{ \begin{array}{ll}
0 & \text{ if } |\lea\cap D_T(x_i)| \leq np' \text{ and }|H_i^\ell| \leq 10 np' \text{ for all }\ell \geq 0, \\ 
1 & \text{ if } |\lea\cap D_T(x_i)| > np' \text{ and } |H_i^\ell| \leq 10 np' \text{ for all }\ell \geq1, \\
\max\{ \ell+1 :  |H_i^\ell| > 10np' \} & \text{ otherwise. }
\end{array}\right.
\end{align}
Observe that $h(x_i)>0$ if and only if $x_i$ has either more than $np'$ leaf-children or at least $10np'$ children of a particular $h$-value.
If this applies, we must colour some edges in $S_{T}(x_i)$ with colour $2$ as otherwise the degree of $F_i$ becomes too large.

Next we define $c$ and $h'$.
If $h(x_i)=0$, then we simply define for each $x \in D_{T}(x_i)$
\begin{align}\label{eq: c1 h' hdT}
 c(x_i x ):= 1 \text{ and } h'(x_i) := |\he\cap D_T(x_i)|.
 \end{align}
If $h(x_i)>0$, then we want to colour the edges $x_iy\in S_T(x_i)$
with colour $1$
if either $|T(y)|$ is large, $h'(y)$ is large, or $|D_T(y)\cap \he|$ is large.
Exactly for this purpose, we define sets $A_i,B_i,B_i'$ below.
To this end, let
$\{y_1,\dots, y_{s}\}:= \{y \in H_i^{h(x_i)-1} : h'(u)>0 \}$ with $h'(y_1)\geq  \dots \geq h'(y_s)$, and 
let $z_1,\dots, z_{ s' }$ be ordering of $D_T(x_i)\setminus \lea$ such that $|\he\cap D_T(z_j)| \geq |\he\cap D_{T}(z_{j'})|$ for $1\leq j < j' \leq s'$. 
Let 
\begin{align}\label{eq: Ai Bi def}
A_i&:= \{ x \in D_{T}(x_i) \colon  |T(x) | > (np')^{-1} |T(x_i)| \}, 
&B_i&:= \left\{ \begin{array}{ll}
\{ y_{i'} \colon 1 \leq i' \leq 2np' \} & \text{ if } s> 5np' \\
\{ y_{i'}\colon 1\leq i' \leq s\}    & \text{ if } s\leq 5np',\\
\end{array}\right. \nonumber \\
B'_i&:=  \{z_{i'}\colon 1\leq i'\leq \min\{s',np'\} \}. & &
\end{align}
Now we define $h'(x_i)$ and the edge colouring $c$ on $S_T(x_i)$.
For each $x\in D_T(x_i)$, let
\begin{align}\label{eq: c def}
c(x_i x)&:= \left\{ \begin{array}{ll}
1 & \text{if } x\in \lea \text{ and } |\lea\cap D_T(x_i)| \leq np', \\
1 & \text{if } x\in A_i \cup B_i \cup B'_i \cup \bigcup_{\ell\colon |H_i^{\ell}| \leq  10np'  } (H_i^\ell\setminus \lea), \\
2 & \text{otherwise}
\end{array}\right.
\end{align}
and we also define 
\begin{align}\label{eq: h' def}
h'(x_i) := \max\{ h'(x)\colon x\in H_i^{h(x_i)-1}, c(x_i x) = 2\}.
\end{align}
By repeating the above for each $i=n,\dots,1$, we obtain edge colouring $c$, functions $h$ and $h'$.  
We continue with observations regarding $c, h$ and $h'$ for later use.
By the definition of $h'$ and $\he$, if a vertex has one heavy leaf child, then it has at least $np'/\log n$ heavy children. Thus for any $x\in V(T)$, we have
\begin{align}\label{eq: h' value range}
h'(x) \in \{0\} \cup \Big[ \frac{np'}{\log{n}} , n\Big].
\end{align}
Since $|T(x_i)| \geq \sum_{x\in A_i} |T(x)| > |A_i| (np')^{-1}|T(x_i)|$, we conclude that
\begin{align}\label{eq: Ai size}
|A_i| < np'.
\end{align}
Moreover, if $c(x_ix)=2$ for some $x\in D_T(x_i)$, then $x \notin A_i$. 
Thus we have
\begin{equation}\label{eq: under red}
\begin{minipage}[c]{0.8\textwidth}\em
if $x\in D_{T}(x_i)$ with $c(x_i x)=2$, then $|T(x)| < (np')^{-1} |T(x_i)|.$ 
\end{minipage}
\end{equation}
For each $i\in [2]$, we define
$$F^i:= \{ e\in T: c(e)=i\}.$$
Let $C^1$ be the component of $F^1$ which contains the root $x_1$. We collect further properties of $h,h'$ and $c$.
\begin{claim}\label{cl: h-value}
Suppose $y\in D_{T}(x_i)$ for some $i\in [n]$. 
Suppose $h(x_i)= \ell>0$ and $c(x_iy)=2$, 
then $h(y)<\ell$. 
Moreover, $D_T(x_i)$ contains at least $np'$ vertices $y$ such that $c(x_iy)=2$ and $h(y) = \ell-1$.
\end{claim}
\begin{proof}
By \eqref{eq: h def} and \eqref{eq: c def},
we have $c(x_i y)=2$ only if $h(y) < h(x_i)=\ell$.

For the second part of the claim, we first consider the case that $h(x_i)=1$ and $|\lea\cap D_T(x_i)|>np'$.
By \eqref{eq: leaf h h' def} and \eqref{eq: Ai Bi def}, we have $\lea \cap (A_i\cup B_i \cup B_i') =\emptyset$. Thus \eqref{eq: leaf h h' def} and \eqref{eq: c def} imply that the edges $x_i y'$ for all $y'\in\lea\cap D_T(x_i)$ are coloured $2$ and $h(y')=h(x_i)-1$, and the claim holds.

If $|\lea\cap D_T(x_i)|\leq np'$ or $h(x_i)>1$, then \eqref{eq: h def} implies that $|H_i^{\ell-1}|> 10np'$.  Hence
$$|H_i^{\ell-1}\setminus (A_i\cup B_i\cup B'_i \cup \lea)| \stackrel{\eqref{eq: Ai Bi def},\eqref{eq: Ai size}}{\geq} 10np'-(np'+5np'+np'+ np') \geq np'.$$ 
Since \eqref{eq: c def} implies $c(x_i y')=2$ for each $y'\in H_i^{\ell-1}\setminus (A_i\cup B_i\cup B'_i\cup \lea )$, there are at least $np'$ vertices $y'$ such that $c(xy')=2$ and $h(y')=\ell-1$.
\end{proof}


\begin{claim}\label{cl: h'-value}
Suppose $i\in [n]$, $h'(x_i) >0$ and $h(x_i)=\ell>0$.
Then $D_{T}(x_i)$ contains at least $np'$ vertices $y$ with $c(x_iy)=1$ as well as $h'(y)\geq h'(x_i)$ and $h(y) = \ell-1$.
Moreover, $D_T(x_i)$ contains at least $np'$ vertices $y$ with $c(x_iy)=2$ as well as $h'(y)\geq {np'}/{\log{n}}$ and $h(y)=\ell-1$.
\end{claim}
\begin{proof}
Let $\{y_1,\dots, y_s\}:= \{ y \in H_i^{\ell-1}\colon h'(y )>0\}$ with $h'(y_1)\geq \dots \geq h'(y_s)$. 
By \eqref{eq: Ai Bi def}, \eqref{eq: c def} and \eqref{eq: h' def},
the assumption that $h'(x)>0$ implies that $B_i\subsetneq \{y_1,\dots, y_s\}$ and so $s> 5np'$. 
Hence
\begin{eqnarray*}
|\{y_{1},\dots, y_{s}\} \setminus (A_i \cup B_i \cup B'_i) |
=|\{y_{2np'+1},\dots, y_{s}\} \setminus (A_i \cup B'_i) | 
\geq 3np' - | (A_i \cup B'_i) | \stackrel{\eqref{eq: Ai Bi def},\eqref{eq: Ai size} }{\geq} np'.
\end{eqnarray*}
Note that by \eqref{eq: Ai Bi def} and \eqref{eq: h' def}, for any $j'\in [2np']$, we have 
$$c(xy_{j'})=1 \text{ and }
h'(x_i) = \max\{ h'(y_j): 2np<j\leq s, y_{j} \notin A_i \cup B_i' \} \leq h'(y_{j'}).$$
Thus, there are at least $2np'$ vertices $y$ such that $c(x_iy)=1$, $h'(y)\geq h'(x_i)$ and $h(y) = \ell-1$.
Furthermore,
for any $y' \in \{y_{2np'+1},\dots, y_{s}\} \setminus (A_i\cup B_i') $, 
we have $c(xy')=2$, $h'(y')>0$ and $h(y')=\ell-1$. 
This with \eqref{eq: h' value range} implies $h'(y') \geq {np'}/{\log{n}}$ and so proves the claim.
\end{proof}

\begin{claim}\label{cl: height and h}
Suppose $T'\in C(F^2)$.
Then
for any vertex $y \in V(T')$, 
the tree $T'(y)$ has height $h(y)$ and $|D_{T'}^{h(y)}(y) | \geq (np')^{h(y)}$. 
In particular,
the height of $T'$ equals $\ell=h(r(T'))$ 
and $|D_{T'}^{\ell}(r(T'))| \geq (np')^{ \ell }$.  
Moreover, for any $y \in V(T)$, we have $0\leq h(y) \leq k$.
\end{claim}
\begin{proof}
Suppose $y\in V(T')$.
We proceed by induction on $\ell' :=h(y)$. 
If $\ell' =0$, 
then by \eqref{eq: h def} and \eqref{eq: c def}, for every vertex $z\in D_T(y)$, we have $c(yz)=1$. 
Thus $T'(y)$ has height $0$ and $|D_{T'}^{0}(y)|= |\{ y\}|= 1 = (np')^{0}$. 
Hence the statement holds for $\ell' =0$.

Assume that the claim holds for $\ell'-1 \geq 0$. 
Suppose $h(y)=\ell'$. 
Then Claim~\ref{cl: h-value}  ensures that any $z \in D_{T'}(y)$ satisfies $h(z)\leq \ell'-1$. 
This with the induction hypothesis shows that height of $T'(y)$ is at most $\ell'$.
Moreover, by Claim~\ref{cl: h-value}, there are at least $np'$ vertices $y_1,\dots, y_{np'}\in D_{T'}(y)$ which satisfy $h(y_i)= \ell'-1$ for $i\in [np']$. 
Thus by the induction hypothesis, $T'$ has height at least $\ell'$, and 
$|D_{T'}^{\ell'}(y)| \geq \sum_{i=1}^{np'} |D_{T'}^{\ell'-1}(y_i)|\geq (np')(np')^{\ell'-1} = (np')^{\ell'}$. 
This proves the first part of Claim~\ref{cl: height and h}.
Moreover, 
for any $y'\in V(T)$,
we conclude that
$$n^{h(y')/(k+1)} \stackrel{\eqref{eq: def Delta p}}{\leq} (np')^{h(y')}\leq |D_{T'}^{h(y') }(y') |    \leq n-1.$$
Thus $0\leq h(y')\leq k$.
\end{proof}

If there are many heavy leaves attached to a component of $F^2$, 
then the $h'$-value of a vertex in $V(F^2)$ is high.
Claim~\ref{cl: h'-value} shows that there are also many heavy leaves attached to a nontrivial component of $F^1$.
To better describe this phenomena later in \ref{F12}, we define the following vertex set $B^*$ and star-forests $L_1, \widehat{L}_2, \widehat{L}_3$ (recall that $C^1$ is the component in $F^1$ which contains $x_1$):
\begin{align}
B^*&:= \bigcup_{T'\in C(F^2) } D_{T'}^{k} (r(T')), & L_1 &:= \bigcup_{x\in B^*} \{ xy : y \in D_T(x) \cap \he\}, \label{eq: def B* L1}\\
\widehat{L}_2 &:= \bigcup_{x\in V(C^1)}\{ xy : y \in D_T(x)\cap \he \}, \text{ and} &
\widehat{L}_3 &:= \bigcup_{x \in V(T)\setminus (B^*\cup C^1)} \{x y : y \in D_T(x) \cap \he \}. \label{eq: def hat L2 L3}
\end{align}
Note that $B^*$ and $V(C^1)$ are disjoint, as for any $x\in B^*$ and $y \in V(C^1)$, the path
$P_T(x_1,x)$ contains exactly $k>0$ edges of colour $2$ and $P_T(x_1,y)$ contains no edge of colour $2$. 
As every edge in $L_1, \widehat{L}_2$ and $\widehat{L}_3$ is incident to a leaf, 
the star-forests $L_1, \widehat{L}_2, \widehat{L}_3$ are pairwise vertex-disjoint.

Observe that by Claim~\ref{cl: height and h}, all edges $e\in L_1$ satisfy $c(e)=1$.
For each $i\in \{2,3\}$, we partition the edges in $\widehat{L}_i$ into two sets $L_i$ and $L'_i$, 
in such a way that the following holds for all $x\in V(T)\setminus B^*$:
\begin{align}\label{eq: L3 prop}
S_{L_i}(x)\cap F^1 = S_{\widehat{L}_i}(x) \cap F^1 \enspace \text{and} \enspace
\frac{1}{3} |S_{\widehat{L}_i}(x) \cap F^2| \leq |S_{L_i}(x)\cap F^2| \leq \frac{1}{2}|S_{\widehat{L}_i}(x) \cap F^2|.
\end{align}
This is possible as \eqref{eq: h' value range} implies 
that either  $|S_{\widehat{L}_i}(x)| = 0$ or $|S_{\widehat{L}_i}(x)| \geq np'/\log{n} \geq 2$,\COMMENT{As long as the number of heavy leaf-edge incident to $x$ with colour $2$ is not 1, we can do the partition as above.} 
and by \eqref{eq: c def} we have $c(e)=c(e')$ for all $e,e'\in S_{\widehat{L}_i}(x)$.
Then it is easy to see that for each $i\in \{2,3\}$ and $x\in V(T)\setminus B^*$, 
we have $d_{L'_i}(x)/2 \leq d_{L_i}(x)$ and $d_{\widehat{L}_i}(x)/3 \leq d_{L_i}(x)$. 
Thus we have
\begin{align}\label{eq: L4 size}
|L_i|\geq \frac{1}{3}|\widehat{L}_i|.
\end{align}
As $L_1, \widehat{L}_2, \widehat{L}_3$ are pairwise vertex-disjoint star-forests and $L_i \subseteq \widehat{L}_i$ for each $i\in \{2,3\}$, 
the star-forests $L_1, L_2, L_3$ are vertex-disjoint 
such that $ L( L_1\cup L_2 \cup L_3) \subseteq \he.$

Note that as $C^1$ will be embedded into the random graph $R$, 
so the images of the centre-vertices of $L_2$ will possess very strong `quasi-random' properties.
Thus, if $|L_2|$ is sufficiently large, the images centre vertices of $L_2$ carry enough `quasi-randomness' to apply later Lemma~\ref{lem: p matching} 
for the sake of embedding $L_2$ at the end of algorithm. 
Indeed, by using \eqref{eq: enough heavy leafs} with property \ref{F12} in Claim~\ref{cl: properties}, 
we can guarantee that $|L_2|$ is never too small. 
To show this, we use some relations among $L_1, L_2$ and $L_3$. 
The only purpose we defined $L_2',L_3'$ is to avoid making the trees in $F'_i$ too `unsymmetrical'.

Recall that, as stated in~\eqref{eq: enough heavy leafs}, we assume that $T$ has at least $4\eta n$ heavy leaves. 
We will adapt our analysis of $T$ according to the following two cases.
\newline

\noindent{\bf CASE 1.} $|L_1| \geq \eta n$.

\noindent
In this case, we let $\last:= L_2$.  \newline

\noindent{\bf CASE 2.} $|L_1| < \eta n$.

\noindent
In this case, we let $\last:= L_2 \cup L_3$.  \newline

Observe that in Case 2, \eqref{eq: enough heavy leafs} and \eqref{eq: L4 size} imply that $|L_2\cup L_3|\geq \eta n$.

Let $$\LA^*:= L(\last)$$ be the set of all leaf-vertices of $\last$.
The vertices in $\LA^*$ (and so the edges on $\last$) will be embedded the final round of our embedding algorithm.

Next, we partition $T-(L_1\cup \last)$ into edge-disjoint forests $F_0,\ldots,F_{k+1}$ and $F_0',\ldots,F_{k}'$ based on the colouring $c$. 
For every edge $xx' \in E(T)$ with $x'\in D_T(x)$, let
$$f(xx'):= |\{e \in E( P_T(x_1,x') ): c(e)=2 \}|.$$
We claim that for any $e=xx' \in E(T)$, we have
\begin{align}\label{eq: f bound k}
f(e)\leq k.
\end{align}
Indeed,
suppose $P_T(x_1,x')$ contains exactly $s=f(xx')$ edges which are coloured by $2$.
Then~\eqref{eq: under red} implies that  $1\leq |T(x')| < (np')^{-s} |T(x_1)| \leq n^{1-s} p'^{-s}$.
As \eqref{eq: def Delta p} implies $np' \geq n^{1/(k+1)}$, we obtain \eqref{eq: f bound k}.
For each $i\in [k+1]$ and $i'\in [k]$, 
let 
\begin{align}\label{eq: def Fi F'i}
F_i := \{ e \in E(F^1)\setminus (L_1\cup \last) \colon f(e) =i-1\} \enspace \text{and} \enspace
F'_{i'} := \{ e \in E(F^2)\setminus (L_1\cup \last) \colon f(e) =i'\}.
\end{align}
By \eqref{eq: f bound k}, $F_1,\dots, F_{k+1}, F'_1,\dots, F'_k, \last, L_1$ form a partition of $E(T)$.
 It will be convenient to define $V(F_0):=\{ x_1\}$ with $E(F_0)=\emptyset$ and $R(F_0):=\emptyset$ as well as $F_0':=\es$.
Moreover,
let 
$$F' := \bigcup_{i\in [k]} F'_i  \enspace\text{ and }  \enspace  F^{\#}:= F'\cup \last.$$
Observe that $F^2\sub F^\#$ and $F^\#$ may contains some edges in $L_2\cup L_3$ of colour $1$.
We also want to classify the components of $F'$ and $F^{\#}$ according to where their roots attach.
For each $i\in [k+1]$, let
$H_i(F')$ be the components of $F'$ whose root lies in $V(F_i)$
and let $H_i(F^\#)$ be similarly defined; that is,
\begin{align*}
H_i(F'):=\{ T'\in C(F')\colon r(T') \in V(F_i)\} \enspace\text{and}\enspace H_i(F^{\#}):=\{ T'\in C(F^{\#})\colon r(T') \in V(F_i)\}.
\end{align*}

For a vertex $y\in \Lambda^*$, if $x:=a_T(y)$ belongs to $V(F_i)\setminus R(F_i)$, then it is easier to use Lemma~\ref{lem: p matching} as $x$ is randomly embedded using the random graph $R$ (as opposed to the case when $x a_T(x)$ is embedded into an edge of $G'$). 
We distinguish such vertices as $\LA_i$ as follows and we partition $\Lambda^*$ into $\LA_1^*,\dots,\LA_{k+1}^*$ as follows.
For each $i\in [k]$, let
\begin{align*}
	\LA^*_i :=  \LA^* \cap \bigcup_{T'\in H_i(F^{\#}) } L(T') \enspace \text{ and }  \enspace 
	\LA_i:= \LA^* \cap  \bigcup_{y\in V(F_i)\setminus R(F_i) } D_{\last}(y) 
\end{align*}

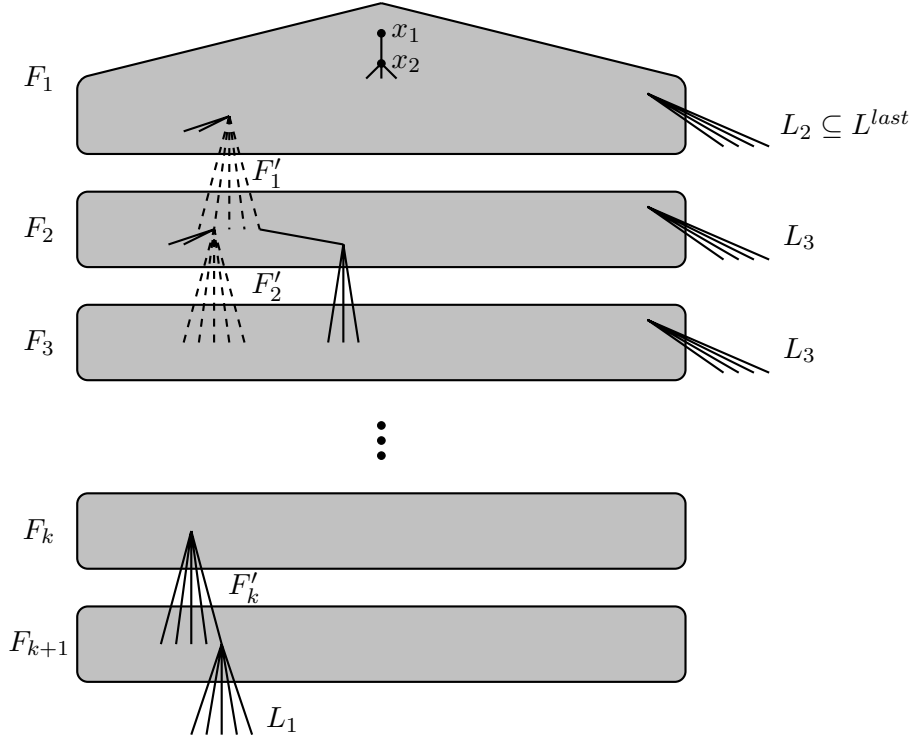
\begin{figure}[t]
\centering
\begin{tikzpicture}

\draw[rounded corners,thick,fill opacity=0.6,fill=black!40]
(0,10)--(4,9)--(4,8)--(-4,8)--(-4,9)--(0,10);

\draw[rounded corners,thick,fill opacity=0.6,fill=black!40] (-4,6.5) rectangle (4,7.5);

\draw[rounded corners,thick,fill opacity=0.6,fill=black!40] (-4,5.0) rectangle (4,6.0);

\draw[rounded corners,thick,fill opacity=0.6,fill=black!40] (-4,2.5) rectangle (4,3.5);

\draw[rounded corners,thick,fill opacity=0.6,fill=black!40] (-4,1.0) rectangle (4,2.0);

\node (v1) at (-4.5,9) {$F_1$};
\node (v1) at (-4.5,7) {$F_2$};
\node (v1) at (-4.5,5.5) {$F_3$};
\node (v1) at (-4.5,3) {$F_{k}$};
\node (v1) at (-4.5,1.5) {$F_{k+1}$};

\node (v1) at (-1.5,7.75) {$F_1'$};
\node (v1) at (-1.5,6.25) {$F_2'$};
\node (v1) at (-1.8,2.25) {$F_k'$};


\node (v1) at (6.1,8.4) {$L_2\subseteq \last$};
\node (v1) at (-1.3,0.5) {$L_1$};
\node (v1) at (5.5,6.9) {$L_3$};
\node (v1) at (5.5,5.4) {$L_3$};

\draw[fill] (0,4) circle (0.05);
\draw[fill] (0,4.2) circle (0.05);
\draw[fill] (0,4.4) circle (0.05);

\draw[fill] (0,9.6) [anchor=west]  node {$x_1$} circle (0.05);
\draw[fill] (0,9.2) [anchor=west]  node {$x_2$} circle (0.05);
\draw[thick] 
(0,9.6)  --(0,9.2)
(0,9.2)-- (0,9.0)
(0,9.2)-- (0.2,9.0)
(0,9.2)-- (-0.2,9)
;

\draw[thick]

(-2,8.5)--(-2.4,8.3)
(-2,8.5)--(-2.6,8.3)

(-2.6,6.8)--(-2.2,7)
(-2.8,6.8)--(-2.2,7)

(-2.5,3)--(-2.5,1.5)
(-2.5,3)--(-2.7,1.5)
(-2.5,3)--(-2.9,1.5)
(-2.5,3)--(-2.3,1.5)
(-2.5,3)--(-2.1,1.5)

(3.5,8.8)--(4.5,8.1)
(3.5,8.8)--(4.7,8.1)
(3.5,8.8)--(4.9,8.1)
(3.5,8.8)--(5.1,8.1)

(3.5,7.3)--(4.5,6.6)
(3.5,7.3)--(4.7,6.6)
(3.5,7.3)--(4.9,6.6)
(3.5,7.3)--(5.1,6.6)

(3.5,5.8)--(4.5,5.1)
(3.5,5.8)--(4.7,5.1)
(3.5,5.8)--(4.9,5.1)
(3.5,5.8)--(5.1,5.1)

(-2.1,1.5)--(-2.1,0.3)
(-2.1,1.5)--(-2.3,0.3)
(-2.1,1.5)--(-2.5,0.3)
(-2.1,1.5)--(-1.9,0.3)
(-2.1,1.5)--(-1.7,0.3)

(-1.6,7)--(-0.5,6.8)
(-0.5,6.8)--(-0.5,5.5)
(-0.5,6.8)--(-0.3,5.5)
(-0.5,6.8)--(-0.7,5.5)
;

\draw[thick,dashed]
(-2,8.5)--(-2,7)
(-2,8.5)--(-2.2,7)
(-2,8.5)--(-2.4,7)
(-2,8.5)--(-1.8,7)
(-2,8.5)--(-1.6,7)

(-2,5.5)--(-2.2,7)
(-2.2,5.5)--(-2.2,7)
(-2.4,5.5)--(-2.2,7)
(-1.8,5.5)--(-2.2,7)
(-2.6,5.5)--(-2.2,7)

;

\end{tikzpicture}
\caption{Illustration of our edge decomposition of $T$.
With dashed lines we indicate (parts of) a component of $H_1(F')$.}\label{fig:estimable}
\end{figure}

We collect further properties of our decomposition of $E(T)$ in the following two claims.

\begin{claim}\label{cl: properties2}
The following statements hold for each $i\in [k]$:
\begin{enumerate}[label={\rm(F0\arabic*)}]
\item\label{F00} The forests $F_1,\ldots,F_{k+1}$ are vertex-disjoint.
\item\label{F01} The forest $F'_i$ is a star-forest that is either empty or each component has size at least $np'/2$.
\item\label{F02}  
$ \displaystyle \Big(\bigcup_{t\in [i]} V(F_{t}) \cup \bigcup_{t\in [i-1]} V(F'_t)\Big)
\cap \Big( \bigcup_{t \in [k+1]\setminus [i]} V(F_{t}) \cup \bigcup_{t \in [k]\setminus[i-1]} V(F'_t)\Big) 
= R(F'_{i}).$
\item\label{F03}  
$\displaystyle \Big(\bigcup_{t\in [i]} V(F_{t}) \cup \bigcup_{t\in [i]} V(F'_t) \Big) 
\cap \Big(\bigcup_{t \in [k+1]\setminus [i]} V(F_{t}) \cup \bigcup_{t\in [k]\setminus [i]} V(F'_t) \Big) 
= R(F_{i+1})\cup R(F'_{i+1}).$
\item\label{F05} For each $T'\in H_i(F') \cup H_i(F^\#)$,
we have $r(T')\in V(F_i)\setminus R(F_i)$.
\end{enumerate}
\end{claim}
\begin{proof}
The statement~\ref{F00} follows directly from the definition.

Observe that for each $i\in [k]$, by definition, $F'_i$ is a star-forest.
By Claim~\ref{cl: h-value}, 
in the graph $F^2$, for every vertex $x\in V(F^2)$,
$|D_{F^2}(x)|$ is either $0$ or at least $np'$. 
Suppose $d_{F^2}(x)\geq np'$, then we have
\begin{eqnarray*}
d_{F^2}(x) - d_{\last}(x) \geq d_{F^2}(x) - d_{(L_2\cup L_3)\cap F^2}(x) 
\stackrel{\eqref{eq: L3 prop}}{\geq } d_{F^2}(x) - \frac{1}{2}d_{(\widehat{L}_2\cup \widehat{L}_3)\cap F^2}(x) \geq 
 \frac{1}{2} d_{F^2}(x)
\geq \frac{np'}{2}.
\end{eqnarray*}
As $|D_{F'_i}(x)| \in \{ 0, d_{F^2}(x)- d_{\last}(x) \}$, we obtain \ref{F01}.
The statements~\ref{F02} and \ref{F03} follow easily from~\eqref{eq: def Fi F'i}.%
To see~\ref{F05},
let $T'\in H_i(F')\cup H_i(F^\#)$ for some $i\in [k]$. As $r(T')$ is a root of a component of $F'$ or $F^\#$, 
the edge $e$ between $r(T')$ and its parent (which exists as $r(T') \neq x_1$ and $x_1$ is a leaf and $c(x_1x_2)=1$.) satisfies $c(e)=1$. 
Hence $x\in V(F_i)$ and $r(T')\notin R(F_i)$.
\end{proof}
Note that \ref{F05} shows that $H_{1}(F'),\dots, H_{k+1}(F')$ form a partition of $C(F')$ and 
$H_1(F^{\#}),\dots, H_{k+1}(F^{\#})$ form a partition of $C(F^{\#})$, thus $\LA_1^*,\dots, \LA_{k+1}^*$ form a partition of $\LA^*$.

\begin{claim}\label{cl: properties}
The following holds for all $i\in [k+1]$ and $j\in [k]$:
\begin{enumerate}[label={\rm(F1\arabic*)}]
\item\label{F11} $\Delta(F_i) \leq 40 k np'$.
\item \label{F12} If $L_1\neq \es$, then $L_2$ contains at least $(np')^{k}$ vertex-disjoint star-components of size at least $\Delta(L_1)/3$. 
Moreover, $\Delta(L_1) \leq  np' $.
\item \label{F13} For any component $T'$ of $F^{\#}$, we have $|T'|\leq \min\{ 2\Delta^{k}, n/M_*\}$.
\item  \label{F14} For any $T' \in C(F') \cup C(F^{\#})$, we have
$\max\{ \frac{|A_{T'}(r({T'})) |}{|B_{T'}(r({T'}))|}, \frac{|B_{T'}(r({T'}))|}{|A_{T'}(r({T'}))|} \} \geq M_*$.
\item \label{F16} $|\bigcup_{i=1}^{k+1} \LA_i | \geq \eta n \cdot \min\{\frac{np'}{\Delta},1\}$.
\item \label{F17} Suppose $x\in L(F'_{\ell})$ for some $\ell\in [k]$, then $|T(x)|\leq n^{1-1/(k+1)}$.
\item \label{F18} If Case 1 applies, then $\Delta\leq n^{3/4}\log{n}$.
\end{enumerate}
\end{claim}
\begin{proof}
Suppose $x_{i'} \in V(F_i)$ and $h(x_{i'})=\ell$.
By  \eqref{eq: c def}, we have that 
$$|D_{F_i}(x_{i'})| \leq  np' + \Big|  A_{i'} \cup B_{i'} \cup B_{i'}' \cup \bigcup_{\ell :  |H_{i'}^\ell| \leq  10np'  } H_{i'}^\ell \Big|  \stackrel{\eqref{eq: Ai size}, \eqref{eq: Ai Bi def}}{\leq} 
np' + (1+5+1+ 10(k+1))np' \leq 40k np'-1.$$
We obtain the penultimate inequality since Claim~\ref{cl: height and h} implies $|H_{i'}^\ell|=0$ for each $\ell > k$.
As $d_{F_i}(x_{i'}) \leq |D_{F_i}(x_{i'})|+1$ for all $i'\in [n]$, we obtain~\ref{F11}.

Suppose that $L_1 \neq \es$ and let $x \in B^*$ be a vertex such that $d_{L_1}(x)= \Delta(L_1)>0$. 
Since $x\in B^*$, by \eqref{eq: def B* L1},  the vertex $x$ is contained in a component $T'\in C(F')$ with $h(r(T'))= k$ (i.e. $T'$ has height $k$) and $h(x)=0$. Thus \eqref{eq: c1 h' hdT} implies that have $h'(x) = |D_T(x)\cap \he| = d_{L_1}(x)=\Delta(L_1)$. 
Let $x= y_{0} y_1\dots y_{k} = r(T')$ be the path $P_T(r(T'),x)$ between $x$ and $r(T')$. 
Claim~\ref{cl: h-value}
implies that $h(y_\ell)=\ell$ for each $\ell\in [k]\cup \{0\}$. 
Since $c(y_{\ell-1} y_{\ell})=2$ for each $\ell\in [k]$, the definition \eqref{eq: h' def} ensures that $h'(y_{k}) \geq h'(x) \geq \Delta(L_1)$.

Also $c(y_{\ell-1}y_{\ell})=2$ for each $\ell \in [k]$, 
 \eqref{eq: f bound k} implies that 
$P_T(x_1,y_{k})$ contains no edge $e$ with $c(e)=2$, thus
thus $y_{k}\in V(C^1)$. 
Claim~\ref{cl: h'-value} shows that $D_{F_1}( y_{k})$ contains at least $np'$ vertices $y$ with $h'(y) \geq h'(y_{k})$ and $h(y)=k-1$
as well as $c(y_{k}y)=1$. So all these vertices $y$ belong to $V(C^1)$. Repeatedly applying Claim~\ref{cl: h'-value} to vertices $y\in D_{F_1}^{\ell}(y_{k})$ with $h'(y)\geq h'(y_k)$ and $h(y)=k-\ell$ for each $\ell\in[k]$, we conclude that
$D_{F_1}^{k}(y_{k}) \subseteq V(C_1)$ contains at least $(np)^{k}$ vertices $y'$ with $h'(y') \geq h'(y_{k})\geq h'(x) \geq \Delta(L_1)$ and $h(y')=0$.
Since $h(y')=0$ and $h'(y')>0$, we have
$$d_{L_2}(y') \stackrel{\eqref{eq: L3 prop}}{\geq } 
 d_{\widehat{L}_2}(y')/3 \stackrel{\eqref{eq: def hat L2 L3}}{=} |\he \cap D_{T}(y')|/3 \stackrel{ \eqref{eq: c1 h' hdT}}{=}  h'(y')/3 \geq \Delta(L_1) /3.$$
 Thus there are at least $(np')^{k}$ distinct vertices $y''$ with $d_{L_2}(y'')\geq \Delta(L_1)/3$. 
As $(np')^{k} (\Delta(L_1)/3) \leq |L_2| \leq n$, \eqref{eq: def Delta p} implies that we have $\Delta(L_1)\leq np'$.
Thus we obtain \ref{F12}.

Note that any tree $T'\in C(F')$ with the root $r(T')$ has height $h(r(T'))\leq k$ by Claim~\ref{cl: height and h}.
An edge $e\in L_2\cup L_3$ does not join any two non-trivial components of $F'$ as $e$ contains a leaf vertex. Also, by \eqref{eq: def B* L1} and \eqref{eq: def hat L2 L3}, no edge in $L_2\cup L_3$ is incident to $B^*$.
Thus for any component $T'' \in C(F^{\#})$ (its height is at most $k$), we have
$|V(T'')| \leq \sum_{\ell=0}^{k} |D^\ell_T(r(T''))| \leq \sum_{\ell=0}^{k} \Delta^{\ell} \leq 2\Delta^{k}.$ 
Suppose $2\Delta^{k} \geq  n/M_*$, then \eqref{eq: def Delta p} with \eqref{eq: hierarchy} implies $k\geq 2$ and we have $np' = M_*^6 \Delta^{k+1} n^{-1} \geq M_*^6 (n/2M_*)^{(k+1)/k} n^{-1} \geq  4 M_*^{2} n^{1/k}$.
For any $y\in D_{F'}(r(T''))$, \eqref{eq: under red} implies that 
$|T(y)| \leq (np')^{-1}n$.
As $V(T'')\subseteq \bigcup_{y\in D_{F'}(r(T''))} T(y) \cup D_{\last}(r(T''))$, we have
$$|V(T'')| \leq \hspace{-0.2cm} \sum_{y \in D_{F'}(r(T''))} |T(y)| + d_{T}(r(T'')) 
\leq \Delta( (np')^{-1}n+1) \stackrel{\eqref{eq: def Delta p}}{\leq} n^{1/k} ( M_*^{-2}n^{1-1/k}/4+1) \leq n/M_* .$$
Thus \ref{F13} holds.

To verify \ref{F14},
consider first some $T'\in C(F')$ with a root $x:= r({T'})$ and $\ell:=h(x)$. Claim~\ref{cl: height and h} implies $1\leq \ell \leq k$. 
As $F' = F^2 \setminus \LA^*$, 
there exists a unique $T'' \in C(F^2)$ such that $T'\subseteq T''$. 
Since every vertex in $\LA^*$ is a leaf of $T$, we know $x=r(T'')$.
Since $\LA^*\subseteq L(L_2\cup L_3)$, we have
$$|D_{T'}^{\ell}(x)|
\geq |D_{T''}^{\ell}(x) \setminus L(L_2\cup L_3)| 
\stackrel{\eqref{eq: L3 prop},\eqref{eq: def hat L2 L3} }{\geq} |D_{T''}^{\ell}(x)\setminus \he| + \frac{1}{2}|D_{T''}^{\ell}(x)\cap \he| 
\stackrel{{\rm Claim}~\ref{cl: height and h}}{\geq} (np')^{\ell}/2.$$
On the other hand, $|\bigcup_{ \ell'=1}^{\ell-1} D_{T'}^{\ell'}(x)| \leq \sum_{\ell'=0}^{\ell-1} \Delta^{\ell} \leq 2\Delta^{\ell'-1}$.
Since exactly one of $A_{T'}(x)$ and $B_{T'}(x)$ contains all vertices in $D_{T'}^{\ell}(x)$, we have
$\max\{ \frac{|A_{T'}(x) |}{|B_{T'}(x)|}, \frac{|B_{T'}(x)|}{|A_{T'}(x)|} \}
\geq \frac{(np')^{\ell} /2}{ 2\Delta^{\ell-1}} 
\geq M_*$ by \eqref{eq: def Delta p}. 

Now we consider some $T'\in C(F^{\#})$. 
If $h(r(T'))=1$, then Claim~\ref{cl: height and h} implies that $T'$ is a star component of $F'$ or $F'\cup L_2\cup L_3$.
By \ref{F01}, \eqref{eq: def hat L2 L3} and \eqref{eq: L3 prop}, it is easy to see that $T'$ has size at least $np'/(3\log{n}) > M_*$; thus \ref{F14} holds in this case. 
If $(T',x)$ with $x:=r(T')$ has height $\ell \in [k]\setminus\{1\}$,
then there exists a unique nontrivial component $T''\in C(F^2)$ such that $T'\cap T'' \neq \emptyset$ and $x=r(T'')$.
 If $T'' = T'\setminus \last$ has height $\ell-1$, 
then there exists a vertex $y\in V(T'')$ such that $y\in D^{\ell-1}_{{T'}}(x)$, $h(y)=0$ and $|D_{\last}(y)|>0$.
However, 
$$|D_{T'}(y)| \geq |D_{ (\widehat{L}_2\cup \widehat{L}_3)\cap F^2}(y)| \stackrel{\eqref{eq: L3 prop}}{\geq} 2|D_{ (L_2\cup L_3) \cap F^2}(y)| \geq 2 |D_{\last}(y)| >0,$$ 
a contradiction.
Hence $T''$ also has height $\ell$. 
Then $|D_{T'}^{\ell}(x)| \geq |D_{T''}^{\ell}(x)|$ while $|\bigcup_{ \ell'=1}^{\ell-1} D_{T'}^{\ell'}(x)| \leq 2\Delta^{\ell-1}$. 
Thus, as before, by \eqref{eq: def Delta p}  we have 
$\max\{ \frac{|A_{T'}(x) |}{|B_{T'}(x)|}, \frac{|B_{T'}(x)|}{|A_{T'}(x)|} \}
\geq \frac{(np')^{\ell}/2 }{ 2\Delta^{\ell-1}} \geq M_*.$
Thus \ref{F14} holds.
 

Note that $\{ D_T^2(x_i) \cap \he: i\in [n]\} $ forms a partition of $\he$, 
as $x_1$ is a leaf and hence $x_2$ is not in $\he$.
For all $i\in [n]$
and $y\in B'_i, y'\in D_T(x_i)\sm B'_i$, by \eqref{eq: Ai Bi def} we know $|\he\cap D_T(y)|  \geq |\he\cap D_T(y')|$. 
This implies that $\sum_{y\in B'_i}|\he\cap D_T(y)| \geq\min\{\frac{np'}{\Delta},1\} |D_T^2(x_i) \cap \he|  $ as 
either $|B'_i|\geq \min\{\frac{np'}{\Delta},1\}  d_{T}(x_i)$ or $\bigcup_{y\in B'_i}\he\cap D_T(y)=D_T^2(x_i)$.
Thus
\begin{align*}
\Big|\bigcup_{i\in [n]} \bigcup_{y\in B'_i}(\he\cap D_T(y))\Big| 
\geq  \sum_{i\in [n]}  \left|D_T^2(x_i) \cap \he\right| \cdot  \min\{\frac{np'}{\Delta},1\}
=  \left| \he\right|\cdot \min\{\frac{np'}{\Delta},1\}
\geq 4\eta n\cdot \min\{\frac{np'}{\Delta},1\} 
\end{align*}
For all $i\in [n]$ and $y\in B'_i$, \eqref{eq: c def} implies that $c(x_iy)=1$, thus $y\in \bigcup_{\ell\in [k+1]}  V(F_{\ell})\setminus R(F_{\ell})$. Then \eqref{eq: L3 prop} implies that at least one third of the vertices in $\he \cap D_T(y)$ belong to $\LA_{\ell}$. Since this holds for every $y\in B'_i$, the above calculation shows \ref{F16}.

Suppose $x\in L(F'_{\ell})$ for some $\ell \in [k]$.
By the definition of $F'_\ell$,
we have $c(xa_T(x))=2$. 
Thus~\eqref{eq: under red} implies that $|T(x)|\leq n(np')^{-1} \leq n^{1- 1/(k+1)}$, thus \ref{F17} holds.

To show \ref{F18}, assume for a contradiction that $\Delta > n^{3/4} \log{n}$ and $|L_1|\geq \eta n$.
Then \eqref{eq: def Delta p} implies $k=1$, and
 $np' \geq \frac{M_* \Delta^{k+1}}{n} \geq n^{1/2} \log^2{n}$.
As every vertex with at least one heavy child, has at least $np'/\log n$ heavy children,
and as $L_1,L_2,L_3$ are vertex-disjoint star-forests,
we conclude that $\Delta(L_1) \geq {np'}/{\log{n}}$.
However, \ref{F12} implies $|L_2| \geq {(np')^2}/{(3\log{n})} \geq  n \log^3{n}>n$ which is a contradiction. 
Therefore, \ref{F18} holds.
\end{proof}

\section{Distribution of $V(T)$}
\label{sec: distribution}

In Section~\ref{sec:prep},
we defined the graph $G'$  and a partition $\{ V_\ih\}_{ \ih\in [r]\times [2]}$ of $V(G)$.
In this section, we define a partition  $\{X_\ih\}_{ \ih\in [r]\times [2]}$ of $V(T)\setminus L(L_1)$, a partition $\{L_{1,\ih}\}_{ \ih\in [r]\times [2]}$ of $L(L_1)$, a subgraph $F^{\circ}$ of $\last$ and a partition $\{Y_{\ih}\}_{ \ih\in [r]\times [2]}$ of $L(F^{\circ})$.
Later we aim to embed the vertices in $(X_{\ih} \setminus L(F^{\circ}))\cup Y_{\ih}\cup L_{1,\ih})$ into $V_{\ih}$ for each $\ih \in [r]\times [2]$.
Having in mind the edge-decomposition of $T$ into $F_1,\ldots,F_{k+1}, F_1',\ldots,F_k',L_1,\last$, which we defined in Section~\ref{sec:prep}, and the intention that the edges in $F_1,\ldots,F_{k+1}, L_1$ mainly are embedded into $R$ while the others are embedded into $G'$, we need to take care of several issues.
For example, assigning a vertex $x$ of $T$ to $X_\ih$ forces us later to embed all $y\in D_T(x)$ with $c(xy)=2$ into $X_\ihb$.
We also want that each $X_{\ih}$ contains enough vertices of $\LA^*$ and $\bigcup_{\ell \in [k+1]} \LA_\ell$ so that 
we have enough freedom at the end of embedding process.

In order to find a suitable collection $\{X_\ih\}_{ \ih\in [r]\times [2]}$ of vertices of $V(T)\sm L(L_1)$, we first describe an algorithm that proceeds in $k+2$ rounds.
For each $\ih \in [r]\times [2]$, let
 $$n_\ih := |V_\ih|.$$ 
For each $\ell\in [k+1]\cup \{0\}$, 
in $\ell$-th round, 
we will distribute vertices in $V(F_\ell)\sm V(R_\ell)$ and $V(T')\sm R(T')$ for each $T'\in H_\ell(F^\#)$ to 
build a collection $\{ Z_\ih^\ell\}_{ \ih\in [r]\times [2]}$ of pairwise disjoint sets such that \ref{Z1}$_\ell$--\ref{Z5}$_\ell$ hold (see below), and at the end we will set
$X_\ih:= \bigcup_{\ell=0}^{k+1} Z_\ih^\ell$.
Our main tool is Lemma~\ref{lem: vector distribution}.
\begin{enumerate}[label={\rm(Z\arabic*)}]
\item\label{Z1}\hspace{-0.2cm}$_{\ell}$ 
$ \displaystyle \max_{\ih,\ih'\in [r]\times[2]} \bigg\{ \Big|  \frac{|Z_\ih^\ell|}{n_\ih}  - \frac{  |Z_{\ih'}^\ell|}{n_{\ih'}}    \Big|\bigg\} 
\leq \min\Big\{ \frac{2r^5 \Delta^{k}}{n}, \frac{r^5}{M_*} \Big\}.$ 
\item\label{Z2}\hspace{-0.2cm}$_{\ell}$ 
$\displaystyle 
\bigcup_{j\in [\ell]\cup \{0\}}\bigcup_{\ih\in [r]\times[2]} Z_\ih^j 
= \bigcup_{j\in [\ell]\cup \{0\}} V(F_j) \cup\bigcup_{T'\in H_j(F^{\#}),\, j\in [\ell]} V(T').$
\item\label{Z3}\hspace{-0.2cm}$_{\ell}$ In Case 2, for each $\ih\in [r]\times [2]$, we have
$|Z_\ih^\ell \cap \LA^*| \geq 2\eta^2 |\LA^*_{\ell}| \frac{n_\ih}{n} - \min\{ 2r^2 \Delta^k, r^2 n/M_*\}.$
\item\label{Z4}\hspace{-0.2cm}$_{\ell}$ For all $j \in [\ell]$,
$T' \in H_{j}(F^{\#})$, and $x\in V(T')$,
if $x\in Z_\ih^j$, then $D_{T'}(x)\subseteq Z_\ihb^j$.
\item\label{Z5}\hspace{-0.2cm}$_{\ell}$ For each $\ih \in [r]\times [2]$, we have
$|Z_\ih^\ell \cap \LA_{\ell}| \geq 2\eta^2 |\LA_{\ell}| \frac{n_\ih}{n} - r^2 \Delta.$
\end{enumerate} \vspace{0.5cm}
Condition~\ref{Z1}$_\ell$ ensures that relative sizes of sets $Z_{\ih}^{\ell}$ resembles the relative sizes of the sets $V_{\ih}$.
Condition~\ref{Z2}$_\ell$ ensures that we actually assign every vertex to a set
whereas the conditions \ref{Z3}$_\ell$ and \ref{Z5}$_\ell$ ensure that enough vertices in $\LA^*$ and vertices in $\LA_\ell$ are assigned to each $Z_\ih^\ell$, respectively.
Condition \ref{Z4}$_\ell$ ensures that for every vertex $x$ assigned to be embedded into $V_{\ih}$ for some $\ih\in [r]\times[2]$, its child $y$ with $c(xy)=2$ is assigned to be embedded into $V_\ihb$.

\bigskip

\noindent {\bf Distribution algorithm.}\newline
Next we describe our distribution algorithm,
which relies on Lemma~\ref{lem: vector distribution}.
In Round 0, we let
$$Z_\ih^0:= \left\{\begin{array}{ll}
\{x_1\} & \text{ if } \ih=(1,1), \\
\emptyset & \text{ if } \ih \in [r]\times[2]\setminus\{(1,1)\}.
\end{array}\right.$$
Clearly, \ref{Z1}$_{0}$ and \ref{Z2}$_{0}$ hold.
We simply define $\LA_0,\LA^*_0:=\emptyset$ and so also \ref{Z3}$_{0}$ and \ref{Z5}$_{0}$ hold.
As $N_{T}(x_1)=\{x_2\}$ and $c(x_1x_2)=1$, also \ref{Z4}$_{0}$ holds.
We proceed to Round $1$. 

\medskip

\noindent {\bf Round $\ell$.} 
We define a set of vectors in $\cF\subseteq\N_0^6$
so that each vector $\bq_x\in \cF$ represents the implications  of the assignment of a vertex $x$ to a certain set $Z_\ih^\ell$.
For each $x \in V(F_\ell) \setminus R(F_{\ell})$,
we define a vector $\bq_x\in \N_0^6$.
The first two coordinates of $\bq_x$ measure how many vertices are forced to be assigned to $Z_\ih^\ell$ and to $Z_\ihb^\ell$, respectively, if we assign $x$ to $Z_\ih^\ell$
whereas the remaining coordinates measure how many vertices of $\LA^*_\ell$ and $\LA_\ell$ are then forced to be assigned in $Z_\ih^\ell$ and $Z_\ihb^\ell$, respectively.
To this end, for each $x \in V(F_\ell) \setminus R(F_{\ell})$, 
if $x$ is not the root of any non-trivial component in $F^\#$,
then we let $\bq_x:=(1,0,0,0,0,0)$ and otherwise let $\bq_x$ be as follows, where $x=r(T')$ for some $T'\in C(F^\#)$:
$$\bq_x:=\big(|A_{T'}(x)|, |B_{T'}(x)|, |A_{T'}(x)\cap \LA^*_\ell|, |B_{T'}(x)\cap \LA_\ell^*|, |A_{T'}(x)\cap \LA_\ell|, |B_{T'}(x)\cap \LA_\ell| \big).$$
Recall that $(A_{T'}(x), B_{T'}(x))$ is the bipartition of $T'$ such that $x\in A_{T'}(x)$.
We also define
\begin{align*}
\mathcal{F}&:= \{ \bq_x\}_{ x \in V(F_\ell)\setminus R(F_{\ell})}.
\end{align*}
Note that 
\begin{align}\label{eq: 34 sum Q}
\sum_{\bq_x \in \cF} (q_3+ q_4) = |\LA^*_\ell|\enspace \text{ and } \enspace\sum_{\bq_x \in \cF} (q_5+ q_6) = |\LA_\ell|.
\end{align}
We apply Lemma~\ref{lem: vector distribution} with the following objects and parameters to obtain a partition $\{\cF_\ih \}_{ \ih\in [r]\times [2]}$ of $\cF$.

\medskip
\noindent
{ \small
\begin{tabular}{c|c|c|c|c|c|c|c|c}
object/parameter & $ \cF $ & $ {n_\ih}/{n} $ & $t$ & $r$ & $\min\{2\Delta^k, n/M_*\} $ &   $\min\{2\Delta^k,n/M_*\}$ & $\Delta$ & $2\eta$
\\ \hline
playing the role of & $\bF$ & $\alpha_{i,h}$ & $t$ & $r$ &  $\Delta_1$ & $\Delta_2$  & $\Delta_3$ & $\beta$
\end{tabular}
}\newline \vspace{0.2cm}

\noindent Indeed, we have $1/r \ll 2\eta \ll 1/t$. 
By \ref{F14} and the definition of $\cF$, (A1)$_{\ref{lem: vector distribution}}$ holds. 
Condition (A2)$_{\ref{lem: vector distribution}}$ holds by \ref{F13} and (A3)$_{\ref{lem: vector distribution}}$ holds by the definition of $n_\ih$ and \ref{G2}.
Then Lemma~\ref{lem: vector distribution} provides a partition  $\{\cF_\ih \}_{ \ih\in [r]\times [2]}$ of $\cF$ satisfying the following:
\begin{align}\label{eq: property 2 dist}
\max_{\ih,\ih'\in [r]\times[2]} \Bigg\{
 \bigg|  \frac{ \sum_{\bq\in \cF_\ih} q_1 + \sum_{\bq\in \cF_\ihb} q_2 }{n_{\ih}}
 - \frac{\sum_{\bq\in \cF_{\ih'}} q_1 + \sum_{\bq\in \cF_{\ihb'}} q_2 }{n_{\ih'}} \bigg| \Bigg\} 
 \leq \min \Bigg\{ \frac{2r^5 \Delta^k}{n}, \frac{r^5}{M_*} \Bigg \},
\end{align}
\begin{align}\label{eq: property 1 dist}
\sum_{\bq\in \cF_\ih } q_{3} + \sum_{\bq\in \cF_{\ihb}} q_{4}
\geq  \frac{n_\ih}{n}\cdot4\eta^2 \sum_{ \bq\in \cF} (q_{3}+q_{4}) -\min\bigg\{ 2 r^2 \Delta^k,  \frac{r^2 n}{M_*} \bigg\}, \text{ and }
\end{align}
\begin{align}\label{eq: property 1 dist 2}
\sum_{\bq\in \cF_\ih } q_{5} + \sum_{\bq\in \cF_{\ihb}} q_{6}
\geq  \frac{n_\ih}{n}\cdot4\eta^2 \sum_{ \bq\in \cF} (q_{5}+q_{6}) - r^2 \Delta.
\end{align}
For each $\ih\in [r]\times[2]$ and $\bq_x \in \cF_{\ih}$, we add $A_{T'}(x)$ to $Z_\ih^\ell$ and $B_{T'}(x)$ to $Z_\ihb^\ell$; that is,
\begin{align*} Z_\ih^\ell:=  
 \{ x\colon \bq_x\in \cF_\ih\}\cup  \bigcup_{ \substack{({T'},x)\colon  \bq_x \in \cF_\ih, \\ {T'}\in H_{\ell}(F^{\#}),\, r({T'})=x }}
A_{T'}(x) \cup \bigcup_{\substack{({T'},x)\colon  \bq_x \in \cF_{ \ihb}, \\ {T'}\in H_{\ell}(F^{\#}),\, r(T')=x }} B_{T'}(x).
\end{align*}
This definition naturally yields \ref{Z4}$_{\ell}$. 
Property \ref{Z2}$_{\ell-1}$ together with the above definition implies~\ref{Z2}$_{\ell}$. 
From the above, we have 
$|Z_\ih^\ell| =  \sum_{\bq\in \cF_\ih} q_1 + \sum_{\bq\in \cF_{\ihb}} q_2$, 
so~\eqref{eq: property 2 dist} implies \ref{Z1}$_{\ell}$.
By \eqref{eq: 34 sum Q},  \eqref{eq: property 1 dist},  \eqref{eq: property 1 dist 2} and the fact that $n_\ih\leq n$, imply
\ref{Z3}$_{\ell}$ and \ref{Z5}$_{\ell}$.
If $\ell=k+1$, then we end the algorithm. Otherwise, we proceed to Round $(\ell+1)$. \newline

Once the above distribution algorithm has terminated, for each $\ih\in [r]\times [2]$, 
we let $X_\ih:= \bigcup_{\ell=0}^{k+1} Z_\ih^\ell$.
Recall that $\bigcup_{\ih\in [r]\times[2]} X_\ih$ consists of all vertices in $T$ except the leaves incident to edges in $L_1$.

We observe that \ref{Z1}$_{0}$--\ref{Z1}$_{k+1}$ and \ref{Z2}$_{k+1}$ imply that the following holds in Case 2:
\begin{align}\label{eq: size of Xih not big case 2}
\sum_{\ih\in [r]\times[2]} |X_\ih| = n - |L_1|, \enspace \text{ and }\enspace |X_\ih| =\frac{(n-|L_1|)n_\ih}{n} \pm (k+2)\min\bigg\{ 2r^5 \Delta^k, \frac{r^5n}{M_*}\bigg\}.
\end{align}

In Case 1, \ref{Z2}$_{k+1}$ with the fact that $|L_1|\geq \eta n$ implies that 
 $$\sum_{\ih\in [r]\times[2]} |X_\ih| = n - |L_1| \leq (1-\eta) n.$$
This together with \ref{Z1}$_{0}$--\ref{Z1}$_{k+1}$ implies that the following holds in Case 1:
\begin{align}\label{eq: size of Xih not big case 1}
|X_\ih| \leq (1-\eta)n_\ih + (k+2) \min\{ 2r^5\Delta^k, r^5n/M_*\} \leq (1-2\eta/3) n_\ih.
\end{align} 

In Case 2,  
we have $|\LA^*|\geq \eta n$. Thus \ref{Z3}$_{0}$--\ref{Z3}$_{k+1}$ imply that in Case 2, for each $\ih\in [r]\times[2]$, we have 
\begin{align}\label{eq: size of Xih W in case 2}
|X_\ih\cap \LA^*| 
&\geq \sum_{\ell=0}^{k+1} \Big(2\eta^2 |\LA^*_\ell| \frac{n_\ih}{n} - \frac{r^2n}{M_*}\Big) 
\geq \frac{3}{2}\eta^3 n_{\ih}.
\end{align}
Also \ref{Z5}$_{1}$--\ref{Z5}$_{k+1}$ imply that, for each $\ih\in [r]\times[2]$, we have 
\begin{align}\label{eq: size of Q*}
\bigg|X_\ih\cap \bigcup_{\ell \in [k+1]} \LA_\ell\bigg|
\geq \sum_{\ell=0}^{k+1} \Big(2\eta^2 |\LA_\ell| \frac{n_\ih}{n} - r^2\Delta\Big) 
\stackrel{\ref{F16}}{ \geq} \eta^3 n_{\ih} \cdot 
\min\bigg\{ \frac{np'}{\Delta},1\bigg	\}.
\end{align}

Now we have a partition $\{X_\ih\}_{\ih\in[r]\times[2]}$ of vertices which is almost well-distributed.
However, later we need that exactly $n_\ih$ vertices are embedded into $V_\ih$.
To ensure this we consider a subgraph $F^{\circ}$ of $\last$ and sets $L_{1,\ih}$ and $Y_{\ih}$ for each $\ih \in [r]\times [2]$ such that the following statements hold:
\begin{enumerate}[label={\rm(L\arabic*)}]
\item\label{L1} $L(L_1) = \bigcup_{\ih \in [r]\times [2]} L_{1,\ih}$ and $L(F^{\circ}) = \bigcup_{\ih \in [r]\times [2]} Y_{\ih}$.
\item\label{L2} For each $\ih \in [r]\times [2]$, we have
$|X_{\ih}\setminus L(F^{\circ}) | + |L_{1,\ih}| +|Y_{\ih}|= n_{\ih}$.
\item\label{L3} We have $F^{\circ}=\emptyset$ in Case 1 and $\Delta(F^{\circ}) \leq np'$ in Case 2.
\item\label{L4} $|L(F^{\circ})| \leq \min\{ r^7 \Delta^k, \frac{r^7 n}{M_*} \}.$
\end{enumerate} 
Note that \ref{L2} implies that the number of vertices which are assigned to $V_\ih$ is exactly $n_\ih$.
Condition~\ref{L3} and~\ref{L4} ensure that $\Delta(F^\circ)$ is small
so that we can embed it into $R$
while $\last\sm F^\circ$ is not too small.

Indeed, such a graph $F^{\circ}$ and sets $L_{1,\ih}$, $Y_{\ih}$ exist by the following claim.

\begin{claim}\label{cl: L1-4}
There exists a partition $\{L_{1,\ih} \}_{ \ih\in [r]\times [2]}$ of $L(L_1)$ and a subgraph $F^{\circ} \subseteq \last$ and a partition $\{ Y_{\ih} \}_{ \ih\in [r]\times [2]}$ of $L(F^{\circ})$ satisfying \ref{L1}--\ref{L4}.
\end{claim}
\begin{proof}
Let 
$\mathcal{I}^+:= \{ \ih\in [r]\times [2] : |X_i^{h}| > n_{\ih} \}$ and $\mathcal{I}^- := \{ \ih\in [r]\times [2] : |X_i^{h}| \leq n_{\ih} \}.$
In Case 1,
by  \eqref{eq: size of Xih not big case 1}, we have $\mathcal{I}^+ = \emptyset$.
Thus, if $\mathcal{I}^{+} \neq \emptyset$ holds, then Case 2 applies and \eqref{eq: size of Xih not big case 2} implies that for each $\ih \in [r]\times [2]$, we have
$|X_{\ih}| \leq n_{\ih} + (k+2) \min\{ 2r^5 \Delta^{k}, r^5 n/M_*\}.$
For each $\ih\in \mathcal{I}^{+}$, we choose a set $M_\ih \subseteq X_{\ih}\cap \LA^*$ such that
\begin{align}\label{eq: M set sizess}
|M_{\ih}| = |X_{\ih}|- n_{\ih}\leq (k+2) \min\{ 2r^5 \Delta^{k}, r^5 n/M_*\}
\end{align}
and  for any $x\in R(\last)\cap X_\ihb$, we have
\begin{align}\label{eq: np' max deg}
|M_{\ih}\cap N_{\last}(x)| \leq np'.
\end{align}
Indeed, since $d_{\last}(x) \leq \Delta$ for all $x\in R(\last)$ and Case 2 applies, we have
\begin{eqnarray*}
\sum_{x\in R(\last) \cap V_\ihb } \min\{ np', d_{\last}(x) \} 
&\geq& \min\Big\{\frac{ np'}{\Delta},1\Big\} \sum_{x\in R(\last) \cap V_\ihb } d_{\last}(x) \geq |X_\ih\cap \LA^*| \cdot  \min\Big\{\frac{np'}{\Delta},1\Big\}  \\
&\stackrel{  \eqref{eq: size of Xih W in case 2}}{ \geq }& \eta^3 n_\ih \cdot \min\Big\{ \frac{np'}{\Delta},  1\Big\}  
\stackrel{\eqref{eq: def Delta p}}{\geq} (k+2) 
\min\Big\{ 2r^5 \Delta^{k}, \frac{r^5 n}{M_*}\Big\}.
\end{eqnarray*}
Thus, by \eqref{eq: M set sizess} we can choose the desired set $M_{\ih}$ of size $|X_\ih|-n_{\ih}$ satisfying \eqref{eq: np' max deg}. For each $\ih \in \mathcal{I}^-$, let $M_{\ih} :=\emptyset$.
By \eqref{eq: size of Xih not big case 2}, we have
$$\sum_{\ih\in \mathcal{I}^+} |M_{\ih}|  = \sum_{\ih\in \mathcal{I}^+} (|X_{\ih}|-n_i) = 
\sum_{\ih \in \mathcal{I}^- } (n_{\ih}- |X_{\ih}|) - |L_1| \leq \sum_{\ih \in\mathcal{I}^- } (n_{\ih}- |X_{\ih}|)  .$$
Thus, we can partition $\bigcup_{\ih\in \mathcal{I}^+} M_{\ih}$ into $\{ Y_{\ih}\}_{  \ih \in \mathcal{I}^-}$ such that $|Y_{\ih}| \leq n_{\ih} - |X_{\ih}|$ for all $\ih \in \mathcal{I}^-$. For each $\ih \in \mathcal{I}^+$, let $Y_{\ih} :=\emptyset$.
Let $F^{\circ}$ be the graph with 
$$V(F^{\circ}) = \bigcup_{\ih \in [r]\times [2]} Y_{\ih} \cup \Big\{ a_{\last}(x) : x\in \bigcup_{\ih \in [r]\times [2]} Y_{\ih}\Big\} \enspace \text{and} \enspace
E(F^{\circ}) = E\big(\last[V(F^{\circ})]\big).$$
As $\bigcup_{\ih \in [r]\times [2]} Y_{\ih} \subseteq \LA^*$, we have that 
$\{Y_{\ih}\}_{ \ih\in [r]\times [2]}$ forms a partition of $L(F^\circ)$.
It is easy to see that by the above definition, for each $\ih \in [r]\times [2]$, we have
$$ |X_{\ih}\setminus L(F^{\circ})| + |Y_{\ih}| \leq n_{\ih}.$$
\COMMENT{
If $\ih \in \mathcal{I}^+$, then $Y_{\ih}=\emptyset$ and 
$|X_{\ih}\setminus M_{\ih}| = n_{\ih}$.
If $\ih \in \mathcal{I}^-$, then $M_{\ih}=\emptyset$ and 
$|X_{\ih}|+ |Y_{\ih}| \leq n_{\ih}$.
}
By \eqref{eq: size of Xih not big case 2}, we have
$|L(L_1)| =  \sum_{\ih\in [r]\times[2]} (n_{\ih} -|X_{\ih}\setminus M_{\ih}| + |Y_{\ih}| ).$ 
Hence there exists a partition $\{L_{1,\ih} \}_{ \ih\in [r]\times [2]}$ of $L(L_1)$ such that $|L_{1,\ih}| + |X_{\ih}\setminus M_{\ih}| + |Y_{\ih}| = n_{\ih}$ for all $\ih \in [r]\times [2]$.
The definition of $L_{1,\ih}$ and $Y_{\ih}$ trivially implies \ref{L1} and \ref{L2}.
Note that in Case 1, we have $\mathcal{I}^+ =\emptyset$ and hence $F^{\circ}=\emptyset$.  
Thus, by \eqref{eq: M set sizess} and \eqref{eq: np' max deg}, both \ref{L3} and \ref{L4} hold. This proves the claim.
\end{proof}

\section{Construction of embedding}\label{sec: embedding}

In this section we describe our algorithm embedding $T$ into $G\cup R$, which succeeds with high probability.
In the previous section, we assigned every vertex of $T\setminus L(L_1)$ to a set $X_{\ih}$ for some $\ih \in [r]\times [2]$ with the intention to embed (essentially) all vertices in $X_\ih$ to $V_\ih$.
As discussed earlier, we proceed in $k+1$ rounds and an additional final round. In round $\ell$, we embed the vertices in $V(F_\ell\cup F_\ell')\sm R(F_\ell)$ and in the final round we embed the vertices in $L(L_1)\cup L(\last)$.

At the beginning we choose disjoint sets $V_{\ih,\ell}$, $V'_{\ih,\ell}$ and $\widehat{V}_{\ih,\ell}$ in $V_\ih$ of size $n_{\ih,\ell},n_{\ih,\ell}',\widehat{n}_{\ih,\ell}$, respectively (see~\eqref{eq: nn' sizes}). 
Later in round $\ell$, we embed the vertices of $X_{\ih} \cap (V(F_\ell)\sm R(F_\ell))$ and $X_\ih\cap L(F_\ell')$ into one of the sets $V_{\ih,\ell},V'_{\ih,\ell},\widehat{V}_{\ih,\ell}$ . 
`Leftover' vertices in each of $V_{\ih,\ell}$, $V'_{\ih,\ell}$ and $\widehat{V}_{\ih,\ell}$ will be covered in the final round.
While embedding the edges of $F_{\ell}$ and $F'_{\ell}$ in each round, 
we keep track of how vertices of $T$ are distributed among neighbourhoods of vertices in $G'$ and among neighbourhoods of vertices in an irregularity-graph of $G'$. 
This information will help us to maintain `super-regularity' in the graph  induced by the `unused' vertices in $G'$. 
For this, we introduce multi-collections $\cB_i, \cB'_i, \cB''_i$, irregularity-graphs $J_{\ih}$ and functions $g_1, g_2, g_3, g_4,g_5,g_6$.

Recall that $|V_\ih|=n_\ih$ and $L(\last)=\LA^*$.
For all $\ell\in [k+1]\cup \{0\}$, $\ell' \in [k]$ and $\ih\in [r]\times [2]$, we define 
\begin{align}\label{eq: nn' sizes}
\begin{split}
n_{\ih,\ell} &:= |X_\ih\cap (V(F_\ell) \setminus R(F_{\ell}))| + \mu n_\ih, \enspace \enspace \enspace \widehat{n}_{\ih,\ell'} := |X_\ih \cap L(F'_{\ell'}) | + \mu n_\ih, \\
n'_{\ih,\ell}&:= \mu n_{\ih}, \text{ and }  \hspace{3.6cm} n_{\ih,k+2} := n_\ih -\Big(\sum_{j\in [k+1]\cup \{0\}} (n_{\ih,j}+n'_{\ih,j}) + \sum_{j\in [k]} \widehat{n}_{\ih,j}\Big).
\end{split}
 \end{align}
 Note that we have 
 $$X_{\ih} \cap \Big(\bigcup_{\ell \in[k+1]\cup \{0\}} (V(F_{\ell})\setminus R(F_{\ell})) \cup  \bigcup_{\ell'\in [k]} L(F'_{\ell'}) \Big)
 = X_{\ih} \cap (V(T)\setminus \LA^*) = X_{\ih}\setminus \LA^*.$$
Thus this with \eqref{eq: nn' sizes}\COMMENT{And use \ref{F02} and \ref{F03} to verifies the sets $V(F_{\ell})\setminus R(F_{\ell}), L(F'_{\ell'})$ are pairwise disjoint for all $\ell, \ell'$. } gives 
$$n_{\ih,k+2}\geq n_{\ih} - |X_{\ih}| + |X_{\ih}\cap \LA^*| - (3k+4)\mu n_{\ih}.$$
Since $\mu\ll \eta$, by \eqref{eq: size of Xih not big case 1} in Case 1 
and by \eqref{eq: size of Xih not big case 2} and \eqref{eq: size of Xih W in case 2} in Case 2, we obtain for each $\ih\in [r]\times[2]$
 \begin{align}\label{eq: nih0 size}
 |n_{\ih,k+2}|\geq \eta^3 n_\ih.
 \end{align}
 \COMMENT{vertices in $\LA^*$ does not belong to any $V(F_\ell)$ or $V(F'_{\ell})$.}
 In order to keep track of the $\epsilon$-regularity of appropriate subgraphs of $G'$, for each $\ih\in [r]\times [2]$, we consider the irregularity-graph $J_{\ih}:=J_{G'[V_{\ih},V_{\ihb}]}(V_{\ih},d,\epsilon)$. Recall that this is defined in \eqref{eq: def J graph}.
Lemma~\ref{lem: max degree HGde} together with \ref{G1} implies that
\begin{align}\label{eq: Delta Hih}
\Delta(J_\ih) \leq 2\epsilon n_\ih.
\end{align}
For each $\ih\in [r]\times [2]$, we define the following multi-collections (that is, we consider multi sets here) of subsets of $V_\ih$:
\begin{align*}
\cB_{\ih}&:= \{N_{G',V_\ih}(u)    : u\in V_\ihb \}, \enspace
 \cB'_{\ih} := \{ N_{J_\ih}(u) : u\in V_\ih \} \enspace  \text{and} \enspace 
 \cB''_{\ih} := \{ N_{G',V_\ih}(v,v') : v,v'\in V_\ihb \}.
\end{align*}
\COMMENT{Note that, if there are two vertices $u,u'\in V_\ihb$ with exactly same neighborhood, then $N_{G'}(u)= N_{G'}(u')$. In the case $\cB_{\ih}$ contains the set $N_{G'}(u)$ twice. So we have $|\cB_{\ih}| = |V_\ihb| $ in any case.}
Now for each $\ih\in [r]\times [2]$, 
we pick a partition 
$$\{ V_{\ih,\ell} : \ell \in [k+2]\cup \{0\} \} \cup \{V'_{\ih,\ell}: \ell\in [k+1]\cup \{0\}\} \cup \{\widehat{V}_{\ih,\ell} : \ell\in [k]\}$$ of $V_\ih$ 
satisfying the following for all $\ell\in [k+1]\cup \{0\}$, $\ell'\in [k]$ and $B\in \cB_{\ih}\cup  \cB'_{\ih}\cup \cB''_{\ih}$:
\begin{enumerate}[label={\rm(V\arabic*)}]
\item\label{V1} $|V_{\ih,\ell}| = n_{\ih,\ell} $, $|V'_{\ih,\ell}| = n'_{\ih,\ell}$ and $|\widehat{V}_{\ih,\ell'}| = \widehat{n}_{\ih,\ell'} $,
\item\label{V2} $|V_{\ih,\ell}\cap B| = \frac{n_{\ih,\ell} |B| }{n_\ih } \pm \epsilon^2 n_\ih$, $|V'_{\ih,\ell}\cap B| =\frac{n'_{\ih,\ell}|B|}{n_\ih} \pm \epsilon^2 n_{\ih}$ and $|\widehat{V}_{\ih,\ell}\cap B| = \frac{\widehat{n}_{\ih,\ell} |B| }{n_\ih } \pm \epsilon^2 n_\ih$.
\end{enumerate}
Indeed, if we choose a partition of $V_\ih$ uniformly at random among all partitions satisfying \ref{V1}, 
then Lemma~\ref{Chernoff Bounds} yields that \ref{V2} holds with probability at least $1-n^{-4}$. 
Thus union bounds with the fact that $|\cB_{\ih}\cup  \cB'_{\ih}\cup \cB''_{\ih}| < n^2$ 
imply that there exists a partition satisfying \ref{V1} and \ref{V2}. 

We need to take particular care of vertices with many grandchildren in $F^\#$.
To this end, for all $\ell \in [k+1]$ and $\ih \in [r]\times[2]$, let
\begin{align}\label{eq: def X' }
\begin{split}
X'_{\ih,\ell}&:= \Big\{x \in X_\ih\cap (V(F_{\ell})\setminus R(F_{\ell})) : |D_{F^\#}(D_{F'}(x)) | \geq \frac{n}{M\log{n}} \Big\}, \\
X'&:=\bigcup_{\ell\in [k+1], \ih \in [r]\times [2]} X'_{\ih,\ell} \enspace \text{ and } \enspace \overline{X'}:= V(T)\setminus X'.
\end{split}
\end{align}
Also, if $\ell\in [k]$ and $x\in L(F'_{\ell})$, then 
$|D_{F^{\#}}(D_{F'}(x))|\leq |T(x)|  \stackrel{\ref{F17}}{\leq} n^{1-1/(k+1)}$. Thus
\begin{align}\label{eq: F' leaf not in X'}
X'\cap L(F'_{\ell}) =\emptyset.
\end{align}
As $\sum_{x\in V(T)} |D_{F^\#}(D_{F'}(x))|\leq |V(T)|$,
it is easy to see that 
\begin{align}\label{eq: X' size M logn}
|X'| \leq M \log{n}.
\end{align}
In fact, we prove below that if $X'\neq \es$, then $k=2$.
For all $x\in V(T)$ and $\ih \in [r]\times [2]$, let 
\begin{align}\label{eq: def gcB}
g_1(x)&:= |D_{F'}(x)|, & g_2(x)  &:= |D_{\last}(x)|, \nonumber \\
g_3(x)&:= |D^2_{F'}(x)| \cdot\mathbbm{1}_{\overline{X'}}(x), &
g_4(x)&:=|D_{\last}(D_{F'}(x))| \cdot\mathbbm{1}_{\overline{X'}}(x), \nonumber \\
g_5(x)&:= |D^2_{F'}(x)|\cdot \mathbbm{1}_{X'(x)}, & g_6(x)&:= |D_{\last}(D_{F'}(x))|\cdot \mathbbm{1}_{X'}(x), \nonumber \\
 \cB_{\ih}^{1},\cB_{\ih}^{2}&:= \cB_{\ih} \cup  \cB'_{\ih} \enspace \text{ and }   & \cB_{\ih}^{3},\cB_{\ih}^{4} & :=  \cB'_{\ih}. 
\end{align}
For each $j\in [2]$ and $\ell\in [k]$, we consider the functions $g_{j+2}$ and $g_{j+4}$ since 
$g_{j+2}(x)$ and $g_{j+4}(x)$ for $x\in V(F_{\ell})\setminus R(F_{\ell})\cup L(F'_{\ell})$
provide information about values of $g_{j}(y)$ for $y\in L(F'_{\ell+1}).$
The functions $g_1, g_2, g_3$ and $g_4$ will play the roles of $d_{F}(x)$ in Lemma~\ref{lem: random matching behaves random} or 
Lemma~\ref{lem: p matching}, and $g_5+g_6$ will play the role of $f$ in Lemma~\ref{lem: random embedding}. 
Note that if $x\in X'$, then we have $g_5(x)+ g_6(x) = |D_{F^\#}(D_{F'}(x))|$.
The above collections $\cB_{\ih}^j$ will play the roles of $\cB_i$ in Lemma~\ref{lem: random embedding} or  Lemma~\ref{lem: random matching behaves random}. 
One of our main concern is that we need to guarantee that
(A2)$_{\ref{lem: random matching behaves random}}$ and (A3)$_{\ref{lem: random matching behaves random}}$ or~\ref{L45A2} and~\ref{L45A3} hold for $G'[V'_{\ih,\ell}, V'_{\ihb,\ell}]$ in each round.

For all $j\in [6]$, $\ell \in [k+1]$, $\ell'\in [k]$ and $\ih \in [r]\times[2]$, we define
\begin{align}\label{eq: m ih j ell def}
m_{\ih}^{j,\ell}:= \sum_{ x\in X_\ih\cap ( V(F_{\ell})\setminus R(F_{\ell}) )} g_j(x)
\enspace \text{and} \enspace \widehat{m}_{\ih}^{j,\ell'}:= \sum_{ x\in X_\ih \cap L(F'_{\ell'})} g_j(x).
\end{align}
By~\eqref{eq: F' leaf not in X'}, we have 
\begin{align}\label{eq:mhat56}
	\widehat{m}^{5,\ell'}_\ih=\widehat{m}^{6,\ell'}_\ih=0.
\end{align}
Let $\widehat{m}^{j,0}_{\ih}:=0$ for all $j\in [6]$ and $\ih\in [r]\times [2]$.
Note that by the definition of $g_j$ and \ref{Z4}$_\ell$, for all $j\in [2]$, $j'\in [6]$ and $\ell \in [k]$, as $R(F'_{\ell}) \subseteq (V(F_{\ell})\setminus R(F_{\ell}))\cup L(F'_{\ell-1})$ by \ref{F02} and \ref{F03}, we have 
\begin{align}\label{eq: m g relation}
m^{j+2,\ell}_{\ihb} +  m^{j+4,\ell}_{\ihb} + \widehat{m}^{j+2,\ell-1}_{\ihb}  \stackrel{(\ref{eq:mhat56})}{=} \widehat{m}^{j,\ell}_{\ih} \enspace \text{and} \enspace m^{j',\ell}_{\ih}, \widehat{m}^{j',\ell}_{\ih} \leq  2\max\{n_\ih, n_\ihb\} \stackrel{\ref{G2}}{\leq} 2tn/r, 
\end{align}
and for each $\ih \in [r]\times[2]$, we have
\begin{align}\label{eq: g sum to W*}
\sum_{\ell \in [k+1] } m^{2,\ell}_{\ihb} + \sum_{\ell \in [k]} \widehat{m}^{2,\ell}_{\ihb} = |X_{\ih} \cap \LA^*|.
\end{align}
For technical reasons, we assume that maximum over the emptyset equals $0$.
For several applications of concentration inequalities,
it will be convenient to define the following for each $j\in [6]$,
$$\Delta_{j}:=\max\Big\{ g_j(x): x \in \bigcup_{\ell\in [k+1]} V(F_{\ell})\setminus R(F_{\ell}) \Big\}
 \enspace  \text{and} 
\enspace \widehat{\Delta}_{j} := \max\Big\{g_j(x): x \in \bigcup_{\ell\in [k]} L(F'_{\ell}) \Big\}.$$
Note that if $k=1$ and $j\in \{5,6\}$ then for any vertex $x\in V(T)$, we have
\begin{align}\label{eq: g56 zero if k=1}
g_j(x)=0; \text{ thus } \Delta_{j}=\widehat{\Delta}_j=0 \text{ and } X' =\emptyset.
\end{align}
Indeed, if $k=1$, we have $g_5(x)=0$ as $F'$ is a star-forest, and we have
$g_6(x) =0$ as $\last \subseteq L_2\cup L_3$ is not incident to any vertex in $B^*=L(F')$. 
Thus, if $X'\neq \emptyset$, then $k\geq 2$ (in fact $k=2$) and for all $x\in V(T)$, $\ih \in [r]\times [2]$ and $\ell \in [k+1]$, we have
\begin{align}\label{eq: X' m g1 g2 value}
 g_1(x), g_2(x) \leq n^{1/2}. 
\end{align}
In order to not repeat the same argument for two cases and for each value of $k$, 
we define the following parameters $\nu$ and $w_*$. 
Note that we always have $w_*\geq M_*^{1/2}$.
Let
\begin{align}\label{eq: def w* nu}
\begin{split}
w_*:= \left\{ \begin{array}{ll}  \min\{ M_* \log{n}, M_*^{1/2} n^{1/2}/\Delta  \} & \text{ if } k=2, \\
M_*\log{n} & \text{ otherwise},
\end{array}\right.  \text{ and } 
\nu :=  \left\{ \begin{array}{ll}  \epsilon^{1/3} & \text{ if Case 2 applies and }  k=1, \\
n^{-1/10} & \text{ otherwise}.
\end{array}\right.
\end{split}
\end{align}

\begin{claim}\label{cl: Delta j not big}
For all $j\in [2]$, $j'\in \{3,4\}$, $j''\in \{5,6\}$ and $j_*\in [6]$, we have
\begin{align*}
 \Delta_j \leq \min\Big\{ \frac{n}{M\log{n}} , n^{1/k}\Big\}, \enspace \Delta_{j'} \leq \frac{n}{M\log{n}}, \enspace  \Delta_{j''} \leq \frac{n}{M_* w_*} \enspace \text{and} \enspace \widehat{\Delta}_{j_*} \leq n^{2/3}.
 \end{align*} 
\end{claim}
\begin{proof}
First, consider the case $k\geq 3$. Then for each $i\in [2]$, we have 
$\Delta_i, \widehat{\Delta}_i \leq n^{1/k}$ and
for $i\in [6]\setminus [2]$, we have
$\Delta_{i}, \widehat{\Delta}_{i} \leq \Delta^2 \leq n^{2/3} \leq \min\{\frac{n}{M\log{n}},\frac{n}{M_*w_*}\}$ by \eqref{eq: def Delta p}. As this proves what we want, we may assume $k\leq 2$.

For all $j_*\in [6]$, $\ell\in [k]$ and a vertex $x\in L(F'_{\ell})$, we have
$g_{j_*}(x) \leq |T(x)|$. 
Thus \ref{F17} implies $\widehat{\Delta}_{j_*}\leq n^{1-1/(k+1)} \leq n^{2/3}$. For $j\in [2]$, \eqref{eq: def Delta p} with the fact that $\Delta_{j}, \widehat{\Delta}_j\leq \Delta$ implies that for $j \in [2]$, we have $\Delta_{j}  \leq \min\{ \frac{n}{M\log{n}} , n^{1/k}\}$.  
For each $j'\in \{3,4\}$, $x\notin X'$, we have 
$g_{j'}(x) \leq |D_{F^{\#}}(D_{F'}(x))|\leq \frac{n}{M\log{n}}$ by the definition of $X'$ and 
the fact that $F^{\#} = F'\cup \last$. 
Thus $\Delta_{j'} \leq \frac{n}{M \log{n}}$ as $g_{j'}(x)=0$ for $x\in X'$.

Finally, for $j''\in \{5,6\}$, if $k=1$, then \eqref{eq: g56 zero if k=1} implies $\Delta_{j''}=0$.
Assume $k=2$ and $\Delta \leq {n^{1/2}}/{M_*^{3/2}}$. 
Then we have $w_* \geq M_*^{2}$ by \eqref{eq: def w* nu} and we obtain 
$\Delta_{j''}\leq \Delta^2 \leq  \frac{M_*^{1/2}n^{1/2}}{ w_*}\cdot \frac{n^{1/2}}{M_*^{3/2}}  \leq \frac{n}{M_* w_*}.$
 If $k=2$ and $\Delta > {n^{1/2}}/{M_*^{3/2}}$, then we have $w_*< M_*^{2}$.
For $x\in V(T)$, we have
 $$|D_{F^{\#}}(D_{F'}(x))| \leq \sum_{x' \in D_{F'}(x)} |T(x')| 
\stackrel{\eqref{eq: under red}}{\leq} |D_{F'}(x)| n (np')^{-1} 
\leq \Delta p'^{-1} \stackrel{\eqref{eq: def Delta p}}{\leq} \frac{n}{M_*^3} \leq \frac{n}{M_* w_*}.$$
 \COMMENT{Note that 
 $\Delta p'^{-1} \leq \Delta\cdot (M_*^6 \Delta^3 n^{-2})^{-1} \leq
 n^2/(M_*^6\Delta^2) \leq n/M_*^3$.
 }
 As $g_{j''}(x) \leq |D_{F^{\#}}(D_{F'}(x))|$ for all $x\in V(T)$, we obtain $\Delta_{j''} \leq  {n}/(M_* w_*)$. This proves the claim.
\end{proof}

Suppose next that $\phi$ is an injective map
defined on a subset of $V(T)$ into $V(G)$ and $\epsilon'>0$ is a (small) error parameter.
Next we define four types of sets of vertices which potentially ruin certain `quasi-randomness' properties of $G'$ on specified subsets of $V_{\ih}$.
These sets consist of vertices whose neighbourhoods are not as we would like them to be.
We later always want to ensure these sets to be small (in some cases even empty).
Recall the definition in~\eqref{def: fsum}.
For all $\ih\in [r]\times[2]$, $\ell\in \{0,\ldots,k+1\}$, $\ell'\in [k+1]$, $j\in [6]$,   and $j'\in [2]$, 
we define
\begin{align}
\begin{split}
\cC^{j,\ell}_{\ih}(\phi,\epsilon')&:= \Big\{ u \in V_\ih : 
\fsum{g_j}{\phi}{ N_{J_\ih, V_{\ih,\ell}\cup V'_{\ih,\ell} }(u)} >\epsilon^{1/2} m_{\ih}^{j,\ell} + \epsilon'  n_\ih \Big\},    \\
\cD^{j,\ell'}_{\ih}(\phi,\epsilon')&:= \Big\{ u \in V_\ih : 
\fsum{g_j}{\phi}{ N_{J_\ih, \widehat{V}_{\ih,\ell} }(u)}  >  \epsilon^{1/2}\widehat{m}_{\ih}^{j,\ell'}+ \epsilon'  n_\ih \Big\},  \\
\widetilde{\cC}^{j',\ell}_{\ih}(\phi,\epsilon')&:= \Big\{ u \in V_\ihb : 
\fsum{g_{j'}}{\phi}{ N_{G', V_{\ih,\ell}\cup V'_{\ih,\ell} }(u)} \neq
(d\pm \epsilon^{1/2}) m_{\ih}^{j',\ell} \pm \epsilon' n_\ih \Big\}, \\
\widetilde{\cD}^{j',\ell'}_{\ih}(\phi,\epsilon')&:= \Big\{ u \in V_\ihb :  \fsum{g_{j'}}{\phi}{ N_{G', \widehat{V}_{\ih,\ell} }(u)} \neq (d\pm \epsilon^{1/2}) \widehat{m}_{\ih}^{j',\ell'} \pm \epsilon' n_\ih \Big\}.\label{eq: def cC cD}
\end{split}
\end{align}

We start our embedding algorithm with a function $\phi_0$ that maps $x_1$ to an arbitrary vertex in $V_{(1,1),0}$.
For each $\ell\in [k+1]$, we will iteratively extend the function $\phi_{\ell-1}$ to $\phi_{\ell}$ satisfying the following properties for all $\ih\in [r]\times [2]$, $j\in \{1,2\}$, $j'\in \{3,4\}$ and $j''\in \{5,6\}$:
\begin{enumerate}[label={\text{\rm($\Phi$\arabic*)}}]
\item\hspace{-0.2cm}$_\ell$\label{Phi1} $\phi_{\ell}\colon\bigcup_{\ell'=0}^{\ell} (F_{\ell'} \cup F'_{\ell'}) \rightarrow G'\cup \bigcup_{\ell'=1}^{\ell} R_{\ell'}$ is an embedding.

\item\hspace{-0.2cm}$_\ell$ \label{Phi2} $\phi_{\ell}( (X_\ih \setminus X') \cap ( V(F_{\ell})\setminus R(F_{\ell}) ) \subseteq V_{\ih,\ell}$ and $\phi_{\ell}( (X_\ih \cap X') \cap ( V(F_{\ell})\setminus R(F_{\ell}) ) \subseteq V'_{\ih,\ell}$.

\item\hspace{-0.2cm}$_\ell$ \label{Phi2'} if $\ell \leq k$, 
then $\phi_{\ell}(X_\ih\cap L(F'_{\ell})) \subseteq \widehat{V}_{\ih,\ell}$.

\item\hspace{-0.2cm}$_\ell$\label{Phi3}
$\cC^{j,\ell}_{\ih} (\phi_\ell,\nu)= \widetilde{\cC}^{j,\ell}_{\ih} (\phi_\ell, \nu )=\cC^{j',\ell}_{\ih}(\phi_\ell,\epsilon) =\emptyset \enspace$
and $\enspace |\cC^{j'',\ell}_{\ih}(\phi_\ell,\epsilon^{1/3})| \leq  2^{-  w_*} n$.

\item\hspace{-0.2cm}$_\ell$\label{Phi4} if $\ell \leq k$, then
$\cD^{j,\ell}_{\ih}(\phi_{\ell},\epsilon_{\ell}) =\cD^{j',\ell}_{\ih}(\phi_\ell, \epsilon_{\ell} )=\cD^{j'',\ell}_{\ih}(\phi_\ell, \epsilon_{\ell} ) =\emptyset \enspace$  and 
$\enspace |\widetilde{\cD}^{j,\ell}_{\ih} (\phi_\ell, \epsilon_{\ell} )| \leq  2^{-  w_*} n.$
\end{enumerate}
Note that it is easy to see that $\phi_0$ satisfies  \ref{Phi1}$_0$-- \ref{Phi4}$_0$.
We proceed with Round 1.
\newline

\noindent {\bf Embedding algorithm.} \newline

\noindent {\bf Round $\ell$ with $\ell \leq k+1$.}
Assume we have defined $\phi_{\ell-1}$ satisfying \ref{Phi1}$_{\ell-1}$--\ref{Phi4}$_{\ell-1}$. 
We first proceed to Step $\ell.1$ and then to Step $\ell.2$. \newline

\noindent {\bf Step $\ell.1$.}
In this step, we embed $V(F_{\ell})\setminus R(F_{\ell})$ into $R_\ell$ by using Lemma~\ref{lem: random embedding}. 
Moreover, we use Lemma~\ref{lem: randomly distributing} to ensure that the value of $g_j$ is well-distributed over the sets in $\cB_{\ih}^{j}$; 
thus concluding that \ref{Phi3}$_\ell$ holds.

Recall that $R_{\ell} \in \bG(n,M_* p')$.
Let $\sE_{1,\ell}$ be the event that there exist a map $\phi'_{\ell}$ extending $\phi_{\ell-1}$ which embeds $F_{\ell}$ into $R_{\ell}$ and a multi-collection $\cF_{\ih,\ell}\subseteq \cB'_{\ih}$ satisfying the following.
\begin{enumerate}[label=\text{($\Phi'$\arabic*)}]
\item \label{phi'1} $\phi'_{\ell}( (X_{\ih}\cap V(F_{\ell})) \setminus X') \subseteq  V_{\ih,\ell}$ and $\phi'_{\ell}(X_{\ih}\cap V(F_{\ell})\cap X') \subseteq V'_{\ih,\ell}$.
\item \label{phi'2} $|\cF_{\ih,\ell}| \leq 2^{-w_*}n$.  
\item \label{phi'3} For all $\ih\in [r]\times [2]$, $B\in \cB'_{\ih}\setminus \cF_{\ih,\ell}$ and $j\in \{5,6\}$, we have 
$\fsum{g_j}{\phi'_{\ell}}{B\cap V'_{\ih,\ell}} \leq \epsilon^{2/5} n_{\ih}$.
\end{enumerate}
We apply Lemma~\ref{lem: random embedding} with respect to the following graphs and parameters to estimate $\Pro[\sE_{1,\ell}]$. \newline

\noindent
{ \small
\begin{tabular}{c|c|c|c|c|c}
object/parameter & $n$ & $(F_{\ell},R(F_{\ell}))$&  $[r]\times [2]$ &   $ (X_{\ih}\cap V(F_{\ell})) \setminus (R(F_{\ell})\cup X')$ &    $V\setminus (\bigcup_{\ih}V_{\ih,\ell}\cup V_{\ih,\ell}')$
\\ \hline
playing the role of & $n$ & $(F,R)$ & $[r]$ &  $X_i $   & $U$ \\ \hline \hline

object/parameter & $t M_* w_*/r$ & $R_\ell$ & $V_{\ih,\ell}$  &  $(X_{\ih}\cap V(F_{\ell})\cap X')\setminus R(F_{\ell})$ &  $\{ B\cap V'_{\ih,\ell} : B\in \cB'_{\ih}\}$ 
\\ \hline
playing the role of & $w$ & $G$  &  $V_i$ &  $X'_i$& $\cB_{i}$ \\ \hline \hline
object/parameter &  $ 3\epsilon$ & $M_*p'$ & $V'_{\ih,\ell}$ & $ \frac{r}{2tn}(g_5(x)+g_6(x))$ &  $\phi_{\ell-1}|_{R(F_\ell)}$ 
\\ \hline
playing the role of &  $ \epsilon$ & $p$  & $ V'_i$   &   $f(x)$  & $\phi'$ \\ 
\end{tabular}
}\newline \vspace{0.2cm}

\noindent 
Next, we check that conditions \ref{L41A1}--\ref{L41A4}  hold
so that we can apply Lemma~\ref{lem: random embedding}.
Condition \ref{L41A1} holds because
\begin{eqnarray*}
|V_{\ih,\ell}| &\stackrel{\eqref{eq: nn' sizes}}{=}&  |X_\ih\cap (V(F_\ell) \setminus R(F_{\ell}))| + \mu n_\ih 
\stackrel{\ref{G2}}{\geq} |X_\ih\cap (V(F_{\ell}) \setminus R(F_{\ell}))| + \mu n/(2tr)\\ 
&\stackrel{\ref{F11},(\ref{eq: def Delta p})}{\geq} &  |(X_\ih\cap V(F_{\ell})) \setminus (R(F_{\ell})\cup X')|  + 12 \cdot (M_*p')^{-1}\cdot 40 knp' + 30 (M_*p')^{-1}\log{n}.
\end{eqnarray*}
Note that we have $|V'_{\ih,\ell}| \stackrel{\eqref{eq: nn' sizes}}{=} \mu n_{\ih} 
\stackrel{\ref{G2}}{\geq} n^{1-1/(k+1)} \log^6{n} \stackrel{ \eqref{eq: def Delta p}}{\geq} (M_*p')^{-1} \log^6{n}.$
This with \eqref{eq: X' size M logn} implies that \ref{L41A2} holds.

For all $\ih\in [r]\times [2]$ and $B\in \cB_{\ih}'$, we have
$$|B\cap V'_{\ih,\ell} | \stackrel{\ref{V2},\eqref{eq: nn' sizes},\eqref{eq: def gcB}}{= }\mu |B| \pm \epsilon^2 n_{\ih} \stackrel{\eqref{eq: nn' sizes},\eqref{eq: Delta Hih} }{\leq} 3\epsilon |V'_{\ih,\ell} |.$$ \COMMENT{We obtain the final inequality since $|B|\leq 2\epsilon n_{\ih}$ and $|V'_{\ih,\ell}|= \mu n_{\ih}$.}
Thus \ref{L41A3} holds.

Condition \ref{L41A4} follows from \eqref{eq: m g relation} and the fact that $\Delta_5+\Delta_6\leq 2n/(M_*w_*)$ holds by Claim~\ref{cl: Delta j not big}.  

As $(3\epsilon)^{1/2}\cdot 2tn/r \leq \epsilon^{2/5} n_{\ih}$ (by~\eqref{eq: hierarchy}), and $(3\epsilon)^2 \cdot 2t M_* w_*/r \geq w_*$,\COMMENT{This part is needed to deduce \ref{phi'2} and \ref{phi'3} from the conclusion of the Lemma~\ref{lem: random embedding}.}  
Lemma~\ref{lem: random embedding} implies 
\begin{align}\label{eq: E1ell prob}
\bPr[\sE_{1,\ell}] \geq 1- n^{-7/4}.
\end{align}
Recall that once $\sE_{1,\ell}$ holds, then a desired embedding $\phi'_{\ell}$ exists. 
Moreover, once $\sE_{1,\ell}$ holds, and if we choose a embedding $\phi'_{\ell}$ uniformly at random among all embeddings of $F_\ell$ into $R_\ell$ extending $\phi_{\ell-1}$ and satisfying \ref{phi'1}--\ref{phi'3}, then such a chosen embedding satisfies more properties with high probability. To prove this, we let $\Pi_{\ell}$ be the set of all injective maps $\sigma: V(F_{\ell}) \setminus (R({F_{\ell}}) \cup X') \rightarrow \bigcup_{\ih\in [r]\times[2]} V_{\ih,\ell}$  such that 
$\sigma( (X_{\ih}\cap V(F_{\ell})) \setminus (R(F_\ell)\cup X') \subseteq V_{\ih,\ell}$ holds for each $\ih \in [r]\times [2]$.
Let $\Pi'_{\ell} \subseteq \Pi_{\ell}$ be the set of all $\sigma \in \Pi_{\ell}$ satisfying the following for all $j\in [4], \ih \in [r]\times [2]$ and $B\in \cB^j_{\ih}$:
\begin{align}\label{eq: E ell 1 holds}
\fsum{g_j}{\sigma}{B\cap V_{\ih,\ell}} = \frac{|V_{\ih,\ell}\cap B| m^{j,\ell}_{\ih} }{ n_{\ih,\ell}} \pm  \epsilon \sqrt{ M_* \Delta_j n_{\ih,\ell}  \log{n}}.
 \end{align}
 Let $\sE'_{1,\ell}$ be the event that there exists a function $\phi'_{\ell}$ extending $\phi_{\ell-1}$ which embeds $F_{\ell}$ into $R_{\ell}$
 satisfying \ref{phi'1}--\ref{phi'3} and 
 $\left.\phi'_{\ell}\right|_{\bigcup_{\ih\in [r]\times[2]}V_{\ih,\ell}} \in \Pi'_{\ell}$.
Now we apply Lemma \ref{lem: randomly distributing} for each $\ih\in [r]\times [2]$ with respect to the following objects and parameters.
\newline

\noindent
{ \small
\begin{tabular}{c|c|c|c|c|c|c|c|c|c}
object/parameter & $( X_\ih \setminus X')\cap (V(F_{\ell}) \setminus R(F_{\ell})) $ & $V_{\ih,\ell}$ & $4$ & $g_j$  & $\epsilon$ & $n_{\ih,\ell}$& $\{ B\cap V_{\ih,\ell} : B\in \cB^j_{\ih}\}$& $\mu^{-1}t^2$ & $M_*\log{n}$
\\ \hline
playing the role of & $U$ & $V$ & $s$ & $f_j$  & $\epsilon$ & $n$ &$\cB_j$ & $t$ & $w_j$ 

\end{tabular}
}\newline \vspace{0.2cm}

\noindent Indeed, this is possible by Claim~\ref{cl: Delta j not big}.
As for $\ih \in [r]\times [2]$ and a function $\sigma$ chosen uniformly at random in among $\Pi_{\ell}$, $\left.\sigma\right|_{V_{\ih,\ell}}$ is uniformly distributed among all injective maps from $( X_\ih \setminus X')\cap (V(F_{\ell}) \setminus R(F_{\ell}))$ to $V_{\ih,\ell}$.
Thus Lemma \ref{lem: randomly distributing} implies that 
a map $\sigma \in \Pi_{\ell}$ chosen uniformly at random will satisfy \eqref{eq: E ell 1 holds} for all $j\in [4]$ and $B\in \cB^j_{\ih}$ with probability at least $1- |V_{\ih,\ell}|^{-3} \geq 1 - n^{-5/2}$. 
Moreover, the maps $\left.\sigma\right|_{V_{\ih,\ell}}$ are mutually independent over all $\ih \in [r]\times[2]$ as their domains are disjoint.
Thus a union bound implies that a randomly chosen $\sigma$ lies in $\Pi'_{\ell}$ with probability at least $1 - n^{-2}$. Consequently
$$|\Pi'_{\ell}| \geq (1-  n^{-2}) |\Pi_{\ell}|.$$
Since $R_{\ell}\in \bG(n,M'p)$ is a binomial random graph, 
the distribution of $R_\ell$ is invariant under vertex permutations.
Moreover, for all $j\in \{5,6\}$, $\ih\in [r]\times[2]$ and $B\in \cB'_{\ih}$
the value of $\fsum{g_j}{\phi'_{\ell} \circ \sigma}{B\cap V'_{\ih,\ell}}$ is invariant
for any vertex permutation $\sigma$ which is the identity on $V(G)\setminus \bigcup_{\ih\in [r]\times[2]} V_{\ih,\ell}$.
Thus, 
assuming $\sE_{1,\ell}$, for any two permutation $\sigma,\sigma' \in \Pi_{\ell}$, if we choose a random embedding $\phi'_{\ell}$ satisfying \ref{phi'1}--\ref{phi'3} (if there are more than one such embedding, then we choose one uniformly at random), then we have
$$\mathbb{P}\Big[\sigma = \left.\phi'_{\ell}\right|_{\bigcup_{\ih\in [r]\times [2]} V_{\ih,\ell}} \Big] = \mathbb{P}\Big[\sigma'= \left.\phi'_{\ell}\right|_{\bigcup_{\ih\in [r]\times [2]} V_{\ih,\ell}}\Big].$$ 
Therefore, we conclude
\begin{align}\label{eq: failure probability 1}
\mathbb{P}[\sE'_{1,\ell} ] \geq \mathbb{P}\Big[\left.\phi'_{\ell}\right|_{\bigcup_{\ih\in [r]\times [2]} V_{\ih,\ell}} \in \Pi'_{\ell} \mid \sE_{1,\ell}\Big] \mathbb{P}[\sE_{1,\ell}] = \frac{|\Pi'_{\ell}|}{|\Pi_{\ell}|}\mathbb{P}[\sE_{1,\ell}] \geq (1-2n^{-3/2}).
\end{align}
\COMMENT{Here, the first inequality is not an equality, because there could be more than one embedding satisfying \ref{phi'1}--\ref{phi'3}. Then we could get a strict inequality.}
Now we verify that once $\sE'_{1,\ell}$ holds, then there exists a map $\phi'_{\ell}$ extending $\phi_{\ell-1}$ which satisfies \ref{Phi2}$_\ell$ and \ref{Phi3}$_\ell$.
We continue with a simple claim for later use.
 \begin{claim}\label{eq: bound by nu}
 $$\epsilon\sqrt{ M_* \Delta_j n_{\ih,\ell}  \log{n}}  \leq \left\{ \begin{array}{ll} \nu n_\ih/2 &  \text{ if } j\in [2],\\ \epsilon n_{\ih} &  \text{ if } j\in [4]\setminus [2]. \end{array}\right.$$
 \end{claim}
 \begin{proof}
Recall that $M= 10k M_*^7$. 
By Claim~\ref{cl: Delta j not big}, we have $\Delta_{j}\leq n/(M\log{n})$ for all $j\in [4]$. 
Thus $$ \epsilon \sqrt{ M_* \Delta_j n_{\ih,\ell}  \log{n}}
\leq \epsilon \sqrt{ M_* \log{n} \cdot \frac{n}{M\log{n}} \cdot n_\ih }  
\leq \epsilon n_\ih.$$
This implies the claim for the case $j\in \{3,4\}$.
By \eqref{eq: def w* nu}, this also implies the claim if $k=1$ and Case 2 applies.

Now we suppose $j\in [2]$. Moreover, we suppose either $k\geq 2$ or Case 1 applies.
Thus, by \ref{F18} and Claim~\ref{cl: Delta j not big}, we conclude $\Delta\leq n^{3/4}\log n$. As $\Delta_j\leq \Delta$, we obtain 
$$ \epsilon \sqrt{ M_* \Delta_j n_{\ih,\ell}  \log{n}}
\leq \epsilon \sqrt{ M_* n^{3/4} \log^2{n}\cdot n_\ih}  
\stackrel{\eqref{eq: def w* nu}}{\leq} \nu n_\ih/2.$$
This proves the claim.
\end{proof}
Note that 
for all $j\in [2]$, $\ih \in[r]\times[2]$ and $B\in \cB^j_\ih$, we have
\begin{eqnarray}\label{eq: g12 V' sum}
\fsum{g_j}{\phi'_{\ell}}{B\cap V'_{\ih,\ell}}
\stackrel{\ref{phi'1}}{\leq} \sum_{x\in X'} g_j(x) \stackrel{\eqref{eq: X' size M logn},\eqref{eq: X' m g1 g2 value}}{\leq}  M n^{1/2} \log{n}.\end{eqnarray}
Now, we assume that $\sE'_{1,\ell}$ holds, 
and demonstrate that an embedding $\phi_\ell'$ satisfying $\sE'_{1,\ell}$ also satisfies \ref{Phi3}$_\ell$.
First, consider $j\in [4]$, $\ih\in [r]\times [2]$ and $u\in V_{\ih}$. By \eqref{eq: Delta Hih} and \ref{V2}, we obtain 
\begin{align}\label{NJih}
	|N_{J_\ih,V_{\ih,\ell}}(u)|\leq 3\epsilon n_{\ih,\ell}.
\end{align}
By \eqref{eq: def gcB}, we have $N_{J_\ih,V_{\ih,\ell}}(u) \in \cB^{j}_{\ih}$.
Consequently, $\sE'_{1,\ell}$ with \eqref{eq: E ell 1 holds} implies that
\begin{eqnarray}\label{eq: Vihell value}
\fsum{g_j}{\phi'_{\ell}}{N_{J_\ih,V_{\ih,\ell}}(u)} 
\stackrel{\eqref{NJih}}{\leq}  \epsilon^{1/2} m^{j,\ell}_{\ih} +  \epsilon \sqrt{ M_* \Delta_j n_{\ih,\ell}  \log{n}} .
\end{eqnarray}\COMMENT{
$\fsum{g_j}{\phi'_{\ell}}{N_{J_\ih,V_{\ih,\ell}}} 
\stackrel{\eqref{NJih}}{\leq} 3\epsilon m_{\ih}^{j,\ell} +  \epsilon \sqrt{ M_* \Delta_j n_{\ih,\ell}  \log{n}}  \leq \epsilon^{1/2} m^{j,\ell}_{\ih} +  \epsilon \sqrt{ M_* \Delta_j n_{\ih,\ell}  \log{n}}.$
}
If $j\in [2]$, then by Claim~\ref{eq: bound by nu} we have
\begin{eqnarray*}
\fsum{g_j}{\phi'_{\ell}}{N_{J_\ih,V_{\ih,\ell}\cup V'_{\ih,\ell}}(u)} 
&=& \fsum{g_j}{\phi'_{\ell}}{N_{J_\ih,V_{\ih,\ell}}(u)} + 
\fsum{g_j}{\phi'_{\ell}}{N_{J_\ih, V'_{\ih,\ell}}(u)} \\
&\stackrel{\eqref{eq: g12 V' sum},\eqref{eq: Vihell value}}{\leq}&
\epsilon^{1/2} m^{j,\ell}_{\ih} +  \nu n_{\ih}/2 + M n^{1/2} \log{n}
\leq \epsilon^{1/2} m^{j,\ell}_{\ih} +  \nu n_{\ih}.
\end{eqnarray*}
If $j\in \{3,4\}$, then by \eqref{eq: def gcB}, we have $g_j(x)=0$ for all $x\in X'$.
Thus by Claim~\ref{eq: bound by nu}, we have
\begin{eqnarray*}
\fsum{g_j}{\phi'_{\ell}}{N_{J_\ih,V_{\ih,\ell}\cup V'_{\ih,\ell}}(u)} 
= \fsum{g_j}{\phi'_{\ell}}{N_{J_\ih,V_{\ih,\ell}}(u)} 
\stackrel{\eqref{eq: Vihell value}}{\leq }\epsilon^{1/2} m^{j,\ell}_{\ih}  + \epsilon n_{\ih}.
\end{eqnarray*}
Thus we have $\cC^{j,\ell}_{\ih} (\phi_\ell,\nu)=\emptyset$ if $j\in [2]$ and $\cC^{j,\ell}_{\ih} (\phi_\ell,\epsilon)=\emptyset$ if $j\in \{3,4\}$.

Now we consider $j\in [2]$, $\ih\in [r]\times [2]$ and $v\in V_{\ih}$.
Observe that \ref{G1},  \ref{V2} and \eqref{eq: def gcB} imply that $|N_{G',V_{\ih,\ell}}(v)|=(d\pm 2\epsilon)n_{\ih,\ell}$ for all $v\in V_\ihb$.
Hence, again by \eqref{eq: def gcB} and Claim~\ref{eq: bound by nu}, we conclude that $\sE'_{1,\ell}$ implies
$$ \fsum{g_j}{\phi'_{\ell}}{N_{G',V_{\ih,\ell}}(v)} 
\stackrel{\eqref{eq: E ell 1 holds}, \text{Claim~\ref{eq: bound by nu}}}{=} (d\pm 2\epsilon)n_{\ih,\ell}\cdot\frac{  m^{j,\ell}_{\ih} }{ n_{\ih,\ell}} + \nu n_\ih/2 
= (d\pm \epsilon^{1/2})m^{j,\ell}_{\ih} \pm \nu n_\ih/2.$$
This implies
$$ \fsum{g_j}{\phi'_{\ell}}{N_{J_\ih,V_{\ih,\ell}\cup V'_{\ih,\ell}}(u)}
\stackrel{\eqref{eq: g12 V' sum},\ref{phi'1}}{\leq}
  \fsum{g_j}{\phi'_{\ell}}{N_{J_\ih,V_{\ih,\ell}}(u)} + M n^{1/2} \log{n}
= (d\pm \epsilon^{1/2})m^{j,\ell}_{\ih} \pm \nu n_\ih.$$
Hence we have $\widetilde{\cC}^{j,\ell}_{\ih}(\phi_\ell,\nu) =\emptyset$.

Consider $j\in \{5,6\}$. 
By \eqref{eq: def gcB} we have $g_j(x)=0$ for any $x\in \overline{X'}$. 
Thus \ref{phi'1} implies $ \fsum{g_j}{\phi'_{\ell}}{ N_{J_\ih,V_{\ih,\ell}}(u) }=0$. 
This implies that for any $u\in V_{\ih}$ such that $N_{J_\ih}(u) \in \cB_i'\setminus \cF_{\ih,\ell}$, by \ref{phi'1}, we conclude that $\sE'_{1,\ell}$ implies
$$\fsum{g_j}{\phi'_{\ell}}{ N_{J_\ih,V_{\ih,\ell}\cup V'_{\ih,\ell}}(u) }
= \fsum{g_j}{\phi'_{\ell}}{ N_{J_\ih, V'_{\ih,\ell}}(u)}
\stackrel{\ref{phi'3}}{\leq} \epsilon^{2/5} n_{\ih} 
\leq \epsilon^{1/2} m_{\ih}^{j,\ell} + \epsilon^{1/3} n_{\ih}.$$
Thus $|\cC^{j',\ell}_{\ih}(\phi_\ell,\epsilon^{1/3})|  \leq |\cF_{\ih,\ell}| \stackrel{\ref{phi'2}}{\leq} 2^{-w_*} n.$
Therefore, if $\sE'_{1,\ell}$ occurs, 
then we have an embedding $\phi'_{\ell}$ which extends $\phi_{\ell-1}$ and satisfies \ref{Phi3}$_\ell$,
otherwise we end the algorithm with failure. Note that \ref{Phi2}$_{\ell}$ holds by the construction of $\phi'_{\ell}$.
If $\ell=k+1$, then we proceed to the final round (observe that then also
\ref{Phi1}$_{k+1}$,\ref{Phi2'}$_{k+1}$ and \ref{Phi4}$_{k+1}$ hold as they are implied from \ref{Phi1}$_{k}$, \ref{Phi2'}$_{k}$ and \ref{Phi4}$_{k}$), otherwise we proceed to Step $\ell.2$.
\newline

\noindent {\bf Step $\ell$.2.}
In this step, we embed  $X_{\ih} \cap L(F'_{\ell})$ into $\widehat{V}_{\ih,\ell}$ by using Lemma~\ref{lem: random matching behaves random} in such a way that  \ref{Phi1}$_{\ell}$, \ref{Phi2'}$_{\ell}$ and \ref{Phi4}$_{\ell}$ hold.
For this, we need to verify
(A1)$_{\ref{lem: random matching behaves random}}$--(A3)$_{\ref{lem: random matching behaves random}}$.

By \ref{F02}, \ref{Phi1}$_{\ell-1}$ and \ref{Phi2}$_{\ell-1}$, 
the set $R(F'_{\ell})=\cen(F'_{\ell})$ is already embedded by $\phi'_{\ell}$ into $\bigcup_{\ih\in [r]\times[2]} (V_{\ih,\ell} \cup V_{\ih,\ell}\cup \widehat{V}_{\ih,\ell-1})$.
By \ref{Z4}$_{0}$--\ref{Z4}$_{k+1}$, we know that for all $x\in X_\ih\cap L(F'_{\ell})$, we have $a_T(x) \in X_\ihb$
and so \ref{Phi2}$_{\ell-1}$ implies that  
$ \phi'_{\ell} ( a_T(x) ) \in V_{\ihb,\ell} \cup V'_{\ihb,\ell} \cup \widehat{V}_{\ihb,\ell-1}.$ 
Thus the parent of $x$ is already embedded into the `correct' cluster.
As the sum of the `$g_1$-value' of the neighbours of the vertices in $\widetilde{\cD}^{1,\ell-1}_{\ihb}(\phi_{\ell-1},\epsilon_{\ell-1})$ is `wrong',
the vertices in $\widetilde{\cD}^{1,\ell-1}_{\ihb}(\phi_{\ell-1},\epsilon_{\ell-1})$ may not satisfy the second condition in (A2)$_{\ref{lem: random matching behaves random}}$.
So we simply remove these vertices and consider the following objects:
\begin{align}\label{eq: defs V''}
\begin{split}
\widetilde{V}_{\ih,\ell} &:= \widehat{V}_{\ih,\ell}\setminus \widetilde{\cD}^{1,\ell-1}_{\ihb}(\phi_{\ell-1},\epsilon_{\ell-1}),   \hspace{1.35cm}
\widetilde{F}_{\ih}:= F'_{\ell}[ R(F'_{\ell}) \cap X_\ihb, L(F'_{\ell}) \cap X_{\ih} ],  \\ 
\widetilde{G}_{\ih} &:= G'[ V_{\ihb,\ell} \cup V'_{\ihb,\ell}\cup  \widehat{V}_{\ihb,\ell-1}, \widetilde{V}_{\ih,\ell} ] \enspace \text{ and } \enspace
\widetilde{J}_{\ih} := J_{\widetilde{G}_{\ih} }( V_{\ihb,\ell} \cup V'_{\ihb,\ell}\cup  \widehat{V}_{\ihb,\ell-1}, d, \epsilon_{\ell-1}^{1/2}). 
\end{split}
\end{align}
\COMMENT{Note that $\widetilde{\cD}^{1,\ell-1}_{\ihb}(\phi_{\ell-1},\epsilon_{\ell-1})$ can contains some vertices outside $\widehat{V}_{\ih,\ell}$.}
Note that we aim to embed $\widetilde{F}_{\ih}$ into $\widetilde{G}_{\ih}$ using Lemma~\ref{lem: random matching behaves random}.
In such an application, $\widetilde{J}_{\ih}$ will play the role of $J_G(U,d,\epsilon)$.

Observe that \ref{Phi4}$_{\ell-1}$ implies that $|\widetilde{\cD}^{1,\ell-1}_{\ihb}(\phi_{\ell-1},\epsilon_{\ell-1})| \leq 2^{-w_*} n$, and so
\begin{eqnarray}\label{eq: V'' size}
|\widetilde{V}_{\ih,\ell}| = \widehat{n}_{\ih,\ell} \pm  2^{-w_*} n \stackrel{\eqref{eq: def w* nu}}{=} (1\pm \epsilon) \widehat{n}_{\ih,\ell} \stackrel{\eqref{eq: nn' sizes}}{\geq}\mu n_\ih/2.
\end{eqnarray}
Now we wish to apply Lemma~\ref{lem: random matching behaves random} for each $\ih\in [r]\times [2]$ with the followng objects and parameters. \newline

\noindent
{ \small
\begin{tabular}{c|c|c|c|c|c|c}
object/parameter & $\widetilde{G}_{\ih}$ &$V_{\ihb,\ell}  \cup V'_{\ihb,\ell}\cup \widehat{V}_{\ihb,\ell-1}$
& $\widetilde{V}_{\ih,\ell}$ & $|\widetilde{V}_{\ih,\ell}|$& $\widetilde{F}_{\ih}$  & $g_j$
\\ \hline
playing the role of & $G$ & $U$ & $V$ & $n$  &  $F$  & $f_j$    \\ \hline \hline
object/parameter& $4$ &    $\{ B\cap \widetilde{V}_{\ih,\ell} : B\in \cB^j_{\ih}\}$ & $\phi'_{\ell}\mid_{R(F_{\ell}')\cap X_\ihb}$ & $\epsilon_{\ell-1}^{1/2}$ & $t^2\mu^{-1}$
\\ \hline
playing the role of & $s$&  $\cB_{j}$  & $\psi$ & $\epsilon$ & $t$  
\end{tabular}
}\newline \vspace{0.2cm}

\noindent In order  to apply Lemma~\ref{lem: random matching behaves random}, we first check that (A1)$_{\ref{lem: random matching behaves random}}$--(A3)$_{\ref{lem: random matching behaves random}}$ hold with the above objects and parameters.
Indeed, $\widetilde{F}_{\ih}$ is a star-forest by \ref{F01}, and $\phi'_{\ell}\mid_{R(F_{\ell}')\cap X_\ihb}$ is an injective map from $R(\widetilde{F}_{\ih})$ to $V_{\ihb,\ell} \cup V'_{\ihb,\ell}\cup \widehat{V}_{\ihb,\ell-1}$ by \ref{Phi2}$_{\ell}$ and \ref{Phi2'}$_{\ell-1}$.

Note that for each $j\in [4]$, by \ref{Z4}$_{\ell}$, \eqref{eq: def gcB} and the definition of $X_{\ih}$, we obtain
$$\sum_{u\in L(\widetilde{F}_{\ih})} g_j(u) \leq \max\big\{ |X_\ih| , |X_\ihb|\big\} \stackrel{\substack{\ref{G2},\eqref{eq: size of Xih not big case 2},\\ \eqref{eq: size of Xih not big case 1},\eqref{eq: V'' size}}}{\leq} t^2 \mu^{-1}  |\widetilde{V}_{\ih,\ell}|.$$
Moreover, we have
\begin{align}\label{eq: L g sum sum}
|L(\widetilde{F}_{\ih})| = \sum_{x\in R(F'_{\ell})\cap X_{\ihb}} d_{F'_{\ell}}(x) 
=\sum_{x\in V(F_{\ell})\setminus R(F_{\ell})\cap X_{\ihb}} g_1(x)
+ \sum_{x\in L(F'_{\ell-1})\cap X_{\ihb}} g_1(x) = m_{\ihb}^{1,\ell}+\widehat{m}_{\ihb}^{1,\ell-1}.
\end{align}
As $\widehat{\Delta}_j\leq n^{2/3}$ by Claim~\ref{cl: Delta j not big}, this implies that
(A1)$_{\ref{lem: random matching behaves random}}$ holds.\COMMENT{Note that $n^{2/3}\leq \frac{\epsilon^{2}_{\ell-1}|\widetilde{V}_{\ih,\ell}| }{\log{|\widetilde{V}_{\ih,\ell}|}}$. }
For each $u \in V_{\ihb,\ell} \cup V'_{\ihb,\ell}\cup \widehat{V}_{\ihb,\ell-1}$, 
we have $N_{G',V_\ih}(u) \in \cB^1_{\ih}$. Thus 
\ref{G1}, \ref{V2} and \eqref{eq: V'' size} yield that 
\begin{eqnarray}\label{eq: A degree good 1}
d_{\widetilde{G}_{\ih} }(u) 
= (d\pm \epsilon)|\widetilde{V}_{\ih,\ell}| \pm (\epsilon^2 n_\ih + |\widetilde{\cD}^{1,\ell-1}_{\ihb}(\phi_{\ell-1},\epsilon_{\ell-1})|) \stackrel{\eqref{eq: V'' size}, \text{\ref{Phi4}$_{\ell-1}$}}{=} (d \pm \epsilon_{\ell-1}^{1/2})|\widetilde{V}_{\ih,\ell}| .
\end{eqnarray}
This verifies the first part of (A2)$_{\ref{lem: random matching behaves random}}$.

To see the second part of (A2)$_{\ref{lem: random matching behaves random}}$, we fix a vertex $v\in \widetilde{V}_{\ih,\ell}$.
By \eqref{eq: def gcB}, we have $N_{G',V_{\ihb} }(v)\in \cB^1_{\ihb}$. 
As $\widetilde{\cC}^{1,\ell}_\ihb(\phi_{\ell},\nu)=\emptyset$ by \ref{Phi3}$_{\ell}$ and as $v\notin \widetilde{\cD}^{1,\ell}_\ihb(\phi_{\ell-1},\epsilon_{\ell})$, we conclude
 that 
\begin{eqnarray}\label{eq: V'V eq g1}
\begin{split}
\fsum{g_1}{\phi'_{\ell}}{N_{G',V_{\ihb,\ell}\cup V'_{\ihb,\ell}}(v)} 
&\stackrel{\eqref{eq: def cC cD}}{=}& 
(d\pm \epsilon^{1/2}) m_{\ihb}^{1,\ell} \pm \nu n_\ihb \enspace \text{ and }\\
\fsum{g_1}{\phi'_{\ell}}{N_{G',\widehat{V}_{\ihb,\ell-1}}(v)}  
&\stackrel{\eqref{eq: def cC cD}}{=}&   
(d\pm \epsilon^{1/2}) \widehat{m}_{\ihb}^{1,\ell-1}   \pm \epsilon_{\ell-1} n_{\ihb}.
\end{split}
\end{eqnarray}
Therefore
\begin{eqnarray*}
\fsum{d_{\widetilde{F}_{\ih}}}{\phi'_{\ell}}{N_{\widetilde{G}_{\ihb}}(v) }
&=&\fsum{d_{F'_{\ell}}}{\phi'_{\ell}}{ N_{G',V_{\ihb,\ell}\cup V'_{\ihb,\ell}}(v)} 
+ \fsum{d_{F'_{\ell}}}{\phi'_{\ell}}{ N_{G',\widehat{V}_{\ihb,\ell}}(v)}  \\
&\stackrel{(\ref{eq: def gcB})}{=}&
\fsum{g_1}{\phi'_{\ell}}{ N_{G',V_{\ihb,\ell}\cup V'_{\ihb,\ell}}(v)} 
 + \fsum{g_1}{\phi'_{\ell}}{ N_{G',\widehat{V}_{\ihb,\ell}}(v)} \\
&\stackrel{\eqref{eq: V'V eq g1}}{ =} &(d\pm \epsilon^{1/2})(m_{\ihb}^{1,\ell}+\widehat{m}_{\ihb}^{1,\ell-1})  \pm  2\epsilon_{\ell-1} n_\ih 
\stackrel{\eqref{eq: V'' size},\eqref{eq: L g sum sum}}{=} d |L(\widetilde{F}_{\ih})| \pm \epsilon_{\ell-1}^{1/2} |\widetilde{V}_{\ih,\ell}|.
\end{eqnarray*}
Hence also the second part of (A2)$_{\ref{lem: random matching behaves random}}$ holds.

Now we verify (A3)$_{\ref{lem: random matching behaves random}}$. 
Recall the definition of $\widetilde{J}_\ih$ in~\eqref{eq: defs V''}.
Note that for any $u\in V_{\ihb,\ell} \cup V'_{\ihb,\ell}\cup \widehat{V}_{\ihb,\ell-1}$, by \ref{V2} and \eqref{eq: def J graph}, we have
$$N_{\widetilde{J}_{\ih}}(u) \subseteq N_{J_{\ihb}}(u)\cap (V_{\ihb,\ell}\cup V'_{\ihb,\ell}\cup \widehat{V}_{\ihb,\ell-1}) \enspace \text{ and } \enspace
N_{J_{\ihb}}(u) \in \cB^{1}_{\ihb}.$$

By \eqref{eq: def cC cD}, \ref{Phi3}$_\ell$ and \ref{Phi4}$_{\ell-1}$,
for each $u\in V_{\ihb,\ell} \cup V'_{\ihb,\ell} \cup \widehat{V}_{\ihb,\ell-1}$ we have
\begin{align}\label{eq: NH value small}
\fsum{g_1}{\phi'_{\ell}}{N_{J_{\ihb}, V_{\ihb,\ell}\cup V'_{\ihb,\ell}}(u)} \leq 
 \epsilon^{1/2} m^{1,\ell}_{\ihb}+ \epsilon^{1/3} n_\ihb \enspace \text{and} \enspace 
\fsum{g_1}{\phi'_{\ell}}{N_{J_{\ihb}, \widehat{V}_{\ihb,\ell-1}}(u)}  
\leq  \epsilon^{1/2} \widehat{m}^{1,\ell-1}_{\ihb}+ \epsilon_{\ell-1} n_\ihb.
\end{align}
Thus, we have 
\begin{eqnarray*}
& & \hspace{-2cm} \qquad\sum_{x,x'\colon\phi'_{\ell}(x) \phi'_{\ell}(x') \in E(\tilde{J}_{\ih}) } d_{\widetilde{F}_{\ih}}(x) d_{\widetilde{F}_{\ih}}(x') 
 \leq  \sum_{ x\in \cen(\widetilde{F}_{\ih}) }\sum_{y\colon \phi'_{\ell}(y) \in  N_{\widetilde{J}_{\ih}}(\phi'_{\ell}(x )) } d_{F'_{\ell}}( y ) d_{F'_{\ell}}(x)  \\
 &\leq & \sum_{ x\in \cen(\widetilde{F}_{\ih}) }d_{F'_{\ell}}(x) \left(
 \fsum{g_1}{\phi'_{\ell}}{N_{J_\ihb, V_{\ihb,\ell}\cup V'_{\ihb,\ell}}(\phi'_{\ell}(x))} 
 +  \fsum{g_1}{\phi'_{\ell}}{N_{J_\ihb, \widehat{V}_{\ihb,\ell-1}}(\phi'_{\ell}(x))} \right)\\
&\stackrel{\eqref{eq: NH value small}}{ \leq}   &
\sum_{ x\in \cen(\widetilde{F}_{\ih}) }d_{F'_{\ell}}(x) \left(  \epsilon^{1/2} (m^{1,\ell}_{\ihb}+ \widehat{m}^{1,\ell-1}_{\ihb}) + 2\epsilon_{\ell-1} n_\ihb\right)  \stackrel{\eqref{eq: m g relation}}{\leq}  2\epsilon_{\ell-1} (tn/r)^2 
\stackrel{\eqref{eq: nn' sizes}}{ \leq} \epsilon_{\ell-1}^{1/2} | V_{\ihb,\ell} \cup V'_{\ihb,\ell}\cup \widehat{V}_{\ihb,\ell-1}|^2.
\end{eqnarray*}
Hence (A3)$_{\ref{lem: random matching behaves random}}$ holds as well.

Therefore, we can indeed apply Lemma~\ref{lem: random matching behaves random} for each $\ih\in [r]\times[2]$ and 
can extend $\phi'_{\ell}$ to $\phi_{\ell}$ such that
for each $\ih\in [r]\times[2]$, 
the function $\phi_{\ell}$ embeds $\widetilde{F}_{\ih}$ into $G'$ in such a way that
$\phi_{\ell}( L(\tilde{F}_{\ih}) )\subseteq \widehat{V}_{\ih,\ell}$ and 
for each $j\in [4]$ and $B\in \cB^j_{\ih}$, we have 
\begin{align}\label{eq: B2 for case 1}
\fsum{g_j}{\phi_{\ell}}{B\cap \widehat{V}_{\ih,\ell}} = 
\sum_{x\in \cen(\widetilde{F}_{\ih}) }\sum_{y\in N_{\widetilde{F}_{\ih}}(x)} \frac{g_j(y) | N_{G}( \phi_{\ell}(x) )\cap B\cap \widetilde{V}_{\ih,\ell}| }{ d| \widetilde{V}_{\ih,\ell}| } \pm  \epsilon^{1/400}_{\ell-1} | \widetilde{V}_{\ih,\ell}|.
\end{align}

This immediately implies that \ref{Phi1}$_\ell$--\ref{Phi2'}$_\ell$ hold and \ref{Phi3}$_\ell$ also holds as $\phi_{\ell}$ extends $\phi'_{\ell}$.
Next we only need to verify \ref{Phi4}$_\ell$.
We fix some $u\in V_\ih$ and $j\in [4]$.
The definitions in \eqref{eq: def gcB} imply that $N_{J_\ih}(u) \in \cB^{j}_{\ih}$. 
Thus
\begin{eqnarray*}
\fsum{g_j}{\phi_{\ell}}{N_{J_{\ih},\widehat{V}_{\ih,\ell}}(u)}
&\stackrel{\eqref{eq: B2 for case 1}}{=}& \sum_{x\in \cen(\widetilde{F}_{\ih}) }\sum_{y\in N_{\widetilde{F}_{\ih}}(x)} 
 \frac{g_j(y) | N_{G}( \phi_{\ell}(x) )\cap N_{J_\ih}(u) \cap \widetilde{V}_{\ih,\ell}| }{ d |\widetilde{V}_{\ih,\ell}| } \pm  \epsilon^{1/400}_{\ell-1}|\widetilde{V}_{\ih,\ell}|\\
&\stackrel{\eqref{eq: Delta Hih}, \eqref{eq: V'' size}}{\leq}&  \sum_{x\in \cen(\widetilde{F}_{\ih}) }\sum_{y\in N_{\widetilde{F}_{\ih}}(x)}  \frac{4 g_j(y)  \epsilon n_\ih }{ d \mu n_\ih } +  \epsilon^{1/400}_{\ell-1} n_\ih 
\leq  \epsilon^{1/2} \widehat{m}_{\ih}^{j,\ell} + \epsilon_{\ell} n_\ih.
\end{eqnarray*}
Here, the penultimate inequality holds since
$|N_{G}( \phi_{\ell}(x) )\cap N_{J_\ih}(u) \cap \widetilde{V}_{\ih,\ell}|\leq |N_{J_\ih}(u)|$ and
the final inequality holds since 
$$\sum_{x\in \cen(\widetilde{F}_\ih)}\sum_{y\in N_{\widetilde{F}_\ih}(x)}g_j(y) 
= \sum_{ y\in X_{\ih}\cap L(F'_{\ell})} g_j(y)  \stackrel{\eqref{eq: m ih j ell def}}{=} 
\widehat{m}_{\ih}^{j,\ell}.$$
Since this holds for all $j\in [4]$ and $u\in V_\ih$, we have
$\cD^{j,\ell}_{\ih}(\phi_{\ell},\epsilon_{\ell}) =\emptyset.$
For each $j\in \{5,6\}$, by \eqref{eq: F' leaf not in X'}, we also have $\cD^{j,\ell}_{\ih}(\phi_{\ell},\epsilon_{\ell}) =\emptyset.$
Hence the first part of \ref{Phi4}$_\ell$ holds.

Consider $j\in [2]$ and $u\in V_\ihb \setminus \cC^{j+4,\ell}_{\ihb}(\phi'_{\ell},\epsilon^{1/3})$. 
Since $N_{G', V_{\ih}}(u)\in \cB^{j}_{\ih}$ and $u\notin \cC^{j+4,\ell}_{\ihb}(\phi'_{\ell},\epsilon^{1/3})$, definition \eqref{eq: def cC cD} implies that
\begin{eqnarray} \label{eq: B property hold j45}
\fsum{g_{j+4}}{\phi_{\ell}}{N_{J_{\ihb},V_{\ihb,\ell}\cup V'_{\ihb,\ell}}(u)}\leq \epsilon^{1/2}m^{j+4,\ell}_{\ihb} + \epsilon^{1/3} n_\ihb \stackrel{\eqref{eq: m g relation}}{\leq} 2 \epsilon^{1/3} n_\ihb.
\end{eqnarray}
Also \ref{Phi4}$_{\ell-1}$ with  \eqref{eq: def cC cD} implies that 
\begin{eqnarray}\label{eq: B property hold j45 2}
\fsum{g_{j+4}}{\phi_{\ell}}{N_{J_{\ihb},\widehat{V}_{\ihb,\ell-1}}(u)}
\leq \epsilon^{1/2}\widehat{m}^{j+4,\ell}_{\ihb} + \epsilon_{\ell-1} n_\ihb \stackrel{\eqref{eq: m g relation}}{\leq} 2 \epsilon_{\ell-1} n_\ihb.
\end{eqnarray}
Also,  \ref{Phi3}$_{\ell}$ and \ref{Phi4}$_{\ell-1}$ imply that
\begin{eqnarray}\label{eq: B property hold j45 3}
\fsum{g_{j+2}}{\phi_{\ell}}{N_{J_{\ihb},V_{\ihb,\ell}\cup V'_{\ihb,\ell}\cup \widehat{V}_{\ihb,\ell-1}}(u)}
\leq \epsilon^{1/2}(m^{j+2,\ell}_{\ihb}+ \widehat{m}^{j+2,\ell-1}_{\ihb}) + \epsilon^{1/3} n_\ihb + 
\epsilon_{\ell-1} n_\ihb \stackrel{\eqref{eq: m g relation}}{\leq} 2 \epsilon_{\ell-1} n_\ihb.
\end{eqnarray}
Thus 
{
\begin{eqnarray*}
& & \hspace{-2.2cm} \fsum{g_j}{\phi_{\ell}}{ N_{G',\widehat{V}_{\ih,\ell}}(u)}\stackrel{\eqref{eq: B2 for case 1}}{=} \sum_{x\in \cen(\widetilde{F}_{\ih}) }\sum_{y\in N_{\widetilde{F}_{\ih}}(x)}  \frac{g_j(y) d_{G',\widetilde{V}_{\ih,\ell}}( u, \phi_{\ell}(x)) }{ d|\widetilde{V}_{\ih,\ell}| } \pm  \epsilon^{1/400}_{\ell-1}|\widetilde{V}_{\ih,\ell}| 
\\
&\stackrel{\eqref{eq: V'' size},\text{\ref{Phi4}$_{\ell-1}$}}{=}& \hspace{-0.5cm} \sum_{x\in \cen(\widetilde{F}_{\ih}) }  \frac{ \sum_{y\in N_{\widetilde{F}_{\ih}}(x)} g_j(y) ( d_{G',\widehat{V}_{\ih,\ell}}(u, \phi_{\ell}(x) ) \pm 2^{-w_*}n )  }{ d \widehat{n}_{\ih,\ell} \pm 2^{-w_*}n  } \pm  \epsilon^{1/400}_{\ell-1}\widehat{n}_{\ih,\ell} \\
&\stackrel{\eqref{eq: def w* nu},\eqref{eq: def gcB}}{=}&\hspace{-0.8cm} \sum_{\phi'_{\ell}(x)\in V_{\ih,\ell}\cup V'_{\ih,\ell}\cup \widehat{V}_{\ih,\ell-1} }  \frac{ \big(g_{j+2}(x)+ g_{j+4}(x)\big) d_{G',\widehat{V}_{\ih,\ell}}( u,\phi_{\ell}(x) ) }{ d\widehat{n}_{\ih,\ell} } \pm  3\epsilon^{1/400}_{\ell-1}n'_{\ih,\ell}  \\
&\stackrel{\ref{V2}}{=}& \hspace{-0.3cm}
(d\pm \epsilon^{1/2}) \fsum{g_{j+2}}{\phi_{\ell}}{V_{\ih,\ell}\cup V'_{\ih,\ell}\cup \widehat{V}_{\ih,\ell-1} }  + (d\pm \epsilon^{1/2}) \fsum{g_{j+4}}{\phi_{\ell}}{V_{\ih,\ell}\cup V'_{\ih,\ell}\cup \widehat{V}_{\ih,\ell-1} } \\
&& \hspace{-0.3cm} \pm 2d^{-1} \fsum{g_{j+2}}{\phi_{\ell}}{ N_{J_{\ihb}, V_{\ih,\ell}\cup V'_{\ih,\ell}\cup \widehat{V}_{\ih,\ell-1} }(u)}
 \pm 2d^{-1} \fsum{g_{j+4}}{\phi_{\ell}}{ N_{J_{\ihb},V_{\ih,\ell}\cup V'_{\ih,\ell}\cup \widehat{V}_{\ih,\ell-1}}(u)} 
  \pm  3 \epsilon^{1/400}_{\ell-1}n'_{\ih,\ell} \\
&\stackrel{ \substack{ \eqref{eq: B property hold j45},\\ \eqref{eq: B property hold j45 2},\eqref{eq: B property hold j45 3}} }{=}& \hspace{-0.3cm} (d\pm \epsilon^{1/2})\fsum{g_{j+2}}{\phi_{\ell}}{V_{\ih,\ell}\cup V'_{\ih,\ell} \cup \widehat{V}_{\ih,\ell-1} } +(d\pm \epsilon^{1/2})\fsum{g_{j+4}}{\phi_{\ell}}{V_{\ih,\ell}\cup V'_{\ih,\ell} \cup \widehat{V}_{\ih,\ell-1} } \pm  \epsilon^{1/500}_{\ell-1} n_\ihb \\
&\stackrel{\text{\ref{Phi2}$_{\ell-1}$},\eqref{eq: m ih j ell def}}{=}&  \hspace{-0.3cm} (d\pm \epsilon^{1/2}) ( m^{j+2,\ell}_{\ihb}+ m^{j+4,\ell}_{\ihb}+ \widehat{m}^{j+2,\ell-1}_{\ihb} ) \pm  \epsilon^{1/500}_{\ell-1} n_\ihb
\stackrel{\ref{G2},(\ref{eq: m g relation})}{=} \hspace{-0.2cm} (d\pm \epsilon^{1/2}) \widehat{m}^{j,\ell}_{\ih} \pm \epsilon_{\ell} n_\ih.
\end{eqnarray*}
}
Here, we obtain fourth equality because, by \ref{V2} and the definition of $J_{\ihb}$, if $\phi_{\ell}(x) \notin N_{J_{\ihb},V_{\ih,\ell}\cup V'_{\ih,\ell}\cup \widehat{V}_{\ih,\ell-1}}(u)$, then 
$d_{G',\widehat{V}_{\ih,\ell}}(u,\phi_{\ell}(x)) = 
(d^2\pm 4\epsilon)\widehat{n}_{\ih,\ell}$, and 
otherwise $d_{G',\widehat{V}_{\ih,\ell}}(u,\phi_{\ell}(x)) \leq \widehat{n}_{\ih,\ell}$.
This shows that $\widetilde{\cD}^{j,\ell}_{\ih}(\phi_{\ell},\epsilon_{\ell}) \subseteq  \cC^{j+4,\ell}_{\ihb}(\phi'_{\ell},\epsilon^{1/3})$. 
Combining this with \ref{Phi3}$_{\ell}$ leads to $|\widetilde{\cD}^{j,\ell}_{\ih}(\phi_{\ell},\epsilon_{\ell})|\leq |\cC^{j+4,\ell}_{\ihb}(\phi'_{\ell},\epsilon^{1/3})| \leq 2^{-w_*} n$.
Consequently, \ref{Phi4}$_{\ell}$ holds.
We proceed to Round $(\ell+1)$.\newline


\noindent {\bf Final round.}
At this stage, 
the algorithm completed Step $(k+1)$.1. 
We have an embedding $\phi'_{k+1}$ of $T- L(\last\cup L_1)$ into $V(G)$
satisfying \ref{Phi1}$_{k+1}$--\ref{Phi4}$_{k+1}$.
Let $\phi_{k+1}:=\phi'_{k+1}$.
In the following we complete the embedding by embedding the remaining edges in $\last\cup L_1$.
We proceed in two steps; first we apply Lemma~\ref{lem: random embedding simple} to embed $L_1\cup F^\circ$ and then Lemma~\ref{lem: p matching} to embed  $\last\sm F^\circ$.

Recall that $\LA^* = \he\cap V(\last) = L(\last)$.
For every $\ih\in [r]\times [2]$, we define yet not covered vertices as
$$V^{\circ}_\ih:= V_\ih\setminus \phi_{k+1}(V(T)\setminus (L(L_1)\cup \LA^*)).$$ 
Then \ref{Phi2}$_{1}$--\ref{Phi2}$_{k+1}$, \ref{Phi2'}$_{1}$--\ref{Phi2'}$_{k+1}$ imply that
\begin{align}\label{eq: Vih0 V'ih size}
V_{\ih,k+2}\subseteq V^\circ_{\ih} \text{ and } 
n^{\circ}_{\ih}:= |V^{\circ}_\ih| = n_\ih - |X_\ih\sm \LA^*|  \stackrel{\ref{V1},\eqref{eq: nn' sizes}}{=} n_{\ih,k+2} + (3k+4) \mu n_\ih.
\end{align}
Note that if Case 1 holds, then $\last=L_2$. 
Thus \ref{Z4}$_1$ implies that for $\ih \in [r]\times [2]$, all parents of vertices in $L(L_2)\cap X_{\ih}$ are in $X_{\ihb}$ (recall that $F^\#=F'\cup \last$), 
and \ref{Phi2}, \eqref{eq: def hat L2 L3} and \eqref{eq: L3 prop} imply that 
\begin{align}\label{eq: case 1 parents of F*}
\phi_{k+1}\big( \{ a_{T}(x) : x\in L(\last)\cap X_{\ih}\}\big) \subseteq V_{\ihb,1}.
\end{align}

Now we consider again both cases simultaneously.
Recall that we have a partition $\{L_{1,\ih}\}_{ \ih\in [r]\times [2]}$ of $L(L_1)$, a subgraph $F^{\circ}\subseteq \last$ and a partition $\{Y_{\ih}\}_{ \ih\in [r]\times [2]}$ of $L(F^{\circ})$ satisfying \ref{L1}--\ref{L4}.

Let $\sE_2$ to be the event that there exists an embedding $\phi_{k+2}$ satisfying the following.
\begin{enumerate}[label=(E\arabic*)]
\item \label{E1} $\phi_{k+2}$ extends $\phi_{k+1}$,
\item \label{E2} $\phi_{k+2}$ embeds $L_1\cup F^\circ$ into $R_{k+2}$, and
\item \label{E3} for each $\ih \in [r]\times [2]$, we have $\phi_{k+2}( L_{1,\ih}\cup Y_{\ih} )\subseteq V_{\ih}^{\circ}$.
\end{enumerate}
First, we apply Lemma~\ref{lem: random embedding simple} with respect to the following graphs and parameters to estimate $\mathbb{P}[\sE_2]$.
\newline

\noindent
{\small
\begin{tabular}{c|c|c|c|c|c}
object/parameter & $(L_1\cup F^{\circ}, R(L_1\cup F^{\circ}))$ &  $R_{k+2}$ & $V^{\circ}_\ih$ & $[r]\times [2]$& $M_*p'$ \\ \hline
playing the role of & $(F,R)$  & $G$  &  $V_i$ & $[r]$ & $p$ \\ \hline \hline
object/parameter & 
$V(G)\sm \bigcup_{\ih\in [r]\times [2]} V_\ih^\circ $  & $L_{1,\ih}\cup Y_{\ih}$ & $ \Delta(L_1 \cup F^{\circ}) $ & $\phi_{k+1}\mid_{R(L_1\cup F^\circ)}$ &\\ \hline
playing the role of & $U$  & $X_i$ & $\Delta$ & $\phi'$ &
\end{tabular}
}\newline \vspace{0.2cm}

\noindent Indeed, \ref{A2,41} trivially holds.
Note that $L_1$ and $F^{\circ} \subseteq \last$ are vertex-disjoint.
Now we verify (A1)$_{\ref{lem: random embedding simple}}$ for each case.
Note that by \eqref{eq: Vih0 V'ih size}, we have
\begin{eqnarray}\notag
|V^\circ_{\ih}|  = n_{\ih} - |X_{\ih} \sm \LA^*| 
&\stackrel{\ref{L2}}{=} &
|X_{\ih}\setminus L(F^{\circ})| + |L_{1,\ih}| + |Y_{\ih}| - |X_{\ih}\setminus \LA^*| \\
&=& |X_{\ih}| - |X_{\ih}\cap L(F^{\circ})| + |L_{1,\ih}\cup Y_{\ih}| - |X_{\ih}|+|X_{\ih}\cap \LA^*|\notag \\
&=& |L_{1,\ih}\cup Y_{\ih}| + |X_{\ih}\cap \LA^*| - |X_{\ih}\cap L(F^{\circ})|.\label{eq: circ size}
\end{eqnarray}
By \eqref{eq: def Delta p}, we have 
\begin{align}\label{eq: np' p'-1}
(np')^{k}\geq M_*^{6k} n^{k/(k+1)} \geq M_* p'^{-1}.
\end{align}
In Case 1, we have $\last=L_2$. Hence $\LA_*=\LA_1$. 
Moreover, \ref{L3} implies $F^{\circ}=\emptyset$ and so $L_1\cup F^{\circ} = L_1$. 
Hence we have
\begin{eqnarray*} 
|V^{\circ}_\ih| 
&\stackrel{\eqref{eq: circ size}}{=}& |L_{1,\ih}\cup Y_{\ih}| + |X_{\ih}\cap \LA^*| 
\stackrel{ \ref{Z5}_1 }{\geq} |L_{1,\ih}\cup Y_{\ih}| + 2\eta^2 |L_2|\frac{n_\ih}{n} -  r^2 \Delta  \\
& \stackrel{\ref{G2},\ref{F12}}{\geq}& |L_{1,\ih}\cup Y_{\ih}|+ 2\eta^2 \cdot \frac{\Delta(L_1)}{3} \cdot(np')^k \cdot r^{-2} -  r^2 \Delta \\
&\stackrel{\eqref{eq: np' p'-1}}{\geq}&
|L_{1,\ih}\cup Y_{\ih}|+ \eta^2 M_* \Delta(L_1)  p'^{-1} \cdot r^{-2}/2 - r^2 \Delta \\
&=&  |L_{1,\ih}\cup Y_{\ih}|+ 2M_* \Delta(L_1) (M_*p')^{-1} - r^2 \Delta \\
&\stackrel{\ref{F12}}{\geq}&  |L_{1,\ih}\cup Y_{\ih}| + 12\Delta(L_1) (M_*p')^{-1} + 
M_* np'(M_*p' \log{n})^{-1} - r^2 \Delta \\
&\geq &  |L_{1,\ih}\cup Y_{\ih}| + 12\Delta(L_1) (M_*p')^{-1} +  30(M_*p')^{-1}\log{n}.
\end{eqnarray*}
Here, we obtain the 
final inequality since
$\frac{n}{\log{n}} \stackrel{\eqref{eq: def Delta p}}{\geq} \frac{M_*\Delta}{2} +
30 n^{k/(k+1)} \log{n}\stackrel{\eqref{eq: def Delta p}}{\geq} r^2\Delta + 30(M_*p')^{-1}\log{n}$.
Hence, \ref{A1,41} holds in Case 1.

In Case 2, \ref{F12} and \ref{L3} imply that $\Delta(L_1\cup F^{\circ})\leq 2np'$. Thus
\begin{eqnarray*} 
|V^\circ_{\ih}| 
&\stackrel{\eqref{eq: circ size}}{=}& |L_{1,\ih}\cup Y_{\ih}| + |X_{\ih}\cap \LA^*| - |X_{\ih}\cap L(F^{\circ})|\\
 &\stackrel{\eqref{eq: size of Xih W in case 2},\ref{L4} }{\geq}& |L_{1,\ih}\cup Y_{\ih}|  + \eta^3 n_\ih  - \min\{r^7 \Delta^k ,r^7 n /M_*\} \\
 &\geq&|L_{1,\ih}\cup Y_{\ih}| +  \eta^3 n_\ih /2\\
&\geq&  |L_{1,\ih}\cup Y_{\ih}|  + 12\Delta(L_1\cup F^{\circ}) (M_*p')^{-1} + 30(M_*p')^{-1}\log{n}.
\end{eqnarray*}
Hence, in both case, \ref{A1,41} holds.
Thus by Lemma~\ref{lem: random embedding simple}, we have 
\begin{align}\label{eq: prob sE2}
\bPr[\sE_2] \geq 1- n^{-2}.
\end{align}
Assume that $\sE_2$ holds, and we choose a function $\phi_{k+2}$ uniformly at random among all functions satisfying \ref{E1}--\ref{E3}, and for each $\ih \in [r]\times [2]$, we let
$$V^{\bullet}_{\ih} := V^{\circ}_\ih\setminus \phi_{k+2}(L(L_1)\cup L(F^{\circ})) \text{ and }F^{\bullet}:= \last\setminus L(F^{\circ}).$$
Thus it remains to embed $F^\bullet$ into $\bigcup_{\ih\in[r]\times[2]}V_\ih^\bullet$.
As $\phi_{k+2}$ is a random variable, $V^{\bullet}_{\ih}$ is also a random variable.  
To be able to finish the embedding, we want to show that $V^{\bullet}_{\ih}$ are nicely chosen with high probability.
For each $\ih \in [r]\times [2]$, we have
\begin{align}\label{eq:size n bullet}
	n^{\bullet}_{\ih}:=|V^{\bullet}_{\ih}|= |V^{\circ}_{\ih}| - |L_{1,\ih}| - |Y_{\ih}| \stackrel{\eqref{eq: circ size}}{=} 
	|X_{\ih}\cap \LA^* \setminus L(F^{\circ})|
	=|L(F^\circ)\cap X_\ih|.
\end{align}
Next we estimate $n^{\bullet}_{\ih}$ and we consider two cases.
In Case 1, as we have $\Delta(L_1) \stackrel{\ref{F12}}{\geq} \frac{np'}{\log{n}}$, $\last=L_2$ and $F^{\circ}=\emptyset$, we conclude that for each $\ih \in [r]\times [2]$, 
\begin{eqnarray}\label{eq: m31i3-h size}
 n^{\bullet}_{\ih} 
\stackrel{\rm \ref{Z5}_{1}}{\geq}  2\eta^2 |L_2| \frac{n_\ih}{n} -  r^2 \Delta  
\stackrel{\ref{F12}}{\geq} \frac{ 2\eta^2 \Delta(L_1) (np')^{k}}{3r^2} - r^2 \Delta
\stackrel{\eqref{eq: def Delta p}}{\geq} \frac{2\eta^2 n}{3 r^2 \log{n}} - \frac{r^2 n}{M\log{n}}
 \stackrel{\eqref{eq: hierarchy}}{\geq} \frac{n}{M\log{n}}.
\end{eqnarray}
Moreover, since $F^\bullet= \last$, \eqref{eq: case 1 parents of F*} and \eqref{eq: m ih j ell def} implies the following in the Case 1:
\begin{align}\label{eq: n bullet size case 1}
m^{2,1}_{\ihb} = n^{\bullet}_{\ih}.
\end{align}
In Case 2, we have
\begin{eqnarray}\label{eq: n diamond size}
n^{\bullet}_{\ih} \stackrel{\eqref{eq: size of Xih W in case 2},\ref{L4} }{\geq}  \frac{3}{2}\eta^3 n_{\ih} - \min\{ r^7 \Delta^{k} , r^7n/M_*\}  \geq \eta^3 n_{\ih}.
\end{eqnarray}
Let $\sE'_2$ be the event that the following holds:
\begin{enumerate}[label={\rm(V3)}]
\item\label{V3} For each $B\in \cB_{\ih} \cup \cB'_{\ih} \cup \cB''_{\ih}$, we have $|V^{\bullet}_{\ih} \cap B| = \frac{|V^{\circ}_\ih \cap B| |V^{\bullet}_{\ih}| }{|V^{\circ}_\ih| } \pm n^{4/5}.$
\end{enumerate}
Lemma~\ref{Chernoff Bounds} together with \eqref{eq: m31i3-h size} and \eqref{eq: n diamond size} implies that the number of sets of size $n^{\bullet}_{\ih}$ that satisfy \ref{V3} is at least $(1- n^{-3})\binom{n^{\circ}_{\ih}}{n^{\bullet}_{\ih}}$.
As the distribution of the random graph $R_{k+2}$ is invariant under any vertex permutation (similar to Remark~\ref{rmk: symmetry}), for any two subsets $A, B$ of $V^{\circ}_{\ih}$ of size $n^{\bullet}_{\ih}$, we have 
\begin{align}\label{eq: uniform choice}
\mathbb{P}[V^{\bullet}_{\ih} = A \mid \sE_2 ] =\mathbb{P}[V^{\bullet}_{\ih} = B \mid \sE_2].
\end{align}
As the number of all possible outcomes of the random variable $(V^{\bullet}_\ih)_{\ih \in [r]\times [2]}$ is $\prod_{\ih\in [r]\times [2]}\binom{n^{\circ}_{\ih}}{n^{\bullet}_{\ih}}$, and the number of outcomes  
$(V^{\bullet}_\ih)_{\ih \in [r]\times [2]}$ satisfying \ref{V3} is at least 
$\prod_{\ih\in [r]\times [2]} (1- n^{-3})\binom{n^{\circ}_{\ih}}{n^{\bullet}_{\ih}}$, \eqref{eq: uniform choice} implies that
\begin{align*}
\mathbb{P}[ \sE'_2\mid \sE_2 ] \geq \prod_{\ih\in [r]\times [2]} \frac{ (1-n^{-3}) \binom{n^{\circ}_{\ih}}{n^{\bullet}_{\ih}}}{\binom{n^{\circ}_{\ih}}{n^{\bullet}_{\ih}}} \geq (1-n^{-2}).\end{align*}
This together with \eqref{eq: prob sE2} implies that 
\begin{align}\label{eq: failure probability 2}
\mathbb{P}[ \sE'_2 \wedge \sE_2 ] \geq 1 - 2n^{-2}.
\end{align}
From now on we assume that $\phi_{k+2}$ is an embedding satisfying $\sE'_2 \wedge \sE_2$, because otherwise we end the algorithm with failure.
Recall that $F^\bullet=\last\sm L(F^\circ)$, so $F^\bullet$ is a star forest.
Then by \ref{L4}, for any $\ih\in [r]\times[2]$, we have
\begin{align}\label{eq: F* degree F**}
 \sum_{x\in R(F^{\bullet})\cap X_{\ih}} d_{F^{\bullet}}(x) 
=  \sum_{x\in R(F^{\bullet})\cap X_{\ih}} d_{\last}(x) \pm\min\{r^7\Delta^k, r^7 n/M_*\}.
\end{align}
We wish to apply Lemma~\ref{lem: p matching} for each $\ih\in [r]\times [2]$ with the following objects and parameters to obtain the final embedding $\phi$ which extends $\phi_{k+2}$ and embeds $T$ into $G\cup R$.  \newline

\noindent
{ \small
\begin{tabular}{c|c|c|c|c|c|c|c|c}
object/parameter & $G'[V^{\bullet}_\ih, V_\ihb ]$ & $V_\ihb $ & $V^{\bullet}_\ih$ &  $F^{\bullet} [ L(F^{\bullet}) \cap X_\ih,  X_{\ihb} ]$   & $\phi_{k+2}\mid_{ R(F^{\bullet}) \cap X_{\ihb}}$ & $\mu^{1/2}$ & $n^{\bullet}_\ih$ & $2^{-w_*/2}$
\\ \hline
playing the role of & $G$ & $U$ & $V$   &  $F$   & $\psi$ & $\epsilon$& $n$  & $\nu$ 
\end{tabular}
}\newline \vspace{0.2cm}

\noindent 
By~\eqref{eq:size n bullet}, $F^{\bullet} [ L(F^{\bullet}) \cap X_\ih,  X_{\ihb} ]$ has exactly $n^\bullet_\ih$ leaves.
Note that \eqref{eq: m31i3-h size} and \eqref{eq: n diamond size} imply that in both Case 1 and Case 2, we have $1/n^{\bullet}_\ih \ll \mu^{1/2}$, and we know that $w_*^{-3} < \mu^{1/2}$ by \eqref{eq: def w* nu}.
In order to apply the lemma, we need to verify that \ref{L45A1}--\ref{L45A3} hold with the parameters specified above. 
To show \ref{L45A1}, consider a vertex $u \in V_\ihb$.
 As $N_{G'}(u) \in \cB_{\ih} \subseteq \cB^1_{\ih}$, we have
\begin{eqnarray}\label{eq: A degree good}
|N_{G'}(u)\cap V^{\bullet}_\ih| 
  &\stackrel{\ref{V3}}{=}&  \frac{d_{G',V^\circ_{\ih}}(u) |V^{\bullet}_{\ih}| }{|V^\circ_{\ih}| } \pm n^{4/5}\nonumber \\
&\stackrel{\rm \ref{G1}, \ref{V2}, \eqref{eq: Vih0 V'ih size}}{=} &
 \frac{   n^{\bullet}_{\ih} }{ n^\circ_{\ih} }((d\pm \epsilon) n_{\ih,k+2} \pm (3k+4)\mu n_{\ih} ) \pm n^{4/5} \nonumber \\
&\stackrel{\eqref{eq: nih0 size},\eqref{eq: Vih0 V'ih size},\eqref{eq: m31i3-h size},\eqref{eq: n diamond size} }{=} & (d\pm \mu^{1/2}) n^{\bullet}_\ih.
\end{eqnarray}
Thus~\ref{L45A1} holds.

To show~\ref{L45A2}, first assume that Case 1 applies.
In this case, we have $F^{\bullet}= \last=L_2$.
For each $v\in V^{\bullet}_{\ih}$, we have
\begin{eqnarray*}
\fsum{d_{F^{\bullet}}}{ \phi_{k+2}}{N_{G', V_{\ihb}}(v)}&\stackrel{\eqref{eq: case 1 parents of F*}}{=}&
\fsum{d_{F^{\bullet}}}{ \phi_{k+2}}{N_{G',V_{\ihb,1}}(v)}
 \stackrel{\eqref{eq: def gcB}}{=} \fsum{g_2}{\phi_{k+2}}{ N_{G',V_{\ihb,1}}(u)} \\
&\stackrel{\text{\ref{Phi3}$_1$}, \text{\eqref{eq: def cC cD}}}{=} &(d\pm\epsilon^{1/2}) m^{2,1}_{\ihb} \pm \nu n_\ihb \stackrel{ \eqref{eq: def w* nu}, \eqref{eq: m31i3-h size},\eqref{eq: n bullet size case 1} }{=}  (d\pm \mu^{1/2})n^{\bullet}_{\ih}.
\end{eqnarray*}
We obtain the penultimate equality since $\widetilde{\cC}_{\ihb}^{2,1}(\phi,\nu)=\es$.
This shows that~\ref{L45A2} holds for Case 1. 

Now assume Case 2 holds.
For each $\ih\in [r]\times[2]$, we let (see~\eqref{eq: def cC cD})
$$D_{\ih}:= \bigcup_{\ell\in [k]} \widetilde{\cD}^{2,\ell}_{\ihb}(\phi_{\ell},\epsilon_{\ell})\subseteq V_\ih.$$
Thus, for any $v\in V^{\bullet}_\ih\setminus D_{\ih}$, we have
\begin{eqnarray}\label{eq: C2 verifies 1}
 \fsum{d_{F^\bullet}}{\phi_{k+2}}{ N_{G',V_{\ihb}}(v)} \hspace{-0.35cm}
&\stackrel{\eqref{eq: def gcB},\eqref{eq: F* degree F**}}{=}& \hspace{-0.35cm}
\sum_{\ell\in [k+1]} \fsum{g_2}{\phi_{k+2}}{N_{G',V_{\ihb,\ell}\cup V'_{\ihb,\ell}}(v)} + \sum_{\ell\in [k]} \fsum{g_2}{\phi_{k+2}}{N_{G',\widehat{V}_{\ihb,\ell}}(v)} \pm \frac{r^7n}{M_*} \nonumber \\
 &\stackrel{\eqref{eq: def cC cD}, \text{\ref{Phi3}$_{\ell}$}  }{=}& \hspace{-0.35cm} \sum_{\ell\in [k+1]} ( (d\pm \epsilon^{1/2})m^{2,\ell}_{\ihb}  \pm \nu n_\ihb) +
 \sum_{\ell\in [k]} ( (d\pm \epsilon^{1/2}) \widehat{m}^{2,\ell}_{\ihb}  \pm \epsilon_{\ell} n_\ihb) \pm   \frac{r^7n}{M_*} \nonumber \\
 &=& \hspace{-0.35cm} (d\pm \epsilon^{1/2})\left (\sum_{\ell\in [k+1]} m^{2,\ell}_{\ihb}+\sum_{\ell\in [k]} \widehat{m}^{2,\ell}_{\ihb}\right) \pm \mu n_\ihb \nonumber \\
 &\stackrel{\eqref{eq: g sum to W*}}{=}& \hspace{-0.35cm} (d\pm \epsilon^{1/2})|\LA^{*}\cap X_\ih| \pm \mu n_\ihb 
 \stackrel{\eqref{eq: n diamond size},\ref{L4}}{=} (d\pm \mu^{1/2}) n^{\bullet}_\ih.
\end{eqnarray}
Here we obtain the second inequality since $\widetilde{\cC}^{2,\ell}_{\ihb}(\phi_{k+2},\nu)=\es$ for all $\ell \in [k+2]$ and because $v\notin D_\ih$.
On the other hand, again $\widetilde{\cC}^{2,\ell}_{\ihb}(\phi_{k+2},\nu)=\es$ for all $\ell \in [k+2]$, so if $u\in V^{\bullet}_\ih\cap D_{\ih}$, then still we have
\begin{eqnarray}\label{eq: special vertices in D}
\hspace{-0.2cm}
\fsum{d_{F^\bullet}}{\phi_{k+2}}{ N_{G',V_{\ihb}}(v)}\hspace{-0.33cm}
&\stackrel{\eqref{eq: def gcB},\eqref{eq: F* degree F**}}{\geq}&\hspace{-0.33cm}
 \sum_{\ell\in [k+1]} \fsum{g_2}{\phi_{k+2}}{ N_{G', V_{\ihb,\ell}\cup V'_{\ihb,\ell}}(u)} - \min\{r^7\Delta^k, r^7 n/M_*\} \nonumber \\
 &\stackrel{\eqref{eq: def cC cD}, \text{\ref{Phi3}$_{\ell}$} }{\geq}&\hspace{-0.33cm}   \sum_{\ell\in [k+1]} \left( (d- \epsilon^{1/2}) m^{2,\ell}_{\ihb} - \nu n_\ihb \right) - \min\{r^7\Delta^k, r^7 n/M_*\}  \nonumber \\
 &\geq&\hspace{-0.33cm} (d - \epsilon^{1/2}) \left(\sum_{\ell\in [k+1]} m^{2,\ell}_{\ihb}\right) - (k+1)\nu n_\ihb - \min\{r^7\Delta^k, r^7 n/M_*\} \nonumber \\
 &\stackrel{\eqref{eq: def gcB},\eqref{eq: m ih j ell def}}{=}& \hspace{-0.33cm} (d - \epsilon^{1/2} )\bigg|X_\ih\cap \bigcup_{\ell=1 }^{k+1} \LA_{i} \bigg|  - (k+1)\nu n_\ihb - \min\{r^7\Delta^k, r^7 n/M_*\}.
\end{eqnarray}

By \eqref{eq: def w* nu}, $w_*\in \{M_*\log{n} , \frac{M_*^{1/2} n^{1/2}}{\Delta}\}$. 
If $w_* > 2\log{n}$, then $D_{\ih}=\emptyset$ by \ref{Phi4}$_{1}$--\ref{Phi4}$_{k}$. 
Thus \eqref{eq: C2 verifies 1} implies that~\ref{L45A2} holds.
If $w_* = \frac{M_*^{1/2} n^{1/2}}{\Delta} \leq 2\log{n}$, then by \eqref{eq: def Delta p}, we have that $k=2$ and $p=  \frac{M_*^{3/2}}{ w_*^{3}n^{1/2}}$. 
Hence we have $\Delta^{k} = \frac{M_* n}{w_*^2} =  \frac{n^2 p}{\Delta}$.
Also since $k=2$, \eqref{eq: def w* nu} implies that $\nu = n^{-1/10}$.
As $p'= M_*^6 p$, we have
\begin{eqnarray}
 \hspace{-0.5cm} \fsum{d_{F^\bullet}}{\phi_{k+2}}{ N_{G',V_{\ihb}}(v)} \hspace{-0.3cm}
&\stackrel{\eqref{eq: def w* nu},\eqref{eq: special vertices in D}}{\geq }& \hspace{-0.3cm} (d - \epsilon^{1/2} )\bigg|X_\ih\cap \bigcup_{\ell=1 }^{k+1} \LA_{i} \bigg|  -  n^{9/10} - \min\Big\{\frac{r^7 n^2 p}{\Delta}, \frac{r^7 n}{M_*}\Big\} \nonumber\\
&\stackrel{\eqref{eq: size of Q*}}{\geq }& \hspace{-0.3cm} \eta^{4} n_{\ih} \cdot \min\Big\{\frac{np'}{\Delta},1\Big\} - n^{9/10} - \min\Big\{\frac{r^7 n^2p'}{ M_*^6 \Delta}, \frac{r^7 n}{M_*}\Big\}\nonumber \\
&\geq & \hspace{-0.3cm} \min\Big\{  \frac{  n^2 p'}{ M_* \Delta}, \frac{\eta^{4}n_{\ih}}{4}\Big\}
\stackrel{\eqref{eq: def Delta p}}{ \geq } \min\Big\{\frac{ M_*^6 n}{w_*^2}, \frac{\eta^{4}n_{\ih}}{4}\Big\} \geq 2^{-w_*/2}  n^{\bullet}_\ih
\end{eqnarray}
Here we obtain the third inequality holds by considering the two case of  $np'< \Delta$ and $np'\geq \Delta$. 
We obtain the final inequality as $w_*\geq M^{1/2}_*$.
Since \ref{Phi4}$_1$--\ref{Phi4}$_k$ imply that $|D_{\ih}|\leq (k+2) 2^{-w_*} n \stackrel{\eqref{eq: n diamond size}}{\leq} 2^{-w_*/2} n_\ih^{\bullet}$, we know~\ref{L45A2} holds in Case 2.

Now we only need to verify~\ref{L45A3}.
For each $uu' \in \binom{V^\bullet_{\ihb}}{2}\setminus  J_{\ihb}$, we have $N_{G',V_{\ih}}(u,u')  \in \cB''_{\ih}$. 
Thus the definition of $J_{\ihb}$ implies that
\begin{eqnarray*}
d_{G',V^{\bullet}_{\ih}}(u,u') 
&\stackrel{ \ref{V3} }{=}& \frac{d_{G',V^\circ_{\ih}}(u,u') }{n^\circ_{\ih}}{n^{\bullet}_{\ih}}  \pm n^{4/5}
\stackrel{\eqref{eq: Vih0 V'ih size} }{=}\frac{d_{G',V_{\ih,k+2}}(u,u')  \pm (3k+4)\mu n_\ih  }{n_{\ih,k+2} \pm (3k+4)\mu n_\ih }n_{\ih}^{\bullet} \pm n^{4/5} \\
&\stackrel{\eqref{eq: nih0 size},
\ref{V2} }{=}& 
 (d^2\pm \mu^{1/2}) n_{\ih}^{\bullet}.
 \end{eqnarray*}
Here, we also use for last equality \eqref{eq: m31i3-h size} in Case 1 and \eqref{eq: n diamond size} in Case 2.
Thus this implies that 
$$J^{\bullet}_{\ih}:= J_{G'[V^{\bullet}_\ih, V_{\ihb} ] }(V_\ihb, d,  \mu^{1/2} ) \subseteq J_{\ihb}.$$
Note that $F^{\bullet}\subseteq \last$. Hence have $d_{F^\bullet}(x) \leq d_{\last}(x) = g_2(x)$ for all $x\in V(T)$. 
In Case 1, as $F^\bullet = \last = L_2$, we have
\begin{eqnarray*}
 \sum_{x,x'\colon \phi_{k+2}(x)\phi_{k+2}(x') \in E(J^{\bullet}_{\ihb})} d_{\last}( x) d_{\last}(x') \hspace{-0.3cm}
&\stackrel{\eqref{eq: def gcB}}{\leq}&  \hspace{-0.3cm} \sum_{ x\in X_{\ihb} \cap (V(F_1)\setminus R(F_1)) } g_2(x)  \fsum{g_2}{ \phi_{k+2} }{ N_{J_{\ihb}}(\phi(x))}\\
&\stackrel{\text{\ref{Phi3}$_1$},\eqref{eq: case 1 parents of F*}}{\leq} &\hspace{-0.3cm} \sum_{ x\in X_{\ihb}} g_2(x) (\epsilon^{1/2} m^{2,1}_{\ihb}+ \nu n_\ih ) \\
&\stackrel{\eqref{eq: m ih j ell def},\eqref{eq: def w* nu}}{\leq}& \epsilon^{1/2}  (m^{2,1}_{\ihb})^2 + n^{19/10}\\
&\stackrel{\eqref{eq: m31i3-h size},\eqref{eq: n bullet size case 1}}{\leq}&   \mu^{1/2} ( n^{\bullet}_{\ih})^2 .
\end{eqnarray*}
In Case 2, we obtain
\begin{eqnarray*}
& & \hspace{-2cm} \sum_{x,x'\colon \phi_{k+2}(x) \phi_{k+2}(x') \in E(J^{\bullet}_{\ihb})  } d_{F^\bullet}(x) d_{F^\bullet}(x') 
 \leq  \sum_{x\colon\phi_{k+2}(x) \in V_\ihb} g_2(x) \fsum{g_2}{ \phi_{k+2} }{ N_{J_{\ihb}}(\phi(x))} \\
&\leq & \sum_{x\colon\phi_{k+2}(x) \in V_\ihb} g_2(x)
\left( \sum_{\ell\in [k+1]}  \fsum{g_2}{ \phi_{k+2} }{ N_{J_{\ihb}, V_{\ihb,\ell}\cup V'_{\ihb,\ell}}(\phi(x))} +  \sum_{\ell\in [k]}  \fsum{g_2}{ \phi_{k+2} }{ N_{J_{\ihb}, \widehat{V}_{\ihb,\ell}}(\phi(x))}  \right)\\
&\stackrel{\text{\ref{Phi3}$_\ell$},\text{\ref{Phi4}$_\ell$}}{\leq }&
\sum_{x\colon\phi_{k+2}(x) \in V_\ihb} g_2(x) \left(\sum_{\ell=1}^{k+1} \left(\epsilon^{1/2} m^{2,\ell}_{\ihb} + \nu n_\ihb \right) +  \sum_{\ell=1}^{k} \left( \epsilon^{1/2}\widehat{m}^{2,\ell}_{\ihb} +\epsilon_{\ell} n_\ihb\right) \right) \\
&\stackrel{\eqref{eq: m g relation},\ref{G2}}{\leq} &
\sum_{x\colon\phi_{k+2}(x) \in V_\ihb}  \mu^{2/3} g_2(x) n_{\ih} 
\stackrel{\eqref{eq: n diamond size}}{\leq} \mu^{1/2} (n^{\bullet}_{\ih})^2.
\end{eqnarray*}
Thus~\ref{L45A3} holds.
Hence, Lemma~\ref{lem: p matching} yields the desired embedding $\phi$, and this finishes the embedding algorithm.

\bigskip

If the above embedding algorithm succeeds, 
then we obtain the desired embedding $\phi$ of $T$ into $G\cup \bigcup_{\ell=1}^{k+2} R_{\ell}$.
Moreover, using a union bound together with \eqref{eq: failure probability 1} and \eqref{eq: failure probability 2} implies that the embedding algorithm succeeds with probability at least $1-  2(k+2)n^{-3/2}.$

Observe that we did not use $R_{k+3}$ so far. 
We will use it in the following section to deal with the case where \eqref{eq: enough heavy leafs} does not hold.

\section{Trees with few heavy leafs} \label{sec: few heavy leaves}

Recall that $\he$ and $\li$ denotes the set of heavy and light leaves in $T$, respectively.
Observe that we have assumed in the previous three sections that $|\he|\geq 4\eta n$ (see~\eqref{eq: enough heavy leafs}).
Now we may assume that $|\he|< 4\eta n$.
We split again into two cases; first we assume that $|\li|\geq 4\eta n$
and otherwise, $T$ has at most $8\eta n$ leaves.
In the latter case,
we use Lemma~\ref{lem: number of leaves}
to conclude that $T$ contains
a collection $\{ P_i : i\in [2\eta n]\}$ of vertex-disjoint $(k+3)$-vertex bare paths.
Indeed, observe that
$$
\frac{n}{k+3}-16\eta n \geq 2\eta n.
$$
\newline

\noindent {\bf CASE A.} $|\li| \geq 4\eta n$.
\medskip

We proceed as follows.
First we remove exactly $4\eta n$ light leaves from $T$ and obtain a new tree $T'$.
Afterwards, we add $4\eta n$ leaves to $T'$ in such a way that the new tree $T^*$
 has at least $4\eta n$ heavy leaves.
As $T^*$ has at least $4\eta n$ heavy leaves,
$T^*$ has an embedding into $G \cup \bigcup_{\ell=1}^{k+2}R_\ell$ with high probability, as we showed in the previous sections, and so does $T'$.
Since we only removed light leaves of $T$ to obtain $T'$,
it is easy to extend the embedding of $T'$ to an embedding of $T$ by using the edges in $R_{k+3}$.

Now we turn to the details.
Let $L \subseteq \li$ be a set of exactly $4\eta n$ light leaves and we set
 $$T':= T- L.$$
Let $x_1,\dots, x_{4\eta n}$ be $4\eta n$ new vertices.
Since every $m$-vertex tree contains at least $m/2$ vertices of degree $1$ or $2$ for every $m\geq 2$,
the tree $T'$ contains at least $2\eta n$ vertices $y_1,\ldots,y_{\eta n}$ of degree $1$ or $2$ (in $T'$). 
We partition $\{ x_1,\dots, x_{4\eta n}\}$ into  sets $X_1,\dots, X_{{ 2 \eta \log{n} }/{ p' } }$ of size ${2np'}/{\log{n}} \pm 1$.
Let $T^*$ be a tree with 
$$V(T^*):= V(T')\cup \{ x_1,\dots, x_{4\eta n}\}\text{ and }
E(T^*)= E(T')\cup \{ y_i x :x \in X_i , i\in  [{ 2\eta \log{n} }/p'] \}.$$
Hence, for each $i\in [{ 2 \eta \log{n} }/{ p' }]$, we have
\begin{align*}\label{eq: artificial heavy}
\frac{2np'}{\log{n}}\leq d_{T^*}(y_i)\leq \frac{2np'}{\log{n}} + 3 < \Delta.
\end{align*}
Thus $x_j$ is a heavy leaf in $T^*$ for all $j\in [4\eta n]$ as well as
$\Delta(T^*)\leq \Delta$ and $T^*$ has at least $ 4\eta n$ heavy leaves.

Since $T^*$ satisfies \eqref{eq: enough heavy leafs}, there exists an embedding $\phi'$ of $T^*$ into $G\cup \bigcup_{\ell=1}^{k+2} R_{\ell}$ with probability at least $1-2(k+2)n^{-3/2}$.
Now, let $\{ a_1,\dots, a_{(1-4\eta) n}\}:= \phi'(V(T'))$.
For each $i\in[(1-4\eta) n]$, 
let $d_i:= d_{T-E(T')}(\phi'^{-1}(a_i))$.
We apply Lemma~\ref{lem: light leaves embedding} with the following objects and parameters. 
\newline

\noindent {\small
\begin{tabular}{c|c|c|c|c|c|c|c|c}
object/parameter & $\phi'(V(T'))$ & $\{ a_1,\dots, a_{(1-4\eta) n}\}$ & $R_{k+3}$ & $|\phi'(V(T'))|$ & $M_*p'$ & $d_i$ & $4\eta n$ & $2np'/\log n$
\\ \hline
playing the role of & $A$ & $B$ & $G$ &  $k$ & $p$ & $d_i$ & $n$ & $\Delta$
\end{tabular}}
\newline \vspace{0.2cm}

\noindent Note that $\sum_{i\in [(1-4\eta) n]} d_i = |L| =4\eta n$ and
$$M_*p'
\geq \frac{p'}{\eta}\cdot \frac{\log 4\eta n}{\log n}
\geq \frac{4np'}{\log n} \cdot \frac{\log 4\eta n}{4\eta n}.
$$
 Lemma~\ref{lem: light leaves embedding} ensures that, with probability $1-o(1)$, $\phi'$ can be extended to an embedding $\phi$ of $T$ into $G\cup \bigcup_{\ell=1}^{k+3} R_{\ell}$. \newline

\noindent {\bf CASE B.} The tree $T$ contains a collection $\{ P_i : i\in [2\eta n]\}$ of vertex-disjoint $(k+3)$-vertex bare paths. 
\medskip

We proceed similarly as in Case A.
This time, we remove the interior vertices of the paths in $\{ P_i : i\in [2\eta n]\}$ and obtain a forest $F$,
say $s_i$ and $t_i$ are the endvertices of $P_i$, which are still contained in $F$.
Again, we consider a set of new vertices $\{ x_1,\dots, x_{2(k+1)\eta n}\}$ and partition
this set into 
${ 2 \eta \log{n} }/{ p' }$ sets $X'_1,\dots, X'_{{  2\eta \log{n} }/{ p' } }$ of size ${(k+1)np'}/{\log{n}} \pm 1$.
Again, we construct a tree $T^*$ which contains $F$ with many heavy leaves as follows
\begin{align*}
	V(T^*)&:= V(F)\cup \{x_1,\dots, x_{2(k+1)\eta n}\}\\
	E(T^*)&:= E(F)\cup \{s_i t_i : i\in [2\eta n]\} \cup  \{ s_i q :q \in X'_i , i\in  [{ 2\eta \log{n} }/{ p' }] \}.
\end{align*}
Hence for each $i\in [2\eta n]$, we have
\begin{align*}
\frac{(k+1)np'}{\log{n}}-2\leq d_{T^*}(s_i),d_{T^*}(t_i)\leq \frac{(k+1)np'}{\log{n}} + 3 < \Delta.
\end{align*}
Thus $x_j$ is a heavy leaf for all $j\in [2(k+1)\eta n]$ 
as well as $\Delta(T^*)\leq \Delta$ and $T^*$ has at least $ 4\eta n$ heavy leaves.

Since $T^*$ satisfies \eqref{eq: enough heavy leafs}, 
there exists an embedding $\phi'$ of $T^*$ into $G\cup \bigcup_{\ell=1}^{k+2} R_{\ell}$ with probability at least $1- 2(k+2)n^{-3/2}$.
Now, let $V':= V(G)\setminus \phi'(V(F))$.
  Now we apply Lemma~\ref{lem: bare paths embedding} with the following objects and parameters. \newline

\noindent{\small
\begin{tabular}{c|c|c|c|c|c|c|c}
object/parameter & $\phi'(s_i)$ & $\phi'(t_i)$ & $R_{k+3}[V'\cup \phi'(\{s_1,t_1,\ldots, s_{\eta n},t_{\eta n}\})]$ & $ k+3$ & $M_*p'$ & $n$ & $M_*$
\\ \hline
playing the role of & $s_i$ & $t_i$ & $G$ &  $k$ & $p$  & $2\eta n$ & $M$
\end{tabular}}
\newline \vspace{0.2cm}

\noindent As $M_*p' \geq M_*(\frac{ \log 2\eta n }{n^{k+1}} )^{1/(k+2)}$ by \eqref{eq: def Delta p}, 
Lemma~\ref{lem: bare paths embedding} ensures that, with probability $1-o(1)$, $\phi'$ can be extended to an embedding $\phi$ of $T$ into $G\cup \bigcup_{\ell=1}^{k+3} R_{\ell}$. This finishes the proof of Theorem~\ref{thm: main}.

\bibliographystyle{amsplain}
\bibliography{littrees}

\end{document}